\documentclass[a4paper,reqno]{amsart}

\DeclareMathAlphabet{\lcal}{U}{dutchcal}{m}{n}
\usepackage{amsmath, amsfonts, amsthm, amssymb}
\usepackage[maxbibnames=9]{biblatex}
\usepackage{mathrsfs}
\usepackage{tikz}
\usetikzlibrary{arrows.meta, positioning, calc}
\usepackage[title]{appendix}

\usepackage{xcolor}
\usepackage{graphicx}

\usepackage{hyperref}
\usepackage{mathtools}

\usepackage{booktabs}

\usepackage{array}
\usepackage{longtable}

\usepackage[shortlabels]{enumitem}

\usepackage{setspace}
\usepackage{needspace}

\defbibheading{bibliography}[\refname]{\section*{#1}}
\newcommand{\df}[1]{{\textit{#1}}{\index{#1}}}

 \numberwithin{equation}{section}

\allowdisplaybreaks

\makeindex

\hypersetup{
    unicode=false,          
    pdftoolbar=true,        
    pdfmenubar=true,        
    pdffitwindow=false,     
    pdfstartview={FitH},    
    pdftitle={Beurling criteria for reproducing kernels},    
    pdfauthor={Shuaibing Luo, Scott McCullough, Georgios Tsikalas},     
    pdfsubject={Reproducing kernels, invariant subspaces, and Agler decompositions},   
    pdfkeywords={reproducing kernels, Beurling--Lax--Halmos pairs, complete Pick property, Agler wedge}, 
    pdfnewwindow=true,      
    colorlinks=false,       
    linkcolor=green,          
    citecolor=green,        
    filecolor=green,      
    urlcolor=green,           
    urlbordercolor={1 1 1}  
}

\newtheorem{theorem}{Theorem}[section]
\newtheorem*{theorem*}{Theorem}
\newtheorem*{conjecture*}{Conjecture}
\newtheorem{lemma}[theorem]{Lemma}
\newtheorem{assumption}[theorem]{Assumption}
\newtheorem{problem}[theorem]{Problem}

\newcommand{\DD}{\mathbb{D}}

\newcommand{\CC}{\mathbb{C}}

\newcommand{\NN}{\mathbb{N}}

\newcommand{\Mult}{\operatorname{Mult}}

\newcommand{\cB}{\mathcal{B}}
\newcommand{\cM}{\mathcal{M}}

\newcommand{\cH}{\mathcal{H}}

\newcommand{\HSF}{\mathcal{F}}
\newcommand{\HSE}{\mathcal{E}}
\newcommand{\HSG}{\mathcal{G}}
\newcommand{\specT}{\sigma_T}  
\newcommand{\genk}{\kappa} 
\newcommand{\genl}{\lambda} 

\newcommand{\ol}[1]{\overline{#1}}
\newcommand{\Oml}{\Omega} 

\newcommand{\Her}{\operatorname{Her}}

\newcommand{\hol}{\operatorname{Hol}}

\newcommand{\opphi}[4]{{#1}\in \Mult({#3},{#4})}

\newcommand{\LK}{\mathcal{K}}
\newcommand{\KL}{\mathcal{L}}
\newcommand{\tOmega}{\Omega^*}

\newcommand{\psif}{\zeta}

\newtheorem{proposition}[theorem]{Proposition}

\newtheorem{corollary}[theorem]{Corollary}
\theoremstyle{definition}
\newtheorem{definition}[theorem]{Definition}
\newtheorem{question}[theorem]{Question}
\newtheorem{example}[theorem]{Example}
\newtheorem{remark}[theorem]{Remark}

  \usepackage{etoolbox}

\begin{document}

\newcommand{\N}{\mathbb{N}}
\newcommand{\T}{\mathbb{T}}

\theoremstyle{definition}

\theoremstyle{remark}

\title[\small BEURLING CRITERIA FOR REPRODUCING KERNELS]{\large   BEURLING CRITERIA FOR REPRODUCING KERNELS}

\author[\small LUO]{SHUAIBING LUO}
\address{SCHOOL OF MATHEMATICS, HUNAN UNIVERSITY, CHANGSHA, HUNAN, 410082,
PR CHINA}
\email{sluo@hnu.edu.cn}

\author[\small MCCULLOUGH]{SCOTT MCCULLOUGH}
\address{DEPARTMENT OF MATHEMATICS, UNIVERSITY OF FLORIDA, GAINESVILLE, FL}
\email{sam@ufl.edu} 
\author[\small TSIKALAS]{GEORGIOS TSIKALAS}
\address{DEPARTMENT OF MATHEMATICS, UC SANTA BARBARA, SANTA BARBARA, CA}
\email{gtsikalas@ucsb.edu}

\small
\begin{abstract}
 A classical theorem of Beurling, Lax and Halmos characterizes the invariant subspaces of the unilateral shift as ranges of isometric multiplication operators on vector-valued Hardy spaces. We introduce Beurling--Lax--Halmos (BLH) pairs of reproducing kernels, which admit an analogue of this theorem. Under varying hypotheses, we characterize the BLH class in several ways: through dilation theory, a complete pseudo-Leech interpolation property, and an Agler decomposition inspired by sums of squares. These characterizations unify and extend a variety of related results in the literature. Examples illustrate the theory and its hypotheses and compare the BLH property with the complete Pick property for pairs. Although complete Pick pairs are expected to be BLH pairs under mild assumptions, the converse fails even for strongly regular pairs of normalized diagonal holomorphic kernels.
\end{abstract}
\maketitle

 \par \normalsize

\section{Introduction}

\onehalfspacing

\subsection{Beurling-type theorems and the main result}\label{ss:intro-question}

Let $H^2$ be the Hardy space of holomorphic functions on the unit disk $\DD$ with
square-summable Taylor coefficients, and let $S$ be the shift $f\mapsto zf$ on $H^2$.
Beurling's theorem \cite{Beuroriginal} states that every nonzero subspace of $H^2$ invariant
under $S$ has the form $\theta H^2$ for an inner function $\theta$; here and below, subspaces of
a Hilbert space are closed. In the form due to Lax \cite{LaxOriginal} and Halmos
\cite{Halmosoriginal}, it states that, for a Hilbert space $\HSE$, every subspace of
$H^2\otimes\HSE$ invariant under $S\otimes I_\HSE$ has the form $\Phi\cdot(H^2\otimes\HSF)$,
where $\HSF$ is a Hilbert space with $\dim\HSF\le\dim\HSE$ and $\Phi$ is an \df{inner
multiplier}, that is, a holomorphic function $\Phi:\DD\to\cB(\HSF,\HSE)$ such that
multiplication by $\Phi$ is an isometry from $H^2\otimes\HSF$ to $H^2\otimes\HSE$; here
\df{$\cB(\HSF,\HSE)$} denotes the space of bounded linear operators from $\HSF$ to $\HSE$. This
theorem is one of the main links between function theory and operator theory; it underlies the
Sz.-Nagy--Foia\c{s} model theory \cite{Nagybook} and the Lax--Phillips scattering theory
\cite{Phillipsscattering}. See \cite{Helson, BallBolotnikovBLH} for treatments and further
connections, \cite[Section~1.7]{Nikeasy1} for its precursors, and
\cite{BeurlingModulesSarkar, RaulBLH, LuoForwBackw} for recent work.

On $H^2$, Beurling's theorem is one of three closely related classical results. The second is
von Neumann's inequality: every contraction $T$ on a Hilbert space satisfies
$\|p(T)\|\le\|p(S)\|=\sup_{\DD}|p|$ for all polynomials $p$, and, by Sz.-Nagy's dilation
theorem \cite{Nagybook}, the same inequality holds for matrix-valued polynomials. The third is
Pick's interpolation theorem and its generalization, Leech's factorization theorem
\cite{Leechoriginal}: if $V$ and $Y$ are bounded holomorphic matrix-valued functions on $\DD$
such that $M_VM_V^*\succeq M_YM_Y^*$ for the corresponding multiplication operators on
vector-valued $H^2$, then $Y=V\Phi$ for a bounded holomorphic function $\Phi$ with
$\sup_{\DD}\|\Phi\|\le1$.

Each of the three has a natural analogue for a Hilbert space $\cH_\ell$ of holomorphic
functions on a domain $\Omega\subseteq\CC^{\lcal{d}}$ with reproducing kernel $\ell$
(Subsection~\ref{ss:intro-main}), in which the multipliers of $\cH_\ell$, that is, the
functions $\varphi$ with $\varphi\cH_\ell\subseteq\cH_\ell$, take the place of the bounded
holomorphic functions on $\DD$, and the tuple $M_z^\ell$ of multiplication operators by the
coordinate functions on $\cH_\ell$ takes the place of $S$. The analogues are the following
statements: every subspace of $\cH_\ell\otimes\HSE$ invariant under the multipliers of
$\cH_\ell$ is the range of a partially isometric multiplier from a space $\cH_\ell\otimes\HSF$;
the tuples $T$ of commuting operators with Taylor spectrum in $\Omega$ (a joint spectrum; see
Section~\ref{prelims}) such that $\|P(T)\|\le\|P(M_z^\ell)\|$ for all matrix-valued
polynomials $P$ are exactly those with $1/\ell(T,T^*)\succeq0$, the condition obtained by
substituting $T$ for $z$ and $T^*$ for $\ol w$ in $1/\ell(z,w)$ (for $H^2$, where
$1/\ell(z,w)=1-z\ol w$, this condition is $I-TT^*\succeq0$); and Leech's theorem holds with
$\cH_\ell$ in place of $H^2$ and with the conclusion that $\Phi$ is a contractive multiplier of
vector-valued $\cH_\ell$. Each of these statements holds if and only if $\ell$ is a
\textit{complete Pick kernel}, that is, if the multipliers of $\cH_\ell$ satisfy the
matrix-valued analogue of Pick's interpolation theorem; see \cite{mcctrent}, \cite{AgMcClocal}
and \cite[Theorem~8.57]{Pickbook}, respectively. In particular, for a single kernel the three
statements are equivalent.

Complete Pick kernels include the kernels of $H^2$, of the Dirichlet space and of the
Drury--Arveson space; see \cite{Pickbook, Hartzsurvey, shalit2014operator} and, for further
examples, \cite{ShimorinDirichletCP, BesovPick}. The kernels of the Bergman space and of the
Hardy spaces of the polydisk and the ball in dimension at least two are not complete Pick
kernels. For the Bergman space, Aleman, Richter and Sundberg \cite{BeurBergman} showed that
every invariant subspace $\cM$ is generated by the wandering subspace $\cM\ominus z\cM$; see
also \cite{ShimBer2, ShimBer1, ScottBergman}.

For kernels that are not complete Pick kernels, a Beurling-type representation can still be
obtained by letting the representing multiplier act from a second space $\cH_k$. A pair
$(k,\ell)$ of kernels is a \textit{Beurling--Lax--Halmos pair}, or \textit{BLH pair}
(Definition~\ref{BLHdef}), if every subspace of $\cH_\ell\otimes\HSE$ invariant under the
multipliers of $\cH_k$ is the range of a partially isometric multiplier from a space
$\cH_k\otimes\HSF$ to $\cH_\ell\otimes\HSE$. Clou\^atre, Hartz and Schillo
\cite{Beurlingfactor} showed that $(s,\ell)$ is a BLH pair whenever $s$ is a nonvanishing
complete Pick kernel and the quotient $\ell/s$ is a reproducing kernel, written
$\ell/s\succeq0$; this generalizes results of \cite{BB1, BB2, SarkarInvII, SarkarInvI}. For
instance, the kernel $\ell(z,w)=\prod_{i=1}^2(1-z_i\ol{w_i})^{-1}$ of the Hardy space
$H^2(\DD^2)$ of the bidisk is not a complete Pick kernel, so not every subspace of
$H^2(\DD^2)\otimes\HSE$ invariant under multiplication by the coordinate functions is the range
of a partially isometric multiplier from a space $H^2(\DD^2)\otimes\HSF$. However, the kernel
\[
  k_0(z,w)=\Big(1-\frac{z_1\ol{w_1}+z_2\ol{w_2}}{2}\Big)^{-1}
\]
is a complete Pick kernel and
\[
  \frac{\ell}{k_0}(z,w)=\frac12\Big(\frac{1}{1-z_1\ol{w_1}}+\frac{1}{1-z_2\ol{w_2}}\Big)\succeq0,
\]
so every subspace of $H^2(\DD^2)\otimes\HSE$ invariant under multiplication by the coordinate
functions is the range of a partially isometric multiplier from a space
$\cH_{k_0}\otimes\HSF$. Passing to two kernels has also proved useful in other problems of
multivariable operator theory \cite{AHMRfact, freeouter, AHMRinterpol, Timotinsarkar,
LuoRecentInvariant, ColRow, mcctsik, Shimorin, Interpoltsik}.

Our main result, Theorem~\ref{intromain}, shows that the three classical results remain tied
together for pairs of kernels. Under regularity conditions described in
Subsection~\ref{ss:intro-main}, the following are equivalent: (i) $(k,\ell)$ is a BLH pair;
(ii) every tuple $T$ of commuting operators such that $\|P(T)\|\le\|P(M_z^\ell)\|$ for all
matrix-valued polynomials $P$ satisfies $1/k(T,T^*)\succeq0$; (iii) the analogue of Leech's
theorem holds in which the hypothesis is imposed on $\cH_k$ and the factor $\Phi$ is a
contractive multiplier of $\cH_\ell$; (iv) $1/k$ is a limit of functions of the form
$Q(z)\big[I-P(z)P(w)^*\big]Q(w)^*$, where $P$ is a matrix-valued polynomial that is a
contractive multiplier of vector-valued $\cH_\ell$ and $Q$ is holomorphic, in the spirit of
Agler's approach to Pick interpolation \cite{Hellinger}.

By the result of \cite{Beurlingfactor}, $(k,\ell)$ is a BLH pair whenever a nonvanishing
complete Pick kernel $s$ satisfies $\ell/s\succeq0$ and $s/k\succeq0$, provided that the
subspaces of $\cH_\ell\otimes\HSE$ invariant under the multipliers of $\cH_k$ are invariant
under those of $\cH_\ell$ (Proposition~\ref{CPinbetween}). For a single kernel, this
sufficient condition is also necessary, since the BLH property forces $\ell$ to be a complete
Pick kernel, and $s=\ell$ then satisfies it. For pairs it is not: in Section~\ref{Exampsec} we
construct BLH pairs satisfying our regularity conditions, on the disk and on the bidisk, for
which no such $s$ exists, and which are not complete Pick pairs in the sense of Shimorin
\cite{Shimorin}. On the bidisk, these pairs are $(k_\tau,\ell)$, with $\ell$ the kernel of
$H^2(\DD^2)$ and $0<\tau\le3-2\sqrt2$; the formula defining $k_\tau$ gives $k_0$ for $\tau=0$,
the case covered by \cite{Beurlingfactor}.

We also obtain explicit criteria in classical settings, under conditions stated in
Subsections~\ref{ss:intro-CP} and~\ref{ss:intro-exact}. On the disk, when $M_z^\ell$ is a
contraction, a pair $(k,\ell)$ is
a BLH pair if and only if $(1-z\ol w)^{-1}/k(z,w)$ is a reproducing kernel. On the ball, for
kernels invariant under the unitary group, the same holds with the Drury--Arveson kernel
$(1-\langle z,w\rangle)^{-1}$ in place of the Szeg\H{o} kernel $(1-z\ol w)^{-1}$, provided that
$M_z^\ell$ is a row contraction, that is,
$\sum_jM_{z_j}^\ell\big(M_{z_j}^\ell\big)^*\preceq I$. On the bidisk, with $\ell$ the kernel of
$H^2(\DD^2)$, $(k,\ell)$ is a BLH pair if and only if
$1/k=(1-z_1\ol{w_1})g_1+(1-z_2\ol{w_2})g_2$ for holomorphic kernels $g_1,g_2$; on
$\DD^{\lcal{d}}$ with $\lcal{d}\ge3$, the analogous condition is sufficient but not necessary,
by an example built from the Kaijser--Varopoulos counterexample to von Neumann's inequality
for three commuting contractions. On the symmetrized bidisk, with $\ell$ the kernel of its
Hardy space, the BLH property is equivalent
to an explicit operator-valued representation of $1/k$. The rest of this introduction makes
these statements precise.

\subsection{Definitions and the main theorem}\label{ss:intro-main}

We now make the terminology precise; Section~\ref{prelims} contains further details. A
function $\ell:\Omega\times\Omega\to\CC$ on a domain $\Omega\subseteq\CC^{\lcal{d}}$ is a
\df{reproducing kernel}, written \df{$\ell\succeq 0$}, if the matrix
$[\ell(z_i,z_j)]_{i,j=1}^n$ is positive semidefinite for all $n$ and all
$z_1,\dots,z_n\in\Omega$; for a nonvanishing kernel $k$, we write $\ell/k\succeq0$ if the
quotient $\ell/k$ is a reproducing kernel. The associated \df{reproducing kernel Hilbert space}
(\df{RKHS}), or \df{Hilbert function space}, \df{$\cH_\ell$} consists of functions on $\Omega$,
contains the functions $\ell_w=\ell(\cdot,w)$, and satisfies $\langle f,\ell_w\rangle=f(w)$ for
$f\in\cH_\ell$ and $w\in\Omega$; elements of \df{$\cH_\ell\otimes\HSE$} are regarded as
$\HSE$-valued functions on $\Omega$. For instance, $H^2$ is the RKHS of the Szeg\H{o} kernel
$\mathfrak{s}_1(z,w)=(1-z\ol w)^{-1}$ on $\DD$, the Hardy space $H^2(\DD^{\lcal{d}})$ is that of
the Szeg\H{o} kernel $\prod_{i=1}^{\lcal{d}}(1-z_i\ol{w_i})^{-1}$ of the polydisk
$\DD^{\lcal{d}}$, and the Drury--Arveson space $H^2_{\lcal{d}}$ is that of the kernel
$\mathfrak{s}(z,w)=(1-\langle z,w\rangle)^{-1}$ on the unit ball $\mathbb{B}_{\lcal{d}}$ of
$\CC^{\lcal{d}}$. Given a second kernel $k$ on $\Omega$, a function
$\Phi:\Omega\to\cB(\HSF,\HSE)$ is a \df{multiplier} from $\cH_k\otimes\HSF$ to
$\cH_\ell\otimes\HSE$ if $\Phi F\in\cH_\ell\otimes\HSE$ for every $F\in\cH_k\otimes\HSF$;
by the closed graph theorem, the operator \index{$M_\Phi$}$M_\Phi F=\Phi F$ is then bounded.
The space of such multipliers is denoted
\df{$\Mult(\cH_k\otimes\HSF,\cH_\ell\otimes\HSE)$}, its norm $\|\Phi\|=\|M_\Phi\|$ is the
\df{multiplier norm}, and $\Phi$ is \df{contractive}, respectively \df{partially isometric},
if $M_\Phi$ has the corresponding property. We write $\Mult(\cH_\ell)$ for $\Mult(\cH_\ell,\cH_\ell)$ and, when the
coordinate functions are multipliers, $M_z^\ell=(M_{z_1}^\ell,\dots,M_{z_{\lcal{d}}}^\ell)$
for the tuple of multiplication operators by them on $\cH_\ell$. A subspace
$\cM\subseteq\cH_\ell\otimes\HSE$ is \df{$\Mult(\cH_k)$-invariant} if $\varphi F\in\cM$ for
all $\varphi\in\Mult(\cH_k)$ and $F\in\cM$, and \textit{shift-invariant} if $z_jF\in\cM$ for
all $j$ and $F\in\cM$. When the coordinate functions are multipliers of $\cH_k$,
$\Mult(\cH_k)$-invariant subspaces are shift-invariant, and under mild conditions
$\Mult(\cH_k)$-invariance, $\Mult(\cH_\ell)$-invariance and shift-invariance coincide
(Proposition~\ref{allequiv}). Complete Pick kernels were described in
Subsection~\ref{ss:intro-question}; see Subsection~\ref{CPconnectsubsec} for an equivalent
description.

\begin{definition} \label{BLHdef}
  A pair $(k, \ell)$ of reproducing kernels on a set $X$ is a \df{Beurling--Lax--Halmos pair}
  (or \df{BLH pair} for short) if, for every Hilbert space $\mathcal{E}$ and every
  $\textup{Mult}(\mathcal{H}_k)$-invariant subspace $\mathcal{M}\subseteq
  \mathcal{H}_{\ell}\otimes\mathcal{E}$, there exist a Hilbert space $\mathcal{F}$ and a
  partially isometric multiplier
  $\Phi\in\text{Mult}(\mathcal{H}_k\otimes\mathcal{F}, \mathcal{H}_{\ell}\otimes\mathcal{E})$
  such that
  \[
    \mathcal{M}=\Phi(\mathcal{H}_k\otimes\mathcal{F}).
  \]
  In the special case $k=\ell,$ we say that $k$ is a \df{BLH kernel}.
  \qed
\end{definition}

By \cite{mcctrent}, the BLH kernels are exactly the complete Pick kernels, with the caveat that
the multiplier $\Phi$ in Definition~\ref{BLHdef} is in general only partially isometric. Indeed,
the dimension of $\HSF$ generally exceeds that of $\HSE$ (though, with some regularity
assumptions, one can actually say more \cite{GreeneRichtSund}).

Our main theorem assumes regularity conditions. A pair $(k,\ell)$ of kernels on a domain
$\Omega\subseteq\CC^{\lcal{d}}$ is \textit{regular}, respectively \textit{strongly regular}, if
it satisfies the conditions of Definition~\ref{def:regular}, in addition to the standing
Assumption~\ref{assume}: both kernels are holomorphic, polynomials are dense in $\cH_k$ and in
$\cH_\ell$, and the coordinate functions are multipliers of both spaces. Roughly, regularity
means the following: the Taylor spectrum of $M_z^\ell$ is $\ol\Omega$, which is polynomially convex; the
function $(z,w)\mapsto1/k(z,\ol w)$ is holomorphic on $\Omega'\times(\Omega')^*$ for an open
set $\Omega'\supseteq\ol\Omega$, where $(\Omega')^*=\{\ol z:z\in\Omega'\}$; there are polynomial maps $r_n:\CC^{\lcal{d}}\to\CC^{\lcal{d}}$,
playing the role of the scalings $z\mapsto tz$ with $t\uparrow1$, that map $\ol\Omega$ into
$\Omega$ and $\Omega'$ into $\Omega'$, such that $1/k(r_n(z),r_n(w))\to1/k(z,w)$ uniformly on
compact subsets of $\Omega'\times\Omega'$ and the maps $p(M_z^\ell)\mapsto(p\circ
r_n)(M_z^\ell)$ are completely contractive (Definition~\ref{kexhaust}); and
the shift-invariant, $\Mult(\cH_\ell)$-invariant and $\Mult(\cH_k)$-invariant subspaces of
$\cH_\ell\otimes\HSE$ coincide for every $\HSE$. Strong regularity requires, in addition, an exhaustion with the dilation property: for
every $\Mult(\cH_\ell)$-contractive tuple $T$ on a Hilbert space $H$
(Definition~\ref{def:mult-contraction}) and every $n$, there are a Hilbert space $\HSF$ and an
isometry $V:H\to\cH_\ell\otimes\HSF$ such that
$p(r_n(T))=V^*\big(p(M_z^\ell)\otimes I_\HSF\big)V$ for every polynomial $p$.
Lemma~\ref{l:dilation-criterion} and Theorem~\ref{t:radial-dilation} give sufficient
conditions for this property.
Section~\ref{Exampsec} verifies these conditions in a range of cases, including pairs
$(k,\ell)$ with $\ell$ the kernel of the Hardy space of the polydisk, of the ball or of the
symmetrized bidisk.

A $\lcal{d}$-tuple $T=(T_1,\dots,T_{\lcal{d}})$ of commuting bounded operators on a Hilbert
space is \textit{$\Mult(\cH_\ell)$-contractive} (Definition~\ref{def:mult-contraction}) if
$\|P(T)\|\le\|P\|_{\Mult(\cH_\ell\otimes\CC^m)}=\|P(M_z^\ell)\|$ for every $m$ and every
polynomial $P\in\CC[z_1,\dots,z_{\lcal{d}}]\otimes\CC^{m\times m}$. By Arveson's dilation
theorem \cite{ArvesonIII}, the $\Mult(H^2_{\lcal{d}})$-contractive tuples are the
\df{row contractions}, that is, the tuples with $\sum_jT_jT_j^*\preceq I$. For a regular pair
$(k,\ell)$ and a tuple $T$ with Taylor spectrum in $\ol\Omega$, the operator $1/k(T,T^*)$ is
defined by the hereditary functional calculus (Subsection~\ref{s:her-calc}); if
$1/k(z,w)=\sum_{\alpha,\beta}c_{\alpha\beta}\,z^\alpha\ol{w}^\beta$ is a hereditary
polynomial, then $1/k(T,T^*)=\sum_{\alpha,\beta}c_{\alpha\beta}\,T^\alpha T^{*\beta}$. The
tuple $T$ is a \textit{$1/k$ contraction} (Definition~\ref{1kdef}) if $1/k(T,T^*)\succeq0$;
for $k=\mathfrak{s}_1$ and $k=\mathfrak{s}$, the $1/k$ contractions are, respectively, the
contractions and the row contractions. A regular pair $(k,\ell)$ is a \textit{co-extension
pair} (Definition~\ref{coextprop}) if every $\Mult(\cH_\ell)$-contractive tuple, which has
Taylor spectrum in $\ol\Omega$ by Remark~\ref{SpInclusion}, is a $1/k$ contraction. The
terminology reflects the fact that a $1/k$ contraction $T$ with Taylor spectrum in $\Omega$ is
unitarily equivalent to the compression of $M_z^k\otimes I_\HSE$ to a co-invariant subspace
\cite[Corollary~7]{Polynomialarazyenglis}, that is, $M_z^k\otimes I_\HSE$ is a co-extension
of $T$.

In terms of kernels, the hypothesis of Leech's theorem says that
$\big[V(z)V(w)^*-Y(z)Y(w)^*\big]\mathfrak{s}_1(z,w)\succeq0$; see also
\cite[Theorem~8.57]{Pickbook} and \cite{RandR}. The pair $(k,\ell)$ is a \textit{complete
pseudo-Leech pair} (Definition~\ref{Leechdef}) if, for all positive integers $n,a,b,p$, points
$z_1,\dots,z_n\in\Omega$ and matrices $V_i\in\CC^{p\times a}$ and $Y_i\in\CC^{p\times b}$ such
that
\[
  \big[(V_iV_j^*-Y_iY_j^*)\,k(z_i,z_j)\big]_{i,j=1}^n\succeq0,
\]
there is a contractive multiplier $\Phi\in\Mult(\cH_\ell\otimes\CC^b,\cH_\ell\otimes\CC^a)$
with $Y_i=V_i\Phi(z_i)$ for all $i$. Thus $\mathfrak{s}_1$ is replaced by $k$ in the
hypothesis, and the conclusion asks for a multiplier of $\cH_\ell$.

The \textit{Agler wedge} $\mathcal{AG}(\ell)$ (Definition~\ref{Aglerwedge}) is the closure,
in the topology of uniform convergence on compact subsets of $\Omega\times\Omega$, of the set
of functions
\[
  Q(z)\big[I-P(z)P(w)^*\big]Q(w)^*,
\]
where $m\ge1$, $P$ is a contractive polynomial multiplier of $\cH_\ell\otimes\CC^m$, and
$Q:\Omega\to\cB(\CC^m,\CC)$ is holomorphic. For the Szeg\H{o} kernel $\ell$ of $\DD^2$,
Agler's theorem \cite{Hellinger} identifies $\mathcal{AG}(\ell)$ with the set of functions
\begin{equation}\label{e:intro-bidisk}
  (1-z_1\ol{w_1})g_1(z,w)+(1-z_2\ol{w_2})g_2(z,w),
\end{equation}
where $g_1,g_2$ are holomorphic kernels (Example~\ref{bidiskex}).

\Needspace{12\baselineskip}
\begin{theorem}\label{intromain}
  If $(k, \ell)$ is a strongly regular pair of reproducing kernels on a domain
  $\Omega\subseteq \CC^{\lcal{d}}$, then the following are equivalent.
  \begin{enumerate}[(i)]\itemsep=6pt
    \item \label{i:im:blh} $(k, \ell)$ is a BLH pair;
    \item \label{i:im:vN} $(k,\ell)$ is a co-extension pair;
    \item \label{i:im:leach} $(k,\ell)$ is a complete pseudo-Leech pair;
    \item \label{i:im:ad} $1/k\in\mathcal{AG}(\ell)$;
    \item \label{i:im:exact} there exist separable Hilbert spaces $\HSF$ and $\HSG$, a
      holomorphic function $Q:\Omega\to\cB(\HSG,\CC)$ and a contractive multiplier
      $\Phi\in\Mult(\cH_\ell\otimes\HSF,\cH_\ell\otimes\HSG)$ such that
      \[
        \frac{1}{k(z,w)}=Q(z)\big[I-\Phi(z)\Phi(w)^*\big]Q(w)^*,\qquad z,w\in\Omega.
      \]
  \end{enumerate}
\end{theorem}

Theorem~\ref{intromain} follows from Theorem~\ref{DerSatz}, which records the hypotheses
needed for each implication; several implications hold under much weaker assumptions.
Item~\ref{i:im:exact} is the exact form of item~\ref{i:im:ad}. No additional decomposability
hypothesis is needed here: Lemma~\ref{l:hol-decomp} shows that $1/k$ is the difference of two
holomorphic kernels, and Corollary~\ref{strongAglertoAgler} supplies the exact representation.

\subsection{Complete Pick factors}\label{ss:intro-CP}

Following \cite{Beurlingfactor}, a complete Pick kernel $s$ with $\ell/s\succeq0$ is a
\textit{complete Pick factor} of $\ell$. By Proposition~\ref{CPinbetween}, if
$\Mult(\cH_k)$-invariant subspaces of $\cH_\ell\otimes\HSE$ are $\Mult(\cH_\ell)$-invariant and
there is a nonvanishing complete Pick kernel $s$ such that
\begin{equation}\label{e:intro-CPbetween}
  \ell/s\succeq0,\qquad s/k\succeq0,
\end{equation}
then $(k,\ell)$ is a BLH pair. In two classical settings, the existence of a nonvanishing
complete Pick kernel $s$ satisfying \eqref{e:intro-CPbetween} is also necessary.

On the disk, let $(k,\ell)$ be a strongly regular pair such that $M_z^\ell$ is a contraction,
as it is, for instance, when $\cH_\ell$ is the Hardy space or a standard weighted Bergman space.
Then $(k,\ell)$ is a BLH pair if and only if $\mathfrak{s}_1/k\succeq0$
(Proposition~\ref{p:disk-universal}). Since $\ell/\mathfrak{s}_1\succeq0$ exactly when
$M_z^\ell$ is a contraction, the kernel $s=\mathfrak{s}_1$ then satisfies
\eqref{e:intro-CPbetween} whenever $(k,\ell)$ is a BLH pair. On the ball
$\mathbb{B}_{\lcal{d}}$, the Drury--Arveson kernel $\mathfrak{s}$ plays the same role for
kernels invariant under the unitary group: if $M_z^\ell$ is a row contraction, then, under mild
additional conditions, $(k,\ell)$ is a BLH pair if and only if $\mathfrak{s}/k\succeq0$, and,
when $(k,\ell)$ is regular, if and only if every row contraction is a $1/k$ contraction
(Example~\ref{regularballex}). Both equivalences fail if $k$ is only assumed invariant under the
maps $z\mapsto e^{i\theta}z$.

Without these hypotheses, a complete Pick kernel $s$ satisfying \eqref{e:intro-CPbetween} need
not exist. Example~\ref{eg:recursivel} gives a strongly regular BLH pair $(k,\ell)$ of diagonal
holomorphic kernels on $\DD$, that is, of kernels of the form
$\sum_{\alpha}c_\alpha z^\alpha\ol{w}^\alpha$, for which no complete Pick kernel $s$ satisfies
\eqref{e:intro-CPbetween}; by Proposition~\ref{p:disk-universal}, $M_z^\ell$ is then not a
contraction. In contrast, Example~\ref{bidiskfamilyex} gives a family of strongly regular BLH
pairs $(k_\tau,\ell)$ on $\DD^2$, with $\ell$ the Szeg\H{o} kernel of $\DD^2$, for which
$M_{z_1}^\ell$ and $M_{z_2}^\ell$ are isometries and no complete Pick kernel $s$ satisfies
\eqref{e:intro-CPbetween}.

Examples~\ref{eg:recursivel} and~\ref{bidiskfamilyex} also separate the BLH property from the
complete Pick property for pairs of kernels, introduced by Shimorin \cite{Shimorin}
(Definition~\ref{CPdef}): $(k,\ell)$ is a complete Pick pair if every interpolation problem
$\Phi(z_i)=W_i$, with points $z_1,\dots,z_n\in\Omega$ and matrices $W_i\in\CC^{N\times N}$,
that satisfies the necessary condition
$\big[\ell(z_i,z_j)I-W_iW_j^*k(z_i,z_j)\big]_{i,j=1}^n\succeq0$ has a solution $\Phi$ among the
contractive multipliers in $\Mult(\cH_k\otimes\CC^N,\cH_\ell\otimes\CC^N)$; for $k=\ell$, this
is the definition of a complete Pick kernel. For pairs of diagonal holomorphic kernels
normalized so that $k(0,0)=\ell(0,0)=1$, the existence of a complete Pick kernel $s$ satisfying
\eqref{e:intro-CPbetween} is necessary, but not sufficient, for the complete Pick property
\cite[Theorem~1.4 and Proposition~7.4]{mcctsik}. Such a complete Pick pair is a BLH pair when
the invariance hypothesis of Corollary~\ref{c:Shim-cert-implies-blh} holds, in particular when
the pair is regular (Theorem~\ref{CPpaircharact}). On the other hand, the pairs of Examples~\ref{eg:recursivel}
and~\ref{bidiskfamilyex} are complete pseudo-Leech pairs but not complete Pick pairs. Thus,
among strongly regular pairs of normalized diagonal holomorphic kernels, the complete Pick
pairs form a proper subclass of the BLH pairs. The implication from the complete Pick property
to the BLH property for general pairs remains open.

\subsection{Exact criteria on the bidisk and the symmetrized bidisk}\label{ss:intro-exact}

Theorem~\ref{intromain} also yields exact criteria on particular domains. If $(k,\ell)$ is a
strongly regular pair on $\DD^2$ and $\Mult(\cH_\ell)=H^\infty(\DD^2)$ isometrically, as for the
Szeg\H{o} kernel, then $(k,\ell)$ is a BLH pair if and only if $1/k$ has the form
\eqref{e:intro-bidisk} for holomorphic kernels $g_1,g_2$ (Example~\ref{bidiskex}); the kernels
$k_\tau$ of Example~\ref{bidiskfamilyex} are defined by a representation of this form. For the
Szeg\H{o} kernel $\ell_{\lcal{d}}$ of $\DD^{\lcal{d}}$ with $\lcal{d}\ge3$, the corresponding
condition $1/k=\sum_{i=1}^{\lcal{d}}(1-z_i\ol{w_i})g_i$ is still sufficient for a strongly
regular pair $(k,\ell_{\lcal{d}})$ to be a BLH pair, but Example~\ref{polydiskex} shows that it
is not necessary. Thus the characterization of Example~\ref{bidiskex} holds on $\DD^2$, where
pairs of commuting contractions satisfy von Neumann's inequality \cite{Ando}, and fails on
$\DD^{\lcal{d}}$ for $\lcal{d}\ge3$, where triples of commuting contractions need not satisfy
it \cite{Varopoulos}; Example~\ref{polydiskex} is built from such a triple.

Example~\ref{symbidiskex} gives an exact criterion for pairs $(k,\ell_G)$, where $\ell_G$ is
the kernel of the Hardy space of the symmetrized bidisk
$G=\{(\lambda_1+\lambda_2,\lambda_1\lambda_2):\lambda\in\DD^2\}$; the kernel $\ell_G$ is not a
complete Pick kernel. If $k$ is a nonvanishing kernel on $G$ that satisfies
Assumption~\ref{assume} and the function $(z,w)\mapsto1/k(z,\ol w)$ is holomorphic on
$\Omega'\times(\Omega')^*$ for an open set $\Omega'\supseteq\ol G$, then $(k,\ell_G)$ is
strongly regular. In that case, $(k,\ell_G)$ is a BLH pair if and only if $1/k(z,w)=Q(z)H_U(z,w)Q(w)^*$, where $Q$ is holomorphic and $H_U$
is an explicit operator-valued hereditary polynomial, affine in a unitary operator $U$ and its
adjoint (see \eqref{e:ex-G-exact}). The proof combines Theorem~\ref{DerSatz} and
Corollary~\ref{strongAglertoAgler} with the realization theorem of Agler and Young
\cite[Theorem~2.5]{AYrealization}.

\subsection{Organization}\label{ss:intro-guide}

Section~\ref{prelims} introduces the regularity conditions and the hereditary functional
calculus. Sections~\ref{s:moreBLH} and~\ref{Leechsecc} work at a level of generality where
analyticity is mostly not needed: Section~\ref{s:moreBLH} introduces the Agler wedge and
related representations of $1/k$, proves the sufficient condition behind the implication
\ref{i:im:ad}$\Rightarrow$\ref{i:im:blh} of Theorem~\ref{intromain}
(Proposition~\ref{gensuffBLH}) and compares the BLH and complete Pick properties, and
Section~\ref{Leechsecc} treats the complete pseudo-Leech property. Section~\ref{Heredsection}
shows that, for regular pairs, the co-extension property is equivalent to item~\ref{i:im:ad}
of Theorem~\ref{intromain} (Theorem~\ref{COEXTISAGLER}), and Section~\ref{BLHandco-extSec}
links the co-extension and BLH properties (Theorem~\ref{BLHISCOEXT}).
Section~\ref{mastersec} assembles these results (Theorem~\ref{DerSatz}),
Section~\ref{Exampsec} contains the examples, and Section~\ref{questions} lists open questions.
Appendices~\ref{appendMontel} and~\ref{appendLurking} prove two results whose proofs use
standard techniques but whose statements we have not found in the literature.
Appendix~\ref{append-radial} proves the radial dilation theorem.

\section{Preliminaries} 
\label{prelims}
 The first two subsections of this section introduce the background and terminology  needed to make  precise the  statement of Theorem~\ref{intromain}. 
 Various regularity assumptions appear in Subsection~\ref{sss:regular}. Apart from the dilation property, most of these are fairly standard and will be familiar to practitioners. 
 The definition of 
  a $1/k$ contraction appears in  Subsection~\ref{s:her-calc} after first summarizing 
   standard facts  about the hereditary functional calculus, a natural functional calculus for multiplication operators on 
   reproducing kernel Hilbert spaces.  This section 
  concludes with the optional Subsection~\ref{subsecconseq}. It explores some consequences of, and contains some contextual observations about, the regularity assumptions.

 For the present purposes, a kernel $\genk$ is a \df{holomorphic kernel} or  an \df{analytic kernel}  if it is a kernel
 over a connected open set (\df{domain})  $\Omega\subseteq \CC^{\lcal{d}}$  and
 $\kappa(z,w)$ is holomorphic in $z$ and $\overline{w}$, where
 $\Omega^*=\{\ol z:z\in\Omega\}$. We use inner products linear in the first variable and
 write $\NN=\{0,1,2,\dots\}$. Hence $h:\Omega\times \Omega^*\to\CC$
 defined by $h(z,w)=\kappa(z,\overline{w})$ is holomorphic.

\subsection{Regular pairs of kernels}
\label{sss:regular}

\begin{assumption}\rm
 \label{assume}
Unless the context indicates otherwise, the following assumptions are in force.
\begin{enumerate}[(i)]\itemsep=6pt
 \item  
 Kernels $\genk(z,w)$  are holomorphic.

 \item \label{i:assume:ii}
     Polynomials are dense in the reproducing kernel Hilbert space $\mathcal{H}_\genk.$ 
 
  \item \label{i:assume:iii}
  The $\lcal{d}$-tuple \df{$M^\genk_z=(M^\genk_{z_1}, \dots, M^\genk_{z_\lcal{d}})$} of operators of multiplication by the coordinate functions on $\cH_\genk$ consists of bounded operators; that is, each $z_j$ is  a multiplier
  of $\cH_\genk.$  \qed
\end{enumerate}
\end{assumption}

\begin{remark}
  \label{r:assume-Omega-bounded}
    Since, as is well known, the functions $\genk(\cdot,w)$ are joint
    eigenvectors for $(M_z^\genk)^*,$ the boundedness assumption of
    item~\ref{i:assume:iii} implies $\Omega\subseteq \specT(M_z^\genk)$  and in particular $\Omega$ is a bounded domain.
    \qed
\end{remark}

The spaces $\Her(\cdot)$ of hereditary functions used in the following definition are
introduced in Subsection~\ref{s:her-calc}.

\begin{definition}\label{kexhaust}
 A pair of kernels  $(k,\ell)$  \df{admits an exhaustion} if there exists an open set  $\Omega'\supset\overline{\Omega}$ and a sequence $\{r_n\}$ of polynomial maps $\CC^{\lcal{d}}\to\CC^{\lcal{d}}$
 such that, setting $k_n(z, w)=k(r_n(z), r_n(w))$ and $\ell_n(z, w)=\ell(r_n(z), r_n(w))$,
   \begin{enumerate}[(i)]\itemsep=6pt
  	\item $k$ is nonvanishing and $1/k$ extends to an element of $\Her(\Omega')$; 
   
        \item  For each $n$ the inclusions $r_n(\overline{\Omega})\subset \Omega$  and  $r_n(\Omega^\prime)\subseteq \Omega^\prime$ hold;

        \item  \label{i:exh:ucs} 
         The sequence $(1/k_n)$ converges to $1/k $ uniformly on compact subsets of $\Omega'\times\Omega'$; 

        \item  \label{i:exh:cc}
         For each $n,$ the map 
        \[
          p(M^{\ell}_z)\mapsto p\circ r_n(M^{\ell}_z),  
        \]    
       for $p\in \CC[z_1,\dots,z_{\lcal{d}}]$    is completely contractive. 
       \end{enumerate}
If in addition, 
\begin{enumerate}[resume*]  \itemsep=6pt
        \item \label{i:annoying}  $\ell=\ell_n g_n$ for some positive kernel $g_n$ on $\Omega\times\Omega$, for all $n\ge 1,$
   \end{enumerate}
   then $(k,\ell)$  \df{admits a  strong exhaustion}. 

   We refer to the pair $(r_n,\Omega^\prime)$ as a \df{$(k,\ell)$-exhaustion}.
\qed
 \end{definition}

\begin{remark}
Property~\ref{i:annoying} says that each $\ell_n$ is a factor of $\ell$. By
Lemma~\ref{l:cert-cont-include}, it implies that each $\textup{Mult}(\mathcal{H}_{\ell_n})$ is
completely contractively contained in $\textup{Mult}(\mathcal{H}_{\ell})$, and hence
property~\ref{i:exh:cc}, since, as $r_n(\Omega)\subseteq\Omega$, for every $m$ and
$P\in\CC[z_1,\dots,z_{\lcal{d}}]\otimes\CC^{m\times m}$ the function $P\circ r_n$ is a multiplier
of $\cH_{\ell_n}\otimes\CC^m$ of norm at most $\|P\|_{\Mult(\cH_\ell\otimes\CC^m)}$. 
\qed\end{remark}

 \begin{remark}
If $(r_n,\Omega^\prime)$ is an exhaustion (resp. strong exhaustion)  and $U$ is the connected
component of $\Omega^\prime$ containing $\Omega,$ then
$(r_n,U)$ is also an exhaustion (resp. strong exhaustion).
\qed \end{remark}

 Examples~\ref{eg:certifiably-annoying} and \ref{eg:certify:cc} provide  classes of kernels where items~\ref{i:annoying}
 and \ref{i:exh:cc}, respectively, of Definition~\ref{kexhaust} are certifiable.

\begin{example}
\label{eg:certifiably-annoying}
Let $h$ be a \df{diagonal} holomorphic kernel on a bounded domain $\Omega\subset\mathbb{C}^{\lcal{d}},$ so 
\[
h(z,w)=\sum_{\alpha\in\mathbb{N}^{\lcal{d}}}h_{\alpha}z^{\alpha}\overline{w}^{\alpha}.
\]
Set $\ell(z,w):=e^{h(z,w)}$. Since, for $0<r<1,$ 
\[
  \cfrac{\ell(z,w)}{\ell(rz, rw)}=e^{\sum_{\alpha}(1-r^{2|\alpha|})h_{\alpha}z^{\alpha}\overline{w}^{\alpha}},
\]
(here $h(rz,rw)$, and hence $\ell(rz,rw)$, is defined by the power series of $h$, which converges absolutely at $(rz,rw)$ since its coefficients are nonnegative)
we conclude that $\ell$ satisfies item~\ref{i:annoying} of Definition~\ref{kexhaust} with $\ell_n(z,w)=\ell(r_n z,r_nw)$
for any sequence $(r_n)$ satisfying $0<r_n<1$ that converges to $1.$ This class of examples includes all diagonal holomorphic kernels 
of the form $\ell=(1-P)^{-1}$, where $P$ is a diagonal positive kernel with $P(z,z)<1$
(so $\ell$ is a complete Pick kernel), by choosing
$h=-\log(1-P)$ and, more generally, their finite products.  
The Bergman kernel $(1-z\ol{w})^{-2}$ for the unit disk in the complex plane is a concrete example of such an $\ell$.
\qed
\end{example}

\begin{example}
\label{eg:certify:cc}
   Suppose $0<r_n<1$ converges to $1$  and define $r_n(z)=r_n z$. Given a unitarily invariant RKHS $\mathcal{H}_{\ell}$ on the $\lcal{d}$-dimensional open unit ball $\mathbb{B}_{\lcal{d}}$ with 
   \[\ell(z,w)=\sum_{j=0}^{\infty}\ell_j\langle z,w \rangle^j \]
   and $\lim\ell_{j+1}/\ell_j=1,$
   \cite[Lemma 2.2]{HartzJEMS} shows that the map $p(M^{\ell}_z)\mapsto p\circ r_n(M^{\ell}_z)$ is completely contractive for all $n\ge 1$. Its proof, which averages over the circle action, applies to every circularly symmetric $\ell$ satisfying Assumption~\ref{assume}. Thus, all such kernels satisfy item~\ref{i:exh:cc} of Definition~\ref{kexhaust}.
\end{example}

\begin{definition}
\label{def:regular}
A pair $(k, \ell)$ of analytic kernels over the domain
$\Omega\subseteq \CC^{\lcal{d}}$ is a \df{regular pair} if it satisfies the
following axioms.

\begin{enumerate}[(i)]\itemsep=6pt
   \item \label{i:regular:specT}
    $\specT(M^{\ell}_z)= \ol{\Oml}$;

  \item \label{i:regular:poly-cvx}
   $\ol{\Omega}$ is polynomially convex;

  \item \label{i:regular:exh}
   The pair  $(k,\ell)$ admits an exhaustion;

  \item \label{i:regular:include}
   For each Hilbert space $\HSE$, the subspaces of $\cH_\ell\otimes\HSE$ that are
   invariant for $M_z^\ell\otimes I_\HSE$ (the \df{shift-invariant} subspaces),
   the $\Mult(\cH_\ell)$-invariant subspaces and the $\Mult(\cH_k)$-invariant
   subspaces coincide.
\end{enumerate}
A $(k,\ell)$-exhaustion $(r_n,\Omega')$ has the \df{dilation property} if, for
every $\lcal{d}$-tuple $T$ of commuting operators on a Hilbert space $H$ such that
\[
  \|P(T)\|\le\|P\|_{\Mult(\cH_\ell\otimes\CC^m)}\quad\text{for all }m\ge1
  \text{ and }P\in\CC[z_1,\dots,z_{\lcal{d}}]\otimes\CC^{m\times m},
\]
and for every $n\ge1$, there exist a Hilbert space $\HSF$ and an isometry
$V:H\to\cH_\ell\otimes\HSF$ such that
\[
  p(r_n(T))=V^*\big(p(M_z^\ell)\otimes I_\HSF\big)V\quad\text{for all }
  p\in\CC[z_1,\dots,z_{\lcal{d}}].
\]
If in addition
 \begin{enumerate}[resume*] \itemsep=6pt
 \item \label{i:regular:dilation}
   $(k,\ell)$ admits an exhaustion with the dilation property,
 \end{enumerate}
 then $(k,\ell)$ is a \df{strongly regular pair}.
\qed
\end{definition}

\begin{remark}\label{r:regular} We pause to make several remarks regarding
Definition~\ref{def:regular}.
\begin{enumerate}[wide, labelwidth=1em, labelindent=0pt] \itemsep=6pt
\item If $(k, \ell)$ is a regular pair, then, by item~\ref{i:regular:exh}, $k$ is
  nonvanishing and $1/k$ extends to an element of $\Her(\Omega')$ for some open
  set $\Omega'\supseteq\ol{\Omega}$.

\item By Proposition~\ref{allequiv}, item~\ref{i:regular:include} holds whenever
  $k$ and $\ell$ satisfy Assumption~\ref{assume},
  $\Mult(\cH_k)\subseteq\Mult(\cH_\ell)$ and the polynomials are weak-$*$ dense in
  $\Mult(\cH_\ell)$. By Corollary~\ref{c:allequiv}, the density condition can be
  dropped when $\ell$ is circularly symmetric (in the sense defined before that
  corollary).

\item In the terminology of Definition~\ref{def:mult-contraction} and
  Lemma~\ref{dilation-decomp}, an exhaustion $(r_n,\Omega')$ has the dilation
  property if and only if, for every $\Mult(\cH_\ell)$-contractive tuple $T$ and
  every $n$, the tuple $r_n(T)$ is unitarily equivalent to the compression of
  $M_z^\ell\otimes I_\HSF$ to a semi-invariant subspace, for some Hilbert space
  $\HSF$. Applied to $T=M_z^\ell$, the dilation property implies
  item~\ref{i:exh:cc} of Definition~\ref{kexhaust}.
  Lemma~\ref{l:dilation-criterion} and Theorem~\ref{t:radial-dilation} give
  sufficient conditions for the dilation property, and
  Corollary~\ref{c:radial-strong} gives a class of strongly regular pairs.
\end{enumerate}
\qed
 \end{remark}

\begin{example}\rm
  \label{eg:BLH-vacuous}
  Item~\ref{i:regular:include} of Definition~\ref{def:regular} cannot be
  weakened to the requirement that every $\Mult(\cH_k)$-invariant subspace be
  $\Mult(\cH_\ell)$-invariant. On $\DD$, let $k(z,w)=(1-z\ol{w})^{-1}$ and
  \[
    \ell(z,w)=(1-z\ol{w})^{-1/2}=\sum_{j\ge0}a_j(z\ol{w})^j,\qquad
    a_j=\frac{(1/2)_j}{j!},
  \]
  so that $a_j$ decreases to $0$. The monomials are orthogonal in $\cH_\ell$ with
  $\|z^j\|^2=1/a_j$, so $S=M_z^\ell$ is the weighted shift with weights
  $\big((j+1)/(j+\frac12)\big)^{1/2}$; hence $\|S\|=\sqrt2$ and
  $\specT(S)=\ol{\DD}$. Every multiplier of $\cH_\ell$ is a bounded holomorphic
  function, so $\Mult(\cH_\ell)\subseteq H^\infty(\DD)=\Mult(\cH_k)$, and every
  $\Mult(\cH_k)$-invariant subspace is $\Mult(\cH_\ell)$-invariant. In fact,
  for every Hilbert space $\HSE$, the only $\Mult(\cH_k)$-invariant subspace of
  $\cH_\ell\otimes\HSE$ is $\{0\}$. Indeed, if $\cM\neq\{0\}$ were one, the closed
  graph theorem would make $\varphi\mapsto M_\varphi|_\cM$ a bounded map from
  $H^\infty(\DD)$ to $\cB(\cM)$, so $\sup_n\|M_{z^n}|_\cM\|<\infty$; but if
  $F=\sum_jz^jF_j\in\cM$ and $F_j\neq0$, then
  $\|z^nF\|^2\ge\|F_j\|^2/a_{j+n}\to\infty$. Thus $(k,\ell)$ is a BLH pair for
  trivial reasons. Let $0<t_n<1$ converge to $1$, $r_n(z)=t_nz$ and
  $\Omega'=2\DD$. Since $\ell$ is circularly symmetric,
  Theorem~\ref{t:radial-dilation}, applied to $T=S$ for item~\ref{i:exh:cc} of
  Definition~\ref{kexhaust} and to arbitrary $T$, shows that $(r_n,\Omega')$ is a
  $(k,\ell)$-exhaustion with the dilation property. Thus
  items~\ref{i:regular:specT}, \ref{i:regular:poly-cvx}, \ref{i:regular:exh}
  and~\ref{i:regular:dilation} of Definition~\ref{def:regular} hold.
  Nevertheless, $S$ is a $\Mult(\cH_\ell)$-contraction that is not a $1/k$
  contraction, since $1/k(S,S^*)=I-SS^*$ and $\|S\|>1$. So the implication of
  Theorem~\ref{BLHISCOEXT} from the BLH property to every
  $\Mult(\cH_\ell)$-contraction being a $1/k$ contraction fails for this pair.
  Item~\ref{i:regular:include} itself fails: $\cH_\ell$ is shift-invariant but
  not $\Mult(\cH_k)$-invariant.
  \qed
\end{example}

\begin{remark}\rm
 \label{r:polycvx}
  We record here several facts about bounded polynomially convex  sets in $\CC^{\lcal{d}}$. 
  
  A set $P\subseteq \CC^{\lcal{d}}$ is an (open) \df{polynomial polyhedron} if there exist
  finitely many polynomials $p_1,\dots,p_n\in \CC[z_1,\dots,z_{\lcal{d}}]$ such that
\begin{equation} 
 \label{e:Ppolyhedra}
 P =\{z\in \CC^{\lcal{d}} : |p_1(z)|<1, \, \dots, \, |p_n(z)| < 1\}.
\end{equation}
Let 
\[
 \widetilde{P} =\{z\in \CC^{\lcal{d}}: |p_1(z)|\le 1, \,  \dots, \, |p_n(z)|\le 1\}.
\]
 Thus $\ol{P}\subseteq \widetilde{P}$ and $\widetilde{P}$ is polynomially convex. 
 
 It is a standard fact (\cite[Theorem 9, p.42]{Boas}, \cite{StoutPolyCvx}) that if  $C\subseteq \CC^{\lcal{d}}$ is compact and polynomially convex and $U\subseteq\CC^{\lcal{d}}$
 is an open set that contains $C,$ then there exist (open) polynomial polyhedra  $Q$ such that
\[
  C\subseteq Q \subseteq U.
\]
 The following consequence is used here.  Suppose $K$ is
 compact and polynomially convex and $\Omega^\prime$ is
 an open set containing $K.$  Choose $P$ as in
 equation~\eqref{e:Ppolyhedra}  such that $K\subseteq P \subseteq \Omega^\prime\cap B$ for an open ball $B\supseteq K$. For each $m\ge2$ the set
\[
  P_{(m)}=\{z\in\CC^{\lcal{d}}:|p_1(z)|<1-\tfrac1m,\,\dots,\,|p_n(z)|<1-\tfrac1m\}
\]
is open and bounded and the union of the $P_{(m)}$ is $P$. Hence, by compactness of $K$, there is an $m$ such that $K\subseteq P_{(m)}$. Now $\widetilde{P}_{(m)}\subseteq P_{(m+1)}\subseteq P\subseteq\Omega^\prime$. Hence $\Omega^{\prime\prime}=P_{(m)}$ is an open set, $K^\prime=\widetilde{P}_{(m)}$ is compact and polynomially convex and
\begin{equation}
    \label{e:Omegapp-in-between}
      K\subseteq \Omega^{\prime\prime} \subseteq K^\prime 
       \subseteq \Omega^\prime.
\end{equation}
 In particular, if $g$ is analytic in $\Omega^\prime,$
 then,  by Oka-Weil, it is the uniform on compacta limit of a sequence
 of polynomials on $\Omega^{\prime\prime}.$  
   
 If $K_j\subseteq \CC^{\lcal{d}}$ for $j=1,2$ are  compact and polynomially convex, then so is $K_1\times K_2.$  Hence if $K_j, K_j^\prime$ and
 $\Omega_j^\prime, \Omega_j^{\prime\prime}$ are as in
 equation~\eqref{e:Omegapp-in-between}, then
\[
 K_1\times K_2 \subseteq \Omega_1^{\prime\prime} \times \Omega_2^{\prime\prime} \subseteq K_1^\prime \times K_2^\prime \subseteq \Omega_1^\prime\times \Omega_2^\prime
\]
 and $K_1^\prime \times K_2^\prime$ is  compact  and polynomially convex. 
 In particular, if $g$ is analytic on $\Omega_1^{\prime}\times \Omega_2^{\prime},$ then there is a sequence of polynomials $p_n$ that
 converges uniformly to $g$ on compact subsets of $\Omega_1^{\prime\prime} \times \Omega_2^{\prime\prime}$. Indeed, apply Oka-Weil to $K_1^\prime\times K_2^\prime$ and $g$. When $K$ is compact and polynomially convex, so is $K_*:=\{\overline{z}:z\in K\}$ and our
 interest is in the case $K_1=K$ and $K_2=K_*.$ 
\qed
\end{remark}

We use the following spectral inclusion repeatedly. The \df{Taylor spectrum} of a tuple $T$
of bounded commuting operators is denoted by \df{$\specT(T)$}; see
\cite{mullerbook} and Curto's survey \cite{Curtosurveyref}.

\begin{lemma}\label{Taylorinclusion}
  Let $S$ and $T$ be $\lcal{d}$-tuples of bounded commuting operators on Hilbert spaces such
  that the map $p(S)\mapsto p(T)$, for $p\in\CC[z_1,\dots,z_{\lcal{d}}]$, is completely
  contractive. If $K\subseteq\CC^{\lcal{d}}$ is compact and polynomially convex and
  $\specT(S)\subseteq K$, then $\specT(T)\subseteq K$.
\end{lemma}

\begin{proof}
  Let $r(A)$ denote the spectral radius of an operator $A$. Applying contractivity to powers
  of a polynomial $p$ gives $r(p(T))\le r(p(S))$. If $\lambda\in\specT(T)$, the spectral
  mapping theorem therefore yields
  \[
    |p(\lambda)|\le r(p(T))\le r(p(S))
      =\max_{\mu\in\specT(S)}|p(\mu)|\le\max_{\mu\in K}|p(\mu)|
  \]
  for every polynomial $p$. Thus $\lambda$ belongs to the polynomially convex hull of $K$,
  which is $K$ itself.
\end{proof}

 \subsection{The hereditary calculus}
 \label{s:her-calc}

Let  \text{Her}$(\Omega)$ denote the set of complex-valued functions $h$ on $\Omega\times\Omega$  having the property that the map
\[(z, w)\mapsto h(z, \overline{w})\in\mathbb{C} \]
is holomorphic on $\Omega\times\tOmega,$ where $ \tOmega=\{\overline{z}:z\in\Omega\}.$   Equip \text{Her}$(\Omega)$ with pointwise addition and scalar multiplication and with the topology of uniform convergence on compact subsets of $\Omega\times\Omega.$ \text{Her}$(\Omega)$ is a complete metrizable locally convex topological vector space. 
If $h\in \Her(\Omega)$ is given by a power series 
\[
   h(z, w)=\sum_{\alpha, \beta\in\mathbb{N}^{\lcal{d}}}h_{\alpha, \beta}z^{\alpha}\overline{w}^{\beta}, \hspace{0.6 cm} z, w\in\Omega,
\]
with uniform and absolute convergence over compact subsets of $\Omega\times\Omega$ 
 and  $T=(T_1, \dots, T_{\lcal{d}})$ is a $\lcal{d}$-tuple of bounded commuting operators on a Hilbert space $H$ with Taylor spectrum in $\Omega$, then  the series 
 \[
   h(T, T^*)=\sum_{\alpha, \beta\in\mathbb{N}^{\lcal{d}}}h_{\alpha,\beta} T^{\alpha}T^{*\beta},
\] 
 converges in operator norm. 
 {More generally, one can define $h(T, T^*)\in\mathcal{B}(H)$ whenever $\specT(T)\subset\Omega$ and $h\in\textup{Her}(\Omega),$  yielding what is known as (Agler’s) Hereditary Functional
Calculus.} We refer the reader to  
\cite[p. 290-291]{Polynomialarazyenglis} and, for further details and proofs of basic (and less basic) facts concerning Taylor's functional calculus, to \cite[Chapter IV - Section 30]{mullerbook}.  Here
we record the properties needed in the sequel that will often be used without comment.

\begin{lemma} \label{l:Her-calc}
  Suppose $\Omega\subseteq \CC^{\lcal{d}}$ is an open set   and $T$ is a
  $\lcal{d}$-tuple of bounded commuting operators on a Hilbert space with Taylor spectrum in $\Omega$.
\begin{enumerate}[(i)]\itemsep=6pt
    \item  The mapping $h\mapsto h(T, T^*)$ on $\textup{Her}(\Omega)$ is complex linear.
    
\item If $\phi\in\textup{Hol}(\Omega)$, then $\phi(T, T^*)=\phi(T)$ in the sense of Taylor's functional calculus.

\item If $\psi=\bar{\phi}$ for $\phi\in\textup{Hol}(\Omega)$, then 
$\psi(T, T^*)=\phi(T)^*$ in the sense of Taylor's functional calculus.
    
\item If $g(z, w)=\overline{h(w, z)},$ then $g(T, T^*)=h(T, T^*)^*$.

    \item  If $\phi, \psi\in\textup{Hol}(\Omega)$ and
    $h\in\textup{Her}(\Omega)$, then 
    \[g(T, T^*)=\phi(T)h(T, T^*)\psi(T)^*,\]
 where  $g(z, w)=\phi(z)h(z, w)\overline{\psi(w)}\in \Her(\Omega).$
 
    \item  \label{i:her-calc:v}  If $h, \, h_n \in \textup{Her}(\Omega),$ and  
    $(h_n)$ converges to  $h$ in $\textup{Her}(\Omega)$, then $h_n(T, T^*)$ converges to $h(T, T^*)$ in operator norm. 

     \item  \label{i:for-coext}
       If polynomial sequences $(p_n)$ and $(q_n)$ converge uniformly on compact subsets
       of $\Omega$ to $f,g\in\hol(\Omega)$, respectively, then
       $p_n(z)\ol{q_n(w)}\to f(z)\ol{g(w)}$ in $\Her(\Omega)$.
\end{enumerate}
\end{lemma}

\begin{lemma}
 \label{l:her-calc+}
    Suppose $\ell$ is a reproducing  kernel  over a bounded domain $\Omega$ that 
   satisfies the conditions of Assumption~\ref{assume}. Assume the 
   domain $\Omega'$ contains the polynomially convex hull of $\specT (M^{\ell}_z)$.
    If  $h\in\Her(\Omega^\prime),$ then, for $a, b\in \Omega,$
\[
 \langle h(M^{\ell}_z, (M^\ell_z)^*) \ell_b,\ell_a \rangle = h(a,b)\ell(a, b).
\]
\end{lemma} 

\begin{proof}
By  Remark~\ref{r:polycvx},  there exist 
an open set $\Omega''$ with $\specT(M^\ell_z)\subseteq \Omega''\subseteq \overline{\Omega''}\subseteq \Omega'$ and  a sequence of hereditary polynomials $\{p_n\}$ such that $p_n\to h$ uniformly over compact subsets of $\Omega''\times \Omega''$. In particular, $p_n\to h$ in $\Her(\Omega'').$ Since
\begin{equation*} 
 \langle p_n(M^{\ell}_z, (M^\ell_z)^*) \ell_b,\ell_a \rangle = p_n(a,b)\ell(a, b)
\end{equation*}
for all $a, b\in\Omega$ and $n\ge 1,$ and $ p_n(M^{\ell}_z, (M^\ell_z)^*)\to h(M^{\ell}_z, (M^\ell_z)^*)  $ in operator norm, the desired equality follows in the limit on $n.$
\end{proof}

The following composition rule is used with polynomial maps $r$, such as the
maps of an exhaustion.

\begin{lemma}\label{l:her-compose}
  Let $G,U\subseteq\CC^{\lcal{d}}$ be open sets and let $r$ be a
  $\lcal{d}$-tuple of polynomials with $r(U)\subseteq G$. Let $h\in\Her(G)$ be the
  limit in $\Her(G)$ of a sequence of hereditary polynomials, and set
  $h_r(z,w)=h(r(z),r(w))$. If $T$ is a $\lcal{d}$-tuple of commuting operators
  with $\specT(T)\subseteq U$, then $\specT(r(T))\subseteq G$ and
  \[
    h\big(r(T),r(T)^*\big)=h_r(T,T^*).
  \]
\end{lemma}

\begin{proof}
  By the spectral mapping theorem, $\specT(r(T))=r(\specT(T))\subseteq G$, and
  $h_r\in\Her(U)$. For a hereditary polynomial $q$, expanding $q_r$ shows that
  $q(r(T),r(T)^*)=q_r(T,T^*)$. If $q_m\to h$ in $\Her(G)$, then
  $(q_m)_r\to h_r$ in $\Her(U)$, because $(z,w)\mapsto(r(z),r(w))$ maps compact
  subsets of $U\times U$ into compact subsets of $G\times G$. The identity
  follows from item~\ref{i:her-calc:v} of Lemma~\ref{l:Her-calc}, applied on
  both sides.
\end{proof}

   This subsection concludes with the definition of a $1/k$ contraction.
   Given a regular pair of kernels $(k,\ell)$ over a domain $\Omega\subseteq \CC^{\lcal{d}},$
   by Definition~\ref{kexhaust} the function $1/k$ extends to an element of $\Her(\Omega')$ for an open set $\Omega'\supseteq\ol{\Omega}.$
   Hence, if $T=(T_1,\dots,T_{\lcal{d}})$ is a tuple of operators on Hilbert space with Taylor spectrum in $\overline{\Omega}$, then $1/k(T,T^*)$ is defined
   via the hereditary functional calculus.

\begin{definition}
    \label{1kdef}
       Given an analytic  kernel  $k$ over a domain $\Omega\subseteq \CC^{\lcal{d}}$ such that $\frac{1}{k}$ 
       extends to an element of $\Her(\Omega')$ for an open set $\Omega'\supseteq\ol{\Omega},$ 
     a tuple $T=(T_1,\dots,T_{\lcal{d}})$ of bounded commuting operators on a Hilbert space with Taylor spectrum in $\overline{\Omega}$
     is a $1/k$ contraction if  $1/k(T,T^*)\succeq 0$. 
\qed\end{definition}

\subsection{Consequences of regularity}\label{subsecconseq}
 This subsection provides some perspective on the spectral condition of item~\ref{i:regular:specT} of 
 Definition~\ref{def:regular}.
 Suppose $\lambda$ is an analytic kernel over a domain $\Omega\subseteq \CC^{\lcal{d}}$
 that satisfies the conditions of Assumption~\ref{assume}. 
 By Remark~\ref{r:assume-Omega-bounded}, $\Omega\subseteq \specT(M_z^\lambda).$
 Density of polynomials in $\cH_\lambda$  (item~\ref{i:assume:ii} of Assumption~\ref{assume}) implies the existence
 of a unique maximal (maximum) domain for $\lambda$
 and hence there is no harm in assuming $\Omega$ is
 already maximal. See Definition~\ref{maximaldef} and Proposition~\ref{r:odd}.
 Moreover, if $\lambda$ satisfies the spectral condition of item~\ref{i:regular:specT} of Definition~\ref{def:regular},
 namely $\specT(M_z^\lambda) = \ol{\Omega},$
 and, in addition, $\Omega$ is the interior of its closure, 
 then in fact $\Omega$ is the maximum domain for $\lambda.$ 
 See Proposition~\ref{p:spec-cond}.  A standard example shows that it is not always
 the case that $\specT(M_z^\lambda) = \ol{\Omega}$ for the unique maximal domain $\Omega$ of $\lambda.$
 The reader satisfied with this summary can safely skip to  Section~\ref{s:moreBLH}. 

\subsubsection{Maximal domains}

{We adopt the following definition of maximal domain.} 

 \begin{definition}\label{maximaldef}
 Assume $\genl$ is an analytic reproducing kernel over a bounded domain $\Omega\subseteq \CC^{\lcal{d}}$. A domain $\Gamma\supset \Omega$ is said to be a \textit{maximal domain} of $\genl$ if:
 \begin{itemize}
     \item[(i)] $\genl$ extends to an analytic kernel on $\Gamma\times\Gamma$ and

     \item[(ii)] there does not exist a domain $\Gamma'\supset\Gamma$ properly containing $\Gamma,$ such that $\genl$ extends to an analytic kernel on $\Gamma'\times\Gamma'$.
 \end{itemize}
\end{definition}

\begin{proposition}  \label{r:odd}
  If  $\lambda$ is an analytic kernel over the  domain $\Omega_\lambda\subseteq \CC^{\lcal{d}}$ that 
   satisfies the conditions of Assumption~\ref{assume}, then there is a unique maximal domain of $\lambda$.
\end{proposition}

The following lemma is standard (see e.g. \cite[Section 5.4]{PaulsenRagh}); we omit the proof.

\begin{lemma}
\label{l:aspire}
  Suppose $h$ and $g$ are analytic reproducing kernels over  connected domains $\Omega_h\subseteq \Omega_g \subseteq \CC^{\lcal{d}}.$
  If the restriction of $g$ to $\Omega_h$ is $h,$ then   $\cH_h=\cH_g$
  and $\Mult(\cH_h)=\Mult(\cH_g).$  More precisely, the map sending a function
  $\psif:\Omega_g\to \CC$ to $\psif\vert_{\Omega_h}$ induces complete isometric isomorphisms
  between $\cH_g$ and $\cH_h$ and $\Mult(\cH_g)$ and $\Mult(\cH_h).$
\end{lemma}

   \begin{proof}[Proof of Proposition~\ref{r:odd}]
    Assume $f,g$ are extensions of $\lambda$ to domains $\Omega_f$ and $\Omega_g$, respectively.
        Since polynomials are dense in $\cH_\lambda,$ there is an orthonormal basis $\{u_n:n\in\NN\}$
    of $\cH_\lambda$ consisting of polynomials. Thus,
 \begin{equation*}
     \lambda(z,w) = \sum u_n(z)\ol{u_n(w)}
 \end{equation*}
  with absolute and uniform convergence on compact subsets of $\Omega_\lambda \times\Omega_\lambda.$ Next, consider the canonical restriction mappings
  \[\iota_f: \mathcal{H}_{f}\to \mathcal{H}_\lambda, \qquad \iota_g: \mathcal{H}_{g}\to \mathcal{H}_\lambda,\]
  obtained by restriction to $\Omega_\lambda.$ By Lemma \ref{l:aspire}, both $\iota_f$ and $\iota_g$ are isometric isomorphisms. In particular, $\{\iota_f^{-1}(u_n)\}$ and $\{\iota^{-1}_g(u_n)\}$ are orthonormal bases for $\mathcal{H}_f$ and $\mathcal{H}_g,$ respectively, which yields 
  \begin{equation}
     \label{f,g:series}f(z, w)=\sum (\iota^{-1}_fu_n)(z) \overline{(\iota^{-1}_fu_n)(w)}, \quad g(z, w)=\sum (\iota^{-1}_gu_n)(z) \overline{(\iota^{-1}_gu_n)(w)},
     \end{equation}
  with absolute and uniform convergence on compact subsets of the respective domains. Now, let $v_n$ denote the (unique) function 
   on $\Gamma:=\Omega_f\cup\Omega_g$ that agrees with $u_n$ on $\Omega_\lambda,$ for any $n$ (note that, since each $u_n$ is a polynomial, $v_n$ is that same polynomial, but viewed as a function on a larger set). Writing $v(z)=(v_n(z))$, define $H:\Gamma\times\Gamma\to\mathbb{C}$ by 
  \[H(z,w)=\langle v(z), v(w) \rangle_{\ell^2}.\]
  Since $\iota_f^{-1}(u_n)=v_n|_{\Omega_f}$ and $\iota_g^{-1}(u_n)=v_n|_{\Omega_g}$, $\ell^2$-summability in \eqref{f,g:series} implies $\ell^2$-summability for $v(z).$ In particular, $v(z)$ is an entry-wise holomorphic and locally bounded function; thus, it is an $\ell^2$-valued holomorphic function. Hence,  $H$ is a holomorphic kernel over $\Gamma\times\Gamma$  that extends both $f$ and $g.$ In particular, $\lambda$ does not have two distinct maximal extensions.
 That a (necessarily unique) maximal domain exists follows from a standard application of Zorn's Lemma.
\end{proof}

\subsubsection{The maximal domain and the spectral condition}

\begin{proposition}
    \label{p:oops-extended}
      Suppose $h$ is an analytic reproducing kernel over a domain $\Omega_h\subseteq \CC^{\lcal{d}}$ 
      that satisfies Assumption~\ref{assume}. If $g$ is an analytic reproducing kernel over a domain $\Omega_g$
      that contains $\Omega_h$ and $g\mid_{\Omega_h\times \Omega_h}=h,$ then $\specT(M^h_z) \supseteq \Omega_g.$
\end{proposition}

\begin{proof}
    Given a multiplier $\varphi:\Omega_g\to\CC$, let $M^g_\varphi$ denote the corresponding
    multiplication operator $M^g_\varphi:\cH_g\to\cH_g.$ Define $M^h_\varphi$ similarly. Lemma~\ref{l:aspire} says that the restriction map $U: \psif\mapsto \psif\vert_{\Omega_h}$
  is an isometric isomorphism between $\cH_g$ and $\cH_h$. The operator $U$ clearly satisfies
     $UM^g_\varphi= M^h_{\varphi\mid_{\Omega_h}} U.$  Thus, for $w\in\Omega_g,$
\[
   \ol{\varphi(w)} U g(\cdot,w) =   U (M^g_\varphi)^*  g(\cdot,w)=  (M^h_{\varphi\mid_{\Omega_h}})^* U g(\cdot,w).
\]
  In particular, with $\varphi(z)=z_j,$  it follows that $U g(\cdot,w)$ is a joint eigenvector of
  $(M^h_z)^*$ with corresponding eigenvalue $\ol{w}.$ Hence, $w\in \specT(M^h_z)$, establishing the result.
\end{proof}

The following proposition says that, under the spectral
assumption $\ol{\Omega}=\specT(M^{\genl}_z)$ of item~\ref{i:regular:specT} of Definition~\ref{def:regular},
it is harmless to assume that $\Omega$ is a maximal (thus the
maximum) domain of $\ell.$ For convenience, set $S=M_z^{\genl}.$

 \begin{proposition}
 \label{p:spec-cond}
   Suppose  $\genl$ is an analytic kernel over the domain $\Omega\subseteq \CC^{\lcal{d}}$ that  
   satisfies the conditions of Assumption~\ref{assume} and let $\Gamma$ denote 
   a maximal   domain of $\genl.$  If $\ol{\Omega} = \specT(S),$ then  $\ol{\Gamma}=\specT(S).$

   In particular, if $\Omega =\specT(S)^\circ,$
   then $\Omega=\Gamma.$  
 \end{proposition}

 \begin{proof}[Proof of Proposition \ref{p:spec-cond}]
   By the definition of a maximal domain for a kernel $\genl$
   over a domain $\Omega,$ we have $\Omega\subseteq \Gamma.$
   Let $S=(S_1,\dots,S_\lcal{d})$.
   By Assumption~\ref{assume} these operators are (defined and)  bounded.  
    Since $\genl$ extends to a reproducing kernel over $\Gamma,$ it follows from Proposition~\ref{p:oops-extended} that for each point $w\in\Gamma$, $\ol{w}$ is
    a joint eigenvalue for $S^*.$ Thus, $\Gamma\subseteq \specT(S),$  and so 
\[
  \specT(S) = \ol{\Omega} \subset \ol{\Gamma} \subseteq \specT(S) . 
\]
Finally, if $\Omega =\specT(S)^\circ$, then 
\[
  \specT(S)^\circ=\Omega\subset\Gamma\subset \specT(S)^{\circ}. \qedhere
\]
 \end{proof}

It should, however, be noted that maximality of the underlying domain (Definition~\ref{maximaldef}) does not imply the spectral assumption (item~\ref{i:regular:specT} of Definition~\ref{def:regular}),
a result implicit in \cite[Theorem~8]{Shields} and the surrounding discussion.
We provide a concrete example below.

\begin{example}\rm
 \label{eg:spec-not-omega}
  Let $B=\bigcup_{j\ge1}\{j^3,j^3+1,\dots,j^3+j-1\}$, let $c(n)$ be the number of
  integers $m<n$ with $m\in B$, and let
  \[
    k(z,w)=\sum_{n\ge0}4^{-c(n)}(z\ol{w})^n,\qquad z,w\in\DD.
  \]
  The monomials are orthogonal in $\cH_k$ with $\|z^n\|^2=4^{c(n)}$, so $M_z$ is
  the weighted shift whose weight $\|z^{n+1}\|/\|z^n\|$ is $2$ for $n\in B$ and
  $1$ otherwise. Thus $\|M_z\|=2$ and $k$ satisfies Assumption~\ref{assume}. As
  $B$ contains $j$ consecutive integers for every $j$, $\|M_z^j\|=2^j$, so the
  spectral radius of $M_z$ is $2$; since the spectrum of an injective unilateral
  weighted shift is a closed disk centred at the origin \cite{Shields},
  $\sigma(M_z)=2\ol{\DD}$. On the other hand, $c(n)=O(n^{2/3})$, so
  $f(x)=\sum_n4^{-c(n)}x^n$ has radius of convergence $1$, and, as its
  coefficients are positive, Pringsheim's theorem shows that $f$ has a singular
  point at $x=1$. Hence $\DD$ is the maximal domain of $k$: if $k$ extended to a
  kernel on a domain $\Gamma$ properly containing $\DD$, then $\Gamma$ would
  contain a disk around some $\zeta\in\T$, and $\lambda\mapsto k(\lambda\zeta,
  \ol{\lambda}\zeta)$, which equals $f(\lambda^2)$ on $\DD$, would extend
  holomorphically to a neighbourhood of $\lambda=1$.
  \qed
\end{example}

\section{A closer look at the BLH condition} 
\label{s:moreBLH}

The notion of a BLH pair extends to pairs of kernels rather more general than regular pairs. We begin with a standard reformulation of the BLH property that is used repeatedly in the sequel; see Proposition~\ref{l_m/k}. Next, we provide some natural sufficient conditions (weaker than regularity) for the existence of a relative 
abundance of  $\Mult(\cH_\genk)$-invariant subspaces of $\cH_\genl\otimes \HSE$; see Proposition~\ref{allequiv} and Corollary~\ref{c:allequiv}. Without such conditions, the property of being a BLH pair might not be very interesting; see Example~\ref{eg:BLH-vacuous}. Next, we introduce the notions of Agler decomposition, weak Agler decomposition and strong Agler decomposition  for pairs of kernels, Definitions~\ref{Aglerdecc}, \ref{weakAgler} and \ref{strongAgler} respectively. We show that the existence of a weak Agler decomposition for $(k,\ell)$ implies the BLH property, provided every $\Mult(\cH_k)$-invariant subspace is $\Mult(\cH_\ell)$-invariant; see Proposition~\ref{gensuffBLH}. This result, in particular, yields the implication item ~\ref{i:im:ad} implies item~\ref{i:im:blh} in Theorem \ref{intromain}. We thus obtain a concrete general condition that is sufficient for the BLH property. A main point of Theorem \ref{intromain} is that this general condition is also necessary, at least under the assumption of strong regularity. Finally, we are naturally led to the connection between the BLH and complete Pick properties, which, though 
not needed in the sequel, is discussed in Subsections~\ref{CPconnectsubsec} and ~\ref{CPsubsec}. Here we draw the reader's attention to the discussion
of Example~\ref{bidiskfamilyex} at the outset of Subsection~\ref{CPsubsec}. This example
illustrates the large gap between CP and BLH pairs.

  Let $\genk,\genl,$ reproducing kernels over a set $X,$ be given.
  Adopting standard notation, let $\genk_b(a) = \genk(a,b)$ for $a,b\in X.$ 
   Given Hilbert spaces $\HSE$ and $\HSF,$
 a function $\Phi:X\to \cB(\HSF,\HSE)$ is a \df{multiplier} from $\cH_\genk \otimes \HSF$ to  $\cH_\genl\otimes \HSE$
 if for each $f\in \cH_\genk\otimes \HSF$ the function $\Phi f: X\to \HSE$ defined by $\Phi f(x)=\Phi(x) f(x)$ for $x\in X$ is in $\cH_\genl\otimes \HSE.$
 In this case, $\Phi$ induces a bounded operator  $M_\Phi:\cH_\genk \otimes \HSF \to \cH_\genl\otimes \HSE$  whose norm is
 the smallest $C\ge 0$ such that
 \[
  X \times X \ni (a,b) \mapsto C^2\genl(a,b)I_{\HSE} -\Phi(a)\Phi(b)^* \, \genk(a,b)
 \]
  is positive semidefinite. Indeed, this fact is proved by observing that, if $b\in X,$ $\Phi\in  \Mult\big(\mathcal{H}_{\genk}\otimes \HSF,\mathcal{H}_{\genl}\otimes \HSE\big)$ and $h\in \HSE$, then \[M_\Phi^* \genl_b\otimes h= \genk_b\otimes \Phi(b)^*h.\]
 In particular, $M_\Phi$ determines $\Phi(b)$ when $\genk(b,b)>0$; when $\genk(b,b)=0$, every
 function in $\cH_\genk\otimes\HSF$ vanishes at $b$, and changing $\Phi(b)$ does not change
 $M_\Phi$. Thus, given  $n$ points $a_1, \dots, a_n\in X$ and $n$ matrices $W_1, \dots, W_n\in \mathbb{C}^{N\times N}$,  a necessary condition for the existence of a multiplier $\Phi \in\Mult\big(\mathcal{H}_{\genk}\otimes\mathbb{C}^N, \mathcal{H}_{\genl}\otimes \mathbb{C}^N\big)$ of norm at most one that satisfies $\Phi(a_i)=W_i,$ for all $1\le i\le n,$ is  
\begin{equation} \label{multinterpol} 
 \begin{bmatrix}\genl(a_i, a_j)I_{\mathbb{C}^N}-W_iW^*_j\genk(a_i, a_j) \end{bmatrix}_{i,j=1}^n\succeq 0.
\end{equation}

The following reformulation of the BLH property is implicitly contained in \cite{Beurlingfactor}. Given a subspace $\cM$ of $\cH_\genl\otimes\HSE$, we denote the reproducing kernel of $\cM$ by \df{$\genl_\cM$}.  
Thus $\genl_{\cM}: X\times X \to \cB(\HSE)$ satisfies
\[
 \langle \genl_{\cM}(a,b) e,f \rangle 
  =  \langle P_{\cM} (\genl_b \otimes e), \genl_a\otimes f
   \rangle,
\]
 where $P_{\cM}$ is the projection onto $\cM$ and
 $e,f\in\HSE.$

\begin{proposition}\label{l_m/k}
 Suppose $(\genk,\genl)$ is a pair of reproducing kernels
over a set $X,$  $\HSE$ is a Hilbert space and 
$\cM\subset \cH_{\genl}\otimes \HSE$ is a subspace.  There is a Hilbert space $\HSF$ and a partially
isometric multiplier $\Phi\in \Mult(\cH_\genk\otimes \HSF,\cH_\genl \otimes \HSE)$ whose range is $\cM$ if and only if there is 
a positive kernel $g_{\cM}$ such that $\genl_{\cM} = g_{\cM} \genk.$ In particular,
if  $(\genk, \genl)$ is a pair of reproducing kernels, then  $(\genk, \genl)$ is BLH if and only if,  for every Hilbert space $\mathcal{E}$ and every 
$\textup{Mult}(\mathcal{H}_{\genk})$-invariant   subspace  $\mathcal{M}\subset \mathcal{H}_{\genl}\otimes\mathcal{E}$,
 there is a positive kernel $g_\cM$ such that 
\[
 \genl_{\cM} = g_{\cM} \genk.
\]
\end{proposition}

\begin{proof}[Proof sketch]
The proof   of the first part of the proposition rests on combining the following basic facts. 
\begin{itemize}
    \item[(i)] A function $L: X\times X \to \mathcal{B}(\mathcal{E})$ is a positive semidefinite kernel if and only if there exists a Hilbert space $\mathcal{F}$ and a function $\Phi: X\to\mathcal{B}(\mathcal{F}, \mathcal{E})$ such that $L(a, b)=\Phi(a)\Phi(b)^*$;
    \item[(ii)] A function $\Phi:X\to\mathcal{B}(\mathcal{F}, \mathcal{E})$ yields a coisometric multiplication operator $M_{\Phi}: \mathcal{H}_{\genk}\otimes\mathcal{F}\to \mathcal{M}$ if and only if the self-adjoint $\mathcal{B}(\mathcal{E})$-valued function defined on $X\times X$ as
    \begin{align*}
 \genl_\cM(a, b)&-\Phi(a)\big(\genk(a, b) I_{\mathcal{F}}\big)\Phi(b)^*\\
        =\ & \genl_\cM(a, b)-\kappa(a, b)\Phi(a)\Phi(b)^*, 
    \end{align*}
is identically equal to zero (see e.g. \cite[Lemma 2.1(b)]{Beurlingfactor}); 
\item[(iii)] $M_{\Phi}: \mathcal{H}_{\genk}\otimes\mathcal{F}\to \mathcal{H}_{\genl}\otimes\mathcal{E}$ is a partial isometry with $M_{\Phi}(\mathcal{H}_{\genk}\otimes\mathcal{F})=\mathcal{M}$ if and only if $M_{\Phi}: \mathcal{H}_{\genk}\otimes\mathcal{F}\to \mathcal{M}$ is a coisometry.   
\end{itemize}

  The second statement follows immediately from the first and the definition of a BLH pair.   
\end{proof}

\begin{corollary} \label{ldividesk}
If $(\genk, \genl)$ is a BLH pair and $\cH_\genl$
is $\Mult(\cH_\genk)$ invariant, then there is a positive kernel $g$ such that $\genl=g\genk.$
In particular, $\textup{Mult}(\mathcal{H}_{\genk})$ includes completely contractively in $\textup{Mult}(\mathcal{H}_{\genl}).$ 
\end{corollary}

The proof of the last statement of Corollary~\ref{ldividesk} is a consequence of the following simple lemma that we record here for future use.

\begin{lemma}
  \label{l:cert-cont-include}
    Given  reproducing kernels $\genk$ and $\genl,$ if there is a (positive) kernel $g$ such that
    $\genl= g \genk,$ then $\Mult(\cH_\genk)\subseteq\Mult(\cH_\genl)$ completely contractively.
\end{lemma}

\subsection{Reduction to shift-invariance}

The line of argument in the next two results follows that  in \cite[Proposition 2.5]{Beurlingfactor}.

\begin{proposition} 
  \label{allequiv} 
   Let $(\genk, \genl)$ be a pair of analytic reproducing kernels satisfying Assumption~\ref{assume}. Suppose $\textup{Mult}(\mathcal{H}_{\genk})\subset \textup{Mult}(\mathcal{H}_{\genl})$ 
   and the polynomials are weak-$*$ dense in $\Mult(\cH_\genl)$.   Given a Hilbert space $\mathcal{E}$ and a  subspace $\mathcal{M}\subset\mathcal{H}_{\genl}\otimes \mathcal{E}$, the following are equivalent:
\begin{enumerate}[(i)]\itemsep=6pt
    \item \label{i:all:i}
       $\mathcal{M}$ is shift-invariant; 
    \item \label{i:all:ii}
        $\mathcal{M}$ is  $\textup{Mult}(\mathcal{H}_{\genl})$-invariant;
    \item \label{i:all:iii}
        $\mathcal{M}$ is  $\textup{Mult}(\mathcal{H}_{\genk})$-invariant.
\end{enumerate}
\end{proposition}

\begin{proof}
   Suppose $\cM\subseteq \cH_\genl\otimes \HSE$
 is shift-invariant. Let $\varphi\in\Mult(\cH_\genl)$ be given. By hypothesis,
 there exists a net of polynomials $(p_\alpha)$ such that the operators
 $T_\alpha=M_{p_\alpha}$ on $\cH_\genl$ converge weak-$*$ to $M_\varphi$.
 The amplification $A\mapsto A\otimes I_\HSE$ is weak-$*$ continuous, so
 $T_\alpha\otimes I_\HSE\to M_\varphi\otimes I_\HSE$ weak-$*$, hence in the
 weak operator topology. Thus, as $\cM$ is invariant for each
 $T_\alpha\otimes I_\HSE$, for $m\in\cM$ and $h\in\cM^\bot$,
\[
 0=\lim_\alpha\langle (T_\alpha\otimes I_\HSE)m,h\rangle
  =\langle (M_\varphi\otimes I_\HSE)m,h\rangle,
\]
 and thus $(M_\varphi\otimes I_\HSE)m\in\cM$ as desired.
 Hence item~\ref{i:all:i} implies item~\ref{i:all:ii}.
 Since the implications item~\ref{i:all:ii} implies item~\ref{i:all:iii} and
 item~\ref{i:all:iii} implies item~\ref{i:all:i} are immediate, the proof is
 complete.
\end{proof}

A kernel $\genk$ on a domain $\Omega\subseteq\CC^{\lcal{d}}$ is
\df{circularly symmetric} if $\tau\Omega\subseteq\Omega$ and
$\genk(\tau z,\tau w)=\genk(z,w)$ for every $\tau\in\T$ and $z,w\in\Omega$. A pair
$(\genk,\genl)$ of kernels on $\Omega$ is a \df{circularly symmetric pair} if both
$\genk$ and $\genl$ are circularly symmetric.

\begin{corollary}
    \label{c:allequiv}
  Let $(\genk,\genl)$ be a pair of analytic kernels on a domain
  $\Omega\subseteq\CC^{\lcal{d}}$ satisfying Assumption~\ref{assume}, such that
  $\genl$ is circularly symmetric. If $\Mult(\cH_\genk)\subseteq\Mult(\cH_\genl)$,
  then the conclusion of Proposition~\ref{allequiv} holds.
\end{corollary}

\begin{proof}
  By Proposition~\ref{allequiv}, it suffices to show that the polynomials are
  weak-$*$ dense in $\Mult(\cH_\genl)$; we follow the proof of
  \cite[Proposition~2.5]{Beurlingfactor}. For $\theta\in\mathbb{R}$, the operator
  $U_\theta f=f(e^{i\theta}\,\cdot)$ is unitary on $\cH_\genl$ (it is the adjoint
  of the unitary operator determined by $\genl_w\mapsto\genl_{e^{i\theta}w}$),
  $\theta\mapsto U_\theta$ is strongly continuous, and
  $U_\theta M_\varphi U_\theta^*=M_{\varphi(e^{i\theta}\cdot)}$ for
  $\varphi\in\Mult(\cH_\genl)$. As $U_\theta$ multiplies homogeneous polynomials of
  degree $n$ by $e^{in\theta}$ and the polynomials are dense
  (Assumption~\ref{assume}), $\cH_\genl$ is the orthogonal sum of the
  finite-dimensional spaces of homogeneous polynomials of each degree; hence the
  Fej\'er means
  $\sigma_N(\varphi)=\int_0^{2\pi}F_N(\theta)\,\varphi(e^{i\theta}\,\cdot)\,\frac{d\theta}{2\pi}$,
  with $F_N$ the Fej\'er kernel, are polynomials. They satisfy
  $\|M_{\sigma_N(\varphi)}\|\le\|M_\varphi\|$ and converge to $M_\varphi$ in the
  strong operator topology; being norm-bounded, they converge weak-$*$.
\end{proof}

\subsection{The Agler wedge}
 A Positivstellensatz is an algebraic certificate of positivity. Typically, such a certificate takes the form of membership in a weighted sums-of-squares cone. In this subsection, we introduce the \df{Agler wedge} $\mathcal{AG}(\ell)$, a cone tailored to the study of the BLH property for a pair of kernels $(k,\ell)$. Its construction is inspired by Agler's approach to Pick interpolation pioneered in \cite{Hellinger}.

\begin{definition} \label{Aglerwedge}
 Let $\ell$ be a reproducing kernel over a bounded domain $\Omega$ that satisfies Assumption \ref{assume}.  The \df{Agler wedge} 
 \df{$\mathcal{AG}(\ell)$} of $\ell$ is the closure in $\text{Her}(\Omega)$ of the set of all
 (scalar-valued) hereditary functions of the form
\begin{equation}\label{aglerfact}
    Q(z)\big[I_{m\times m}-P(z)P(w)^* \big]Q(w)^*, 
\end{equation}
where  $Q:\Omega\to\mathcal{B}(\mathbb{C}^{m}, \mathbb{C})$ is holomorphic, $P\in\textup{Mult}(\mathcal{H}_{\ell}\otimes\CC^{m})$ is a contractive polynomial multiplier and $m\ge 1$.
\qed
\end{definition}

\begin{remark}\rm
 \label{r:AG-only-on-mults}
    The Agler wedge depends only on the matrix unit balls of $\Mult(\cH_\ell)$, that is,
    on the unit balls of $\Mult(\cH_\ell\otimes\mathbb{C}^m)$ for all $m\ge1$.
\qed\end{remark}
   \begin{remark} Note that $Q(w)$ is naturally identified with an element of the dual of $\CC^m.$ Hence,
     $Q(w)^*$ is naturally understood as a vector in $\CC^m.$  From this perspective, it would
     be natural to replace $Q^*$ with $R$ and write, in equation~\eqref{aglerfact},
\[
  R(z)^*\big[I_{m\times m}-P(z)P(w)^* \big]R(w).
\]
 However, we prefer to keep the adjoints on the right in the spirit of the hereditary functional calculus.  
 \qed\end{remark}

\begin{lemma}\label{closed-under-sums}
The Agler wedge $\mathcal{AG}(\ell)$  is a cone; that is, it  is closed under addition
and multiplication by nonnegative scalars.
\end{lemma}

\begin{proof} 
 It suffices to show the dense set that determines the Agler wedge is a cone. To this end, let 
 contractive polynomial multipliers $P_1, P_2,$ 
  analytic functions $Q_1$ and $Q_2$ and $a\ge 0$
 be given. The  (matrix-valued) polynomial $P$  
\[
  P(z)=\begin{bmatrix}
  P_1(z) & 0 \\ 
  0    & P_2(z)
\end{bmatrix}
\]
  is a contractive multiplier and 
  setting
   $Q=\begin{bmatrix} \sqrt{a}\, Q_1 & Q_2 \end{bmatrix},$
gives
\[
  a\, Q_1(I-P_1P_1^*)Q_1^*+Q_2(I-P_2P_2^*)Q_2^*=Q(I-PP^*)Q^*. \qedhere
\]
\end{proof}

\begin{remark}\label{Dullcone}
If $Q$ and $P$ are as in \eqref{aglerfact} and $0<r<1$, then
\[
  Q(I-r^2PP^*)Q^*=r^2\,Q(I-PP^*)Q^*+(1-r^2)\,QQ^*,
\]
which converges to $Q(I-PP^*)Q^*$ in $\Her(\Omega)$ as $r\to1$. Thus, the
definition of $\mathcal{AG}(\ell)$ remains unchanged if we only consider
contractive polynomial multipliers $P$ of $\mathcal{H}_\ell\otimes\mathbb{C}^m$
satisfying $\|P(z)\|<1$ for all $z\in\Omega$.
\qed
\end{remark}

\begin{definition} \label{Aglerdecc}
  A pair of kernels $(k, \ell)$ \df{admits an Agler decomposition} if $k$ is nonvanishing and $1/k\in\mathcal{AG}(\ell)$. 
  \qed
\end{definition} 

 The notion of an Agler decomposition naturally extends to pairs of kernels that do not satisfy any holomorphy assumptions. 
 
\begin{definition}\label{weakAgler}
  A pair of reproducing kernels $(k, \ell)$ over a set $X$
  \df{admits a weak Agler decomposition} if $k$ is nonvanishing and there exist a sequence of positive integers $\{m_n\}$, 
 functions $Q_n:X\to\mathcal{B}(\mathbb{C}^{m_n}, \CC)$ and contractive multipliers $\Phi_n\in\textup{Mult}(\mathcal{H}_{\ell}\otimes\CC^{m_n})$ such that       
    \begin{equation}    \label{again1/k=}  
         Q_n(a)\big[I-\Phi_n(a)\Phi_n(b)^*\big]Q_n(b)^* \longrightarrow \cfrac{1}{k(a,b)}
    \end{equation}
{pointwise  in $X\times X$}.
\qed
\end{definition}

\begin{remark}\label{r:weak-include}
Because each $\Phi_n$ is a contractive multiplier of $\cH_\ell,$ the functions   $(I-\Phi_n(a)\Phi_n(b)^*) \ell(a,b)$ are psd.
Thus \eqref{again1/k=}  yields $\ell/k\succeq 0$ and thus Lemma~\ref{l:cert-cont-include} gives  a completely contractive inclusion for the multiplier algebras without any additional assumptions on the pair $(k, \ell).$  
\qed\end{remark}

\begin{remark}
If $(k, \ell)$ admits an Agler decomposition, then it admits a weak Agler decomposition.
\qed
\end{remark}

A strong Agler decomposition replaces the limiting representations in the definitions of
Agler and weak Agler decompositions by a single exact identity. The coefficient spaces may
then be arbitrary separable Hilbert spaces.

\begin{definition}\label{strongAgler}
A pair of kernels $(k, \ell)$  over a set $X$ \textit{admits a strong Agler decomposition} if $k$ is nonvanishing and
 there exist  separable Hilbert spaces $\mathcal{G}, \HSF,$ 
    a function $Q:X\to
\mathcal{B}(\mathcal{G}, \mathbb{C})$ and a contractive multiplier $\Phi\in \Mult(\cH_\ell\otimes \HSF, \cH_\ell\otimes \mathcal{G})$ such that
\begin{equation} \label{strongAglerb}
     \frac{1}{k(z,w)} = Q(z) \left (I-\Phi(z)\Phi(w)^* \right ) Q(w)^*
\end{equation}
over $X\times X.$
\qed
\end{definition}

\begin{remark}\label{r:strong-weak}
If $(k, \ell)$ admits a strong Agler decomposition, then it admits a weak Agler decomposition:
if $P_n$ and $R_n$ are the orthogonal projections onto the spans of the first $n$ vectors of
orthonormal bases of $\mathcal{G}$ and $\HSF$, then $QP_n$ and $P_n\Phi R_n$, regarded as
$\CC^{1\times n}$- and $\CC^{n\times n}$-valued functions (padded with zeros if needed),
satisfy Definition~\ref{weakAgler}, since $P_n\to I$ and $R_n\to I$ strongly and
$M_{P_n\Phi R_n}=(I\otimes P_n)M_\Phi(I\otimes R_n)$ is contractive. By Corollary~\ref{strongAglertoAgler}, under Assumption~\ref{assume} every weak Agler decomposition yields a strong one with $Q$ holomorphic.
\qed\end{remark}

\subsection{A general sufficient condition} 
 This subsection is devoted to the statement and proof of
 the following sufficient condition for a pair $(k,\ell)$
 to have the BLH property.

\begin{proposition}\label{gensuffBLH}
 If  $(\genk,\genl)$ admits a weak Agler decomposition,
 $\HSE$ is a Hilbert space and $\cM\subseteq \cH_\genl\otimes\HSE$
 is a $\Mult(\cH_\genl)$-invariant subspace, then
\[
 \frac{\genl_\cM}{\genk} \succeq 0.
\]
 In particular, if $(\genk,\genl)$ admits a weak Agler
 decomposition and, for each Hilbert space $\HSE,$  every $\Mult(\cH_\genk)$-invariant subspace $\mathcal{M}\subset\mathcal{H}_{\genl}\otimes\mathcal{E}$ is 
 $\Mult(\cH_\genl)$-invariant, then $(\genk, \genl)$ is a BLH pair.
\end{proposition}

\begin{corollary}
    \label{c:AD+invariant} 
       Suppose $(k,\ell)$ is a pair of kernels over a domain $\Omega\subseteq\CC^{\lcal{d}}$ that satisfies Assumption~\ref{assume}, and $\ell$ is circularly symmetric. 
       If $(k,\ell)$ admits a weak Agler decomposition, then
       $(k,\ell)$ is a BLH pair. 
\end{corollary}

\begin{proof}
Since $(k,\ell)$ admits a weak Agler decomposition, $\ell/k\succeq0$ by Remark~\ref{r:weak-include}, so
Lemma~\ref{l:cert-cont-include} gives $\Mult(\cH_k)\subseteq\Mult(\cH_\ell)$.
By Corollary~\ref{c:allequiv}, for each Hilbert space $\HSE$, a subspace of
$\cH_\ell\otimes\HSE$ is $\Mult(\cH_\ell)$-invariant if and only if it is
$\Mult(\cH_k)$-invariant. Proposition~\ref{gensuffBLH} completes the proof.
\end{proof}

\noindent The proof of Proposition~\ref{gensuffBLH} requires the following standard lemma.

\begin{lemma} \label{ellM/k}
Let $\genl$ be a reproducing kernel on a set $X$ and assume 
$\HSF,\HSG$ are Hilbert spaces and 
$\opphi{\Phi}{X}{ \mathcal{H}_{\genl}\otimes\mathcal{F}}{\mathcal{H}_{\genl}\otimes\mathcal{G}}$ is a contractive multiplier.  
If 
$\HSE$ is  a Hilbert space and $\mathcal{M}\subset\mathcal{H}_{\genl}\otimes\mathcal{E}$ is a $\textup{Mult}(\mathcal{H}_{\genl})$-invariant subspace,  then 
\[\big(Q(a)\big[I-\Phi(a)\Phi(b)^*\big]Q(b)^*\big)\genl_\cM(a, b)\succeq 0, \]
for any $Q: X\to \mathcal{B}(\mathcal{G}, \mathbb{C})$. 
\end{lemma}

\begin{proof}
Fix $\mathcal{M}$ and $\mathcal{E}$ as above. Clearly, $\widetilde{\Phi}(a)=I_{\mathcal{E}}\otimes \Phi(a)$ defines a contractive multiplier from $\mathcal{H}_{\genl}\otimes\mathcal{E}\otimes\mathcal{F}$ to $\mathcal{H}_{\genl}\otimes\mathcal{E}\otimes\mathcal{G}$. 
For a function $f\in \cM,$ vectors $e\in \HSE$, $x\in \mathcal{F}$
and $y\in \mathcal{G},$ and a point $a\in X,$
\[
  \langle  M_{\widetilde{\Phi}} (f \otimes x), (\genl_a\otimes e)\otimes y \rangle 
   = \langle \Phi_{x,y} f, \genl_a\otimes e \rangle,
\]
 where $\Phi_{x,y}(a)= \langle \Phi(a)x,y\rangle.$ 
 Since $\Phi_{x,y}$ is a scalar $\mathcal{H}_{\genl}$-multiplier
 and $\cM$ is multiplier-invariant, 
 it follows that $M_{\widetilde{\Phi}}$ maps $\cM\otimes \HSF$
 to $\cM\otimes \mathcal{G}.$  Thus, letting $V$ denote the inclusion of $\cM\otimes \HSF$
 into $\cH_\genl\otimes \HSE\otimes \HSF,$
\[
  V^* M_{\widetilde{\Phi}}^* (P_{\cM}\otimes I_\HSG) (\genl_a\otimes e\otimes g)
   = V^* (\genl_a \otimes e\otimes \Phi(a)^* g).
\]
Since $M_{\widetilde{\Phi}}^*$ is a contraction, 
\begin{equation}
\label{W-star-W}
  0  \preceq  W^*(I-M_{\widetilde{\Phi}} M_{\widetilde{\Phi}}^*) W
\preceq W^*W - W^* M_{\widetilde{\Phi}} VV^* M_{\widetilde{\Phi}}^*W
=:R,
\end{equation}
 where $W$ is the inclusion of $\cM\otimes\HSG$ into $\cH_\genl\otimes \HSE\otimes \HSG$.
On the other hand,  for points $a_1,a_2\in X,$ vectors $e_1,e_2\in \HSE$
  and $g_1,g_2\in \HSG$  
\begin{equation*}
\begin{split}
      \langle R & (P_{\cM} \otimes I_\HSG )(\genl_{a_2}\otimes e_1\otimes g_1), \, (P_{\cM} \otimes I_\HSG)  (\genl_{a_1}\otimes e_2\otimes g_2)  \rangle 
\\ & =  \bigg \langle \big [\genl_{\mathcal{M}}(a_1,a_2)\otimes I_{\mathcal{G}}-\widetilde{\Phi}(a_1)(\genl_{\mathcal{M}}(a_1,a_2)\otimes I_{\mathcal{F}})\widetilde{\Phi}(a_2)^*  
\big ] e_1\otimes g_1,e_2\otimes g_2 \bigg \rangle.
\end{split}
\end{equation*}
 Combining this last equation with
 \eqref{W-star-W}   shows
 that the function
\[
 X\times X \ni (a,b) \mapsto 
 \genl_{\mathcal{M}}(a, b)\otimes I_{\mathcal{G}}-\widetilde{\Phi}(a)(\genl_{\mathcal{M}}(a, b)\otimes I_{\mathcal{F}})\widetilde{\Phi}(b)^*
\]
 is psd.   
Conjugating by $\widetilde{Q}(a):= I_{\mathcal{E}}\otimes Q(a)$ obtains,
\begin{align*}
0 &\preceq \widetilde{Q}(a)\big(\genl_{\mathcal{M}}(a,b)\otimes I_{\mathcal{G}}-\widetilde{\Phi}(a)(\genl_{\mathcal{M}}(a,b)\otimes I_{\mathcal{F}})\widetilde{\Phi}(b)^*\big)\widetilde{Q}(b)^*  \\ 
&= \widetilde{Q}(a)(\genl_{\mathcal{M}}(a,b)\otimes I_{\mathcal{G}})\widetilde{Q}(b)^*-\widetilde{Q}(a)\widetilde{\Phi}(a)(\genl_{\mathcal{M}}(a,b)\otimes I_{\mathcal{F}})\widetilde{\Phi}(b)^*\widetilde{Q}(b)^*  \\ 
&= Q(a)Q(b)^*\genl_{\mathcal{M}}(a,b)-Q(a)\Phi(a)\Phi(b)^*Q(b)^*\genl_{\mathcal{M}}(a,b),\\
&= \big(Q(a)\big[I-\Phi(a)\Phi(b)^*\big]Q(b)^*\big)\genl_{\mathcal{M}}(a,b),
\end{align*}
as desired.
\end{proof}

\begin{proof}[Proof of Proposition \ref{gensuffBLH}] 
Assume
$Q_n(1-\Phi_n\Phi_n^*)Q_n^*\to 1/\kappa$  pointwise with $Q_n, \Phi_n$ as in the assumptions of a weak Agler decomposition. For a 
$\textup{Mult}(\mathcal{H}_{\genl})$-invariant subspace  $\mathcal{M} \subset \cH_\genl \otimes \HSE$,
\begin{equation}
    \label{e:pointwiseQn}
   \big(Q_n(a)\big[I-\Phi_n(a)\Phi_n(b)^*\big]Q_n(b)^*\big)\genl_\cM(a, b)\to \cfrac{\genl_\cM(a, b)}{\genk(a, b)}
   \end{equation}
pointwise over $X\times X.$ For any finite set $F\subset X$ and positive integer $n,$   Lemma \ref{ellM/k} yields 
\[
  \begin{pmatrix} (Q_n(a)[I-\Phi_n(a)\Phi_n(b)^*]Q_n(b)^*)\lambda_\cM(a,b) \end{pmatrix}_{a,b\in F} \succeq0.
\]
Pointwise  convergence in equation~\eqref{e:pointwiseQn} gives 
\[
\begin{pmatrix} \frac{\lambda_\cM (a,b) }{\kappa (a,b) } \end{pmatrix}_{a,b\in F} \succeq0.
\]
Since $F$ was arbitrary, 
$\genl_{\mathcal{M}}/\genk\succeq 0.$  An application of
Proposition~\ref{l_m/k} completes the proof.
\end{proof}

\subsection{Connection to the complete Pick property}  \label{CPconnectsubsec}
Decompositions like that of \eqref{again1/k=} for $1/\genk$ suggest  a natural link with the complete Pick property.
 A kernel $\kappa$ is an (irreducible) complete Pick (or CP) kernel if and only if there exists a function $\delta:X\to\mathbb{C}\setminus\{0\}$ and a positive kernel $P$ so that
\begin{equation}\label{univCP1}
\kappa(a, b)=\frac{\delta(a)\overline{\delta(b)}}{1-P(a, b)}, \hspace{0.2 cm} \text{ for all }a,b\in X.
\end{equation} 
See 
\cite{McCulloughcarath, McCulloughlocal, Quiggin, Pickbook}. 
A now standard (and usually very helpful) observation is that, in this setting, one may write $P(a, b)=\Psi(a)\Psi(b)^*$, where $\Psi\in\textup{Mult}(\mathcal{H}_{\genk}\otimes \mathcal{K},\mathcal{H}_{\genk} )$ is a contractive multiplier and $\mathcal{K}$ an auxiliary Hilbert space.
 Thus, rearranging the identity of equation~\eqref{univCP1} 
shows
\begin{equation}   \label{univCP2}
 \frac{1}{\kappa(a,b)} = Q(a) [1-\Psi(a)\Psi(b)^*]\ol{Q(b)},
\end{equation}
where $Q=1/\delta.$ In this way we see that if $\genk$ is a CP kernel, then
$(\genk,\genk)$ admits a strong Agler decomposition when $\cH_\genk$ is separable.
Indeed, the space $\cH_P$ is contractively contained in
$\cH_{1/(1-P)}=\delta^{-1}\cH_\genk$, so $\mathcal{K}$ in \eqref{univCP2}
can then be chosen separable. 
  On the other hand, if $(\genk,\genk)$ admits a weak Agler 
decomposition, then $(\genk,\genk)$ is a BLH pair by Proposition~\ref{gensuffBLH} and  thus, 
from \cite[Theorem~0.7]{mcctrent}, $\genk$ is a CP
kernel.

 The significance of the complete Pick property is the following interpolation criterion. 
 A (nonvanishing) kernel $\genk$ is a CP kernel if and only if the space of multipliers of $\cH_\genk$ 
 supports the following matrix-valued (complete) version of Pick interpolation: 
 For each pair of positive integers $n$ and $m,$ subset $F=\{a_1,\dots,a_n\}$ of  $X$ and 
 $\Lambda_1,\dots,\Lambda_n\in\CC^{m\times m}$ such that the matrix
\[
  \big[(I-\Lambda_i\Lambda_j^*)\genk(a_i,a_j)\big]_{i,j=1}^n
\]
 is positive semidefinite, there exists a contractive $\CC^{m\times m}$-valued multiplier $\Phi$ of $\cH_\genk\otimes\CC^m$   such that
\[
 \Phi(a_j)=\Lambda_j.
\]
Since Agler introduced the notion of a (complete) Pick kernel in unpublished notes in the late 1980s, there has been a flurry of activity that ultimately led to the development of a very successful theory. Important recent (and less recent) developments include
 \cite{AHMRfact, Chalmsimply, weakPickprod, ColRow, Charadvances, Wickcorona, DBrtoPick}.

The result below is an easy consequence of Lemma~\ref{ellM/k}.

\begin{proposition}\label{CPinbetween}
  Suppose  $(\genk, \genl)$ is a pair of  reproducing kernels
 over a set $X$, with $\genk$ nonvanishing, such that
 if  $\HSE$ is a Hilbert space and $\mathcal{M}\subset\mathcal{H}_{\genl}\otimes\mathcal{E}$ is $\Mult(\cH_\genk)$-invariant, then $\mathcal{M} $ is 
 $\Mult(\cH_\genl)$-invariant. 
 If there exists a nonvanishing CP kernel $s$ such that 
 \[\genl/s\succeq 0, \qquad  s/\genk\succeq 0,\]
then $(\genk, \genl)$ is a BLH pair.
\end{proposition}

The proof uses the observation contained in the following lemma.

\begin{lemma}
    \label{l:l/s}
      Suppose  $\genl$ and $s$ are reproducing kernels over $X$
      and $s$ is a nonvanishing CP kernel as in equation~\eqref{univCP2}. If $\genl/s\succeq 0,$  then $\Psi\in \Mult(\cH_\genl\otimes \mathcal{K},\cH_\genl)$ 
      is a contractive multiplier and moreover,
    \[
       t=\frac{1}{1-\Psi \Psi^*}
    \]
     is a CP kernel and $\genl/t\succeq 0.$
\end{lemma}

\begin{proof}
 By assumption, 
\[
  \genl/s = Q(I-\Psi\Psi^*) Q^* \, \genl
\]
 for a Hilbert space $\mathcal{K},$ a contractive multiplier $\Psi\in \Mult(\cH_s\otimes \mathcal{K}, \mathcal{H}_s),$ and  a
 nonvanishing scalar-valued function $Q.$ Thus, the assumed
 positivity of $\genl/s$ implies
\[
  (I-\Psi \Psi^*) \genl\succeq 0
\]
 and hence $\Psi$ is a contractive multiplier of $\mathcal{H}_{\genl}$ and $\genl/t\succeq0$
 as claimed.
\end{proof}

\begin{proof}[Proof of Proposition~\ref{CPinbetween}]
Let $\HSE$ be a Hilbert space and let
$\cM\subseteq\cH_\genl\otimes\HSE$ be $\Mult(\cH_\genk)$-invariant.
By hypothesis, $\cM$ is $\Mult(\cH_\genl)$-invariant. By Lemma~\ref{l:l/s},
the multiplier $\Psi$ in \eqref{univCP2}, with $s$ in place of $\genk$, is a
contractive multiplier from $\cH_\genl\otimes\mathcal{K}$ to $\cH_\genl$.
Lemma~\ref{ellM/k}, with the scalar function $Q$ in \eqref{univCP2}, gives
$\genl_\cM/s\succeq0$. Since $s/\genk\succeq0$, the Schur product gives
\[
  \frac{\genl_\cM}{\genk}
  =\frac{\genl_\cM}{s}\,\frac{s}{\genk}\succeq0.
\]
Proposition~\ref{l_m/k} therefore gives the BLH property.
\end{proof}

\begin{remark}
By Theorem 1.1 in \cite{Beurlingfactor}, 
 any pair of the form $(s, \genl)$, where $s$ is a nonvanishing CP kernel and $\genl/s\succeq 0$, is a BLH pair, without any additional assumptions on either $s$ or $\genl$.    Here
 $\Mult(\cH_s)\subseteq \Mult(\cH_\lambda)$ by Lemma~\ref{l:cert-cont-include}, since
 $\lambda/s\succeq 0$, so $\Mult(\cH_s)$-invariance makes sense for subspaces of
 $\cH_\lambda\otimes\HSE$ for every Hilbert space $\HSE$.
\qed\end{remark}

\subsection{Comparison with the complete Pick property for pairs} \label{CPsubsec}
It is well known that the BLH and CP properties agree for single nonvanishing kernels \cite{mcctrent}. The situation, however, is very different in the two-kernel setting, as we now explain in brief, with details to follow.  There is a natural
notion of a \textit{CP pair} of kernels.  For a pair $(k,\ell)$ of diagonal holomorphic kernels (defined below) 
the existence of a CP kernel $s$ as in Proposition~\ref{CPinbetween} is necessary \cite[Theorem 1.4]{mcctsik}, but not sufficient \cite[Proposition 7.4]{mcctsik}, for $(k,\ell)$ to be a CP pair.
In Example~\ref{bidiskfamilyex} in Section \ref{Exampsec}, we  construct  a strongly regular BLH pair $(k, \ell)$   of diagonal holomorphic kernels for which    there is no CP kernel $s$ with $\ell/s\succeq 0$ and $s/k\succeq 0$. Thus, the hypotheses of Proposition~\ref{CPinbetween},   while sufficient,   are not even necessary for the BLH property, evidencing a rather significant gap between the BLH and CP properties, even in the presence of strong regularity. 

\begin{definition}[\cite{Shimorin}] \label{CPdef}
A pair $(\genk, \genl)$ of kernels on  $X$ is a \df{complete Pick pair} (or \df{CP pair} for short)  if, for every positive integer $n,$ for every choice of points $a_1, \dots, a_n\in X,$ for every positive integer $N$ and for every choice of  matrices $W_1, \dots, W_n\in\mathbb{C}^{N\times N}$  for  which  the matrix in \eqref{multinterpol} is positive semidefinite, there exists a multiplier $\Phi \in\Mult\big(\mathcal{H}_{\genk}\otimes\mathbb{C}^N, \mathcal{H}_{\genl}\otimes \mathbb{C}^N\big)$ of norm at most one that satisfies, 
 \begin{align*}
  \Phi(a_i)=W_i,
\end{align*}
for $i=1,2,\dots,n.$ 
\qed
\end{definition} 

In the special case $\genk=\genl$, we recover the usual definition of a CP kernel.

 A kernel $s$ on $X$ is a \df{strong Shimorin certificate for $(k, \ell)$} if it is a complete Pick kernel and there exist positive kernels $h, g$ on $X$ with $k=(1-h)s$ and $ \ell=sg$. 
 
Proposition~\ref{CPinbetween} has the following immediate consequence.

 \begin{corollary}
          \label{c:Shim-cert-implies-blh}
     Suppose  $(\genk, \genl)$ is a pair of  reproducing kernels over a set $X$ such that
$\genk$ is nonvanishing and,
whenever $\HSE$ is a Hilbert space and $\mathcal{M}\subset\mathcal{H}_{\genl}\otimes\mathcal{E}$ is $\Mult(\cH_\genk)$-invariant, $\mathcal{M} $ is 
 $\Mult(\cH_\genl)$-invariant. If $(\genk,\genl)$ possesses a 
 strong Shimorin certificate, then $(\genk,\genl)$ is a BLH pair.
 \end{corollary}

 \begin{proof}
      By hypothesis, there are a CP kernel $s$  and positive kernels $g,h$ such that $\genk=(1-h)s$ and $\genl =sg$; since $\genk$ is nonvanishing, so is $s$.
  As $\genk(x,x)>0$ and $s(x,x)>0$, we have $h(x,x)<1$ for every $x\in X$, so $|h(x,y)|^2\le h(x,x)h(y,y)<1$ for all $x,y\in X$.
  Thus, both $s/\genk = 1/(1-h)=\sum_{j\ge0}h^j\succeq 0$ and $\genl/s=g\succeq 0$,  and therefore, by Proposition~\ref{CPinbetween}, $(\genk,\genl)$
  is a BLH pair. 
 \end{proof}

\begin{theorem}[\cite{Shimorin}]\label{Shimsuff}
Suppose $(k, \ell)$ is a pair of reproducing kernels. If there exists a strong
 Shimorin certificate for $(k,\ell),$  then $(k, \ell)$ is a complete Pick pair.   
\end{theorem}

While it is not known if existence of a strong Shimorin certificate is necessary for the complete Pick property in general, we can often say more.  An analytic kernel $\kappa$ over a bounded domain $\Omega$ (containing the origin) is diagonal if it has a power series representation of the form
\[
 \kappa(z,w) =\sum_{\alpha} \genk_\alpha (z\ol{w})^\alpha.
\]
 The kernel is \df{normalized} if $\genk_0=1$; if $\genk_0\ne0$, it can be normalized by dividing by $\genk_0$. Under Assumption~\ref{assume}, every coefficient $\genk_\alpha$ is strictly positive, since each monomial $z^\alpha$ belongs to $\cH_\genk$. 
 
 \begin{theorem}[\cite{mcctsik}] \label{CPpaircharact}
 A normalized diagonal holomorphic pair whose Taylor coefficients are all strictly positive
 is a complete Pick pair if
 and only if it possesses a 
strong Shimorin certificate.
 \end{theorem}

Hence, by Corollary~\ref{c:Shim-cert-implies-blh}, a CP pair as in
Theorem~\ref{CPpaircharact} that satisfies the hypotheses of that corollary
(for instance, a regular one) is BLH.

 \section{Douglas--Leech factorizations} \label{Leechsecc}

   Agler-type decompositions are known to yield Pick interpolation theorems; see \cite{Hellinger} and \cite{Pickbook}. Thus, one expects pairs $(k, \ell)$ supporting an Agler decomposition to admit interpolation theorems of Pick type. Our main objective in this section is to establish such a result. Specifically, in Subsection~\ref{Leechsec} we show that, when the kernel functions of $k$ are linearly independent, the weak Agler decomposition yields a certain Douglas--Leech interpolation property, termed the \textit{complete pseudo-Leech property} (Theorem \ref{AglerdectoLeech}).  In Subsection~\ref{decompsec} it is also shown that, under the same assumption and a mild decomposability assumption on $k$, which holds for every nonvanishing holomorphic kernel on a domain (Lemma~\ref{l:hol-decomp}), the weak and strong Agler decompositions are both equivalent to the complete pseudo-Leech property (Theorem \ref{Leech=wAg=sAg}).

\subsection{The pseudo-Leech property} \label{Leechsec}

\begin{definition} \label{Leechdef}
A pair $(k, \ell)$ on a set $X$ is called a \textit{complete pseudo-Leech pair} if, for all positive integers $n,p,a,b,$ whenever 
\begin{equation} \label{leech-data1}
    \big[(V_iV_j^*-Y_iY_j^*)k(z_i,z_j)\big]_{i,j=1}^n\succeq0
\end{equation}
for points $z_1, \dots, z_n\in X$,  matrices $V_1, \dots, V_n\in \mathbb{C}^{p\times a}$, $Y_1, \dots, Y_n\in \mathbb{C}^{p\times b},$  there exists a contractive multiplier $\opphi{\Phi}{X}{ \mathcal{H}_{\ell}\otimes \mathbb{C}^b}{  \mathcal{H}_{\ell}\otimes \mathbb{C}^a}$
 such that $Y_i=V_i\Phi(z_i)$ for all $i$.
 \qed
\end{definition}

\begin{remark}\label{pseudooutline}
    The complete pseudo-Leech property says that any partially defined positivity condition of the form $(VV^*-YY^*)k\succeq 0$ always factors through a contractive multiplier of $\mathcal{H}_{\ell}$. Note how this differs from the Leech interpolation property that is typically associated with a pair of kernels \cite[Corollary 2.3]{Shimorin}. Indeed, the latter implies the CP property. Example~\ref{bidiskfamilyex} shows that the complete pseudo-Leech property does not imply the CP property, even for strongly regular diagonal holomorphic pairs. See also
     Question~\ref{q:pseudo-Pick}.
    \qed
\end{remark}

Under the invariance hypothesis of Proposition~\ref{gensuffBLH}, the complete pseudo-Leech property implies the BLH property. 

\begin{proposition}\label{LeechtoBLH}
   Let $(k, \ell)$ be a pair of reproducing kernels over a set $X$ such that $k$ is nonvanishing and,
 for  each Hilbert space $\HSE,$  every $\Mult(\cH_k)$-invariant subspace $\mathcal{M}\subset\mathcal{H}_{\ell}\otimes\mathcal{E}$ is 
 $\Mult(\cH_{\ell})$-invariant. If $(k, \ell)$ is complete pseudo-Leech, then it is BLH.
\end{proposition}

\begin{proof} 
Assume $(k, \ell)$ is complete pseudo-Leech. We will show that it is a BLH pair.  To do so it suffices, by
Proposition~\ref{l_m/k}, to show if $\HSE$ is a Hilbert
space and $\cM\subseteq \cH_\ell\otimes \HSE$ 
is  a $\Mult(\cH_k)$-invariant  subspace, then $\ell_{\cM}/k\succeq 0.$ By hypothesis, such an $\cM$ is $\Mult(\cH_\ell)$-invariant, so it suffices to show that
$\ell_{\cM}/k\succeq 0$ whenever $\cM\subseteq \cH_\ell\otimes \HSE$ is $\Mult(\cH_\ell)$-invariant.
 Namely, we show,
\begin{equation*} 
    \bigg[\cfrac{\ell_{\mathcal{M}}(z_i, z_j)}{k(z_i, z_j)}\bigg]_{1\le i, j\le n}\succeq 0,
\end{equation*}
for any positive integer $n$ and points $z_1, \dots, z_n\in X.$
 Since we are working over a finite set of points, we may write $1/k(z_i,z_j)=A(z_i)A(z_j)^*-B(z_i)B(z_j)^*$, where $A: \{z_1, \dots, z_n\}\to\mathbb{C}^{1\times m}$ and 
$B: \{z_1, \dots, z_n\}\to\mathbb{C}^{1\times m}$ for some integer $m$.  By substitution, the  trivial positivity condition $k/k=1\succeq 0$  becomes
\[\big[(A(z_i)A(z_j)^*-B(z_i)B(z_j)^*)k(z_i, z_j)\big]\succeq 0,\]
which is of the form \eqref{leech-data1}. The complete pseudo-Leech hypothesis thus yields a
contractive multiplier $\opphi{\Phi}{X}{ \mathcal{H}_{\ell}\otimes \mathbb{C}^m}{  \mathcal{H}_{\ell}\otimes \mathbb{C}^m}$
 such that $B(z_i)=A(z_i)\Phi(z_i)$ for all $i$. In particular, 
 \[\cfrac{1}{k(z_i, z_j)}=A(z_i)\big(I-\Phi(z_i)\Phi(z_j)^* \big)A(z_j)^*. \]
 Lemma \ref{ellM/k}, applied with $A$ extended by $0$ off $\{z_1, \dots, z_n\}$, then yields $\ell_{\mathcal{M}}/k\succeq 0$ over $\{z_1, \dots, z_n\}.$ Since $\{z_1, \dots, z_n\}$ and $\mathcal{M}$ were arbitrary, the proof is complete. 
\end{proof}

We next show that the existence of a weak Agler decomposition implies the complete pseudo-Leech property.  

\begin{theorem}\label{AglerdectoLeech}
Let $(k, \ell)$ be a pair of reproducing kernels such that the kernel functions $\{k_z\}$ are linearly independent. If $(k, \ell)$ admits a weak Agler decomposition, then it is a complete pseudo-Leech pair.
\end{theorem}
We require two preparatory results. We begin with a WOT version of Montel's theorem. The proof   is relegated to Appendix~\ref{appendMontel}.

\begin{lemma}\label{noholmontel}
Suppose $\mathcal{H}_k$ is a RKHS over a set $X$ and $\mathcal{E},\mathcal{F}$
are Hilbert spaces. Given a net $(\Phi_\alpha)$ of contractive multipliers in
$\Mult(\mathcal{H}_k\otimes\mathcal{E},\mathcal{H}_k\otimes\mathcal{F})$,
there exist a subnet $(\Phi_\beta)$ and a contractive multiplier
$\Phi\in\Mult(\mathcal{H}_k\otimes\mathcal{E},\mathcal{H}_k\otimes\mathcal{F})$
such that $M_{\Phi_\beta}\to M_\Phi$ in the weak operator topology. Moreover,
$\Phi_\beta(x)\to\Phi(x)$ in the weak operator topology of
$\mathcal{B}(\mathcal{E},\mathcal{F})$ for every $x\in X$ with $k(x,x)>0$.
\end{lemma}

Next, we will see how the Agler decomposition allows one to obtain global Douglas--Leech factorizations from partially defined Leech data. The proof is based on a variation of the
standard lurking isometry argument, which leads to a network realization formula. While results similar to these appear in the literature, in the absence of a concrete
reference we have included a proof in Appendix \ref{appendLurking}.

\begin{theorem} \label{realization.}
Let $\ell$ be a reproducing kernel on a set $X$ and fix Hilbert spaces $\mathcal{F}, \mathcal{G}$ and 
a contractive multiplier $\opphi{\Phi}{X}{\mathcal{H}_{\ell}\otimes\mathcal{F}}{ \mathcal{H}_{\ell}\otimes\mathcal{G}}$, {with $\|\Phi(z)\|<1$ for all $z\in X$}. Suppose $F=\{z_1,z_2,\dots,z_n\}\subseteq X$
is finite. If   there exist Hilbert spaces $  \mathcal{K}, \mathcal{L}, \mathcal{J}$, a function $Q:X\to\mathcal{B}(\mathcal{G},\mathbb{C})$ and partially defined data $V:F\to\mathcal{B}(\mathcal{K}, \mathcal{J}), Y:F\to\mathcal{B}(\mathcal{L}, \mathcal{J})$  such that 
\[
  h(z,w):=Q(z)\big[I-\Phi(z)\Phi(w)^*\big]Q(w)^*
\]
 is nonvanishing on $F\times F$ and 
\begin{equation*} 
(V(z)V(w)^*-Y(z)Y(w)^*)h^{-1}(z, w)\succeq 0, \hspace*{0.4 cm} z, w\in F,
\end{equation*}
then there exists a contractive multiplier $\opphi{\Psi}{X}{\mathcal{H}_{\ell}\otimes\mathcal{L}}{\mathcal{H}_{\ell}\otimes\mathcal{K}}$ such that $Y(z)=V(z)\Psi(z)$ for all $z\in F.$
\end{theorem}

\begin{proof}[Proof of Theorem \ref{AglerdectoLeech}]
Assume
\[
  \big[(V_iV_j^*-Y_iY_j^*)k(z_i,z_j)\big]_{i,j=1}^n\succeq0
\]
over a finite set $F=\{z_1,\dots,z_n\}\subseteq X$ of distinct points (repeated
points are merged by stacking the corresponding $V_i$ and $Y_i$ and padding the
remaining $V_j$ and $Y_j$ with zero rows), where
$V_i\in\mathcal{B}(\mathcal{K},\mathcal{J})$ and
$Y_i\in\mathcal{B}(\mathcal{L},\mathcal{J})$, and
$\mathcal{K},\mathcal{L},\mathcal{J}$ are finite-dimensional Hilbert spaces.
Let $F'=\{z\in F:\ell(z,z)>0\}$. If $F'=\emptyset$, take $\Psi_1=0$ and
proceed to the last paragraph. Otherwise, relabel the points so that
$F'=\{z_1,\dots,z_q\}$.

Since $(k,\ell)$ admits a weak Agler decomposition, there are positive
integers $i_m$, functions $Q_m:X\to\mathcal{B}(\mathbb{C}^{i_m},\CC)$ and
contractive multipliers
$\Phi_m\in\Mult(\mathcal{H}_\ell\otimes\mathbb{C}^{i_m})$ such that
\[
  h_m:=Q_m(I-\Phi_m\Phi_m^*)Q_m^*\longrightarrow1/k
\]
pointwise on $X\times X$. Redefine $\Phi_m(z)=0$ whenever $\ell(z,z)=0$.
This changes neither $M_{\Phi_m}$ nor $h_m$ on $F'\times F'$, and
$\|\Phi_m(z)\|\le1$ now holds for every $z\in X$.

Choose $\epsilon_m>0$ decreasing to $0$. Linear independence of the kernel
functions implies that $[k(z_i,z_j)]_{i,j=1}^q$ is positive definite, so
\[
  \big[(\epsilon_m I_{\mathcal{J}}+V_iV_j^*-Y_iY_j^*)k(z_i,z_j)\big]_{i,j=1}^q
  \succ0.
\]
Since $k$ is nonvanishing, the reciprocal kernels $h_m^{-1}$ are eventually
defined on $F'\times F'$ and converge there to $k$. Passing to a subsequence
of the $h_m$ and re-indexing, we may
therefore assume that each $h_m$ is nonvanishing on $F'\times F'$ and
\[
  \big[(\epsilon_m I_{\mathcal{J}}+V_iV_j^*-Y_iY_j^*)h_m^{-1}(z_i,z_j)\big]_{i,j=1}^q
  \succ0.
\]
Replace each $\Phi_m$ by $r_m\Phi_m$, with $0<r_m<1$ sufficiently close to $1$
that nonvanishing, the convergence to $1/k$ on $F'\times F'$ and these strict
inequalities are preserved; continue to denote the resulting kernels by
$h_m$. Then $\|\Phi_m(z)\|<1$ for every $m$ and $z\in X$.
For each $m$, there are a positive integer $p$ and a function
$R_m:F'\to\mathcal{B}(\mathbb{C}^p,\mathcal{J})$ such that
\begin{align*}
  \epsilon_m I_{\mathcal{J}}+V_iV_j^*-Y_iY_j^*
  &=h_m(z_i,z_j)R_m(z_i)R_m(z_j)^*\\
  &=Q_m(z_i)\big[I-\Phi_m(z_i)\Phi_m(z_j)^*\big]Q_m(z_j)^*\\
  &\qquad{}\times R_m(z_i)R_m(z_j)^*,\qquad 1\le i,j\le q.
\end{align*}
Apply Theorem~\ref{realization.} on $F'$ with the augmented operators
\[
  \begin{bmatrix}V_i&\sqrt{\epsilon_m}I_{\mathcal{J}}\end{bmatrix}:
  \mathcal{K}\oplus\mathcal{J}\longrightarrow\mathcal{J}.
\]
It gives a contractive multiplier
\[
  \Psi_m\in\Mult\big(\mathcal{H}_\ell\otimes\mathcal{L},
                     \mathcal{H}_\ell\otimes(\mathcal{K}\oplus\mathcal{J})\big)
\]
with $Y_i=V_i\Psi_{1,m}(z_i)+\sqrt{\epsilon_m}\Psi_{2,m}(z_i)$ for
$1\le i\le q$, where $\Psi_m=\operatorname{col}(\Psi_{1,m},\Psi_{2,m})$.
By Lemma~\ref{noholmontel}, a subnet converges in the weak operator topology
to a contractive multiplier
$\Psi=\operatorname{col}(\Psi_1,\Psi_2)$, both as multiplication operators and
pointwise at every $z$ with $\ell(z,z)>0$. Here
$\Psi_1(z)\in\mathcal{B}(\mathcal{L},\mathcal{K})$ and
$\Psi_2(z)\in\mathcal{B}(\mathcal{L},\mathcal{J})$.
Since $\|\Psi_{2,m}(z_i)\|\le1$ for $z_i\in F'$, taking the limit along this
subnet yields $Y_i=V_i\Psi_1(z_i)$ on $F'$, and $\Psi_1$ is a contractive
multiplier.

For each $z_i\in F\setminus F'$, the diagonal entry of the original data,
divided by $k(z_i,z_i)>0$, gives $V_iV_i^*\succeq Y_iY_i^*$. By Douglas'
lemma, there is a contraction $W_i:\mathcal{L}\to\mathcal{K}$ such that
$Y_i=V_iW_i$. Redefine $\Psi_1(z_i)=W_i$ at these points. Every element of
$\mathcal{H}_\ell$ vanishes there, so this does not change the multiplication
operator or its norm. Thus $Y_i=V_i\Psi_1(z_i)$ for every $i$, proving the
complete pseudo-Leech property.
\end{proof}

\subsection{Decomposable kernels}  \label{decompsec}
A nonvanishing reproducing kernel $k$ on a set $X$ is \df{decomposable} if
$1/k=VV^*-YY^*$ for functions $V:X\to\cB(\mathcal{K},\CC)$ and
$Y:X\to\cB(\mathcal{L},\CC)$, where $\mathcal{K}$ and $\mathcal{L}$ are separable
Hilbert spaces; that is, if $1/k$ is the difference of two positive kernels whose
reproducing kernel Hilbert spaces are separable. For kernels on an arbitrary set,
$1/k$ is a difference of two positive kernels if, for instance, it has only
finitely many negative squares \cite[Theorem 1.1.3]{RKPontryagin}; for more
conditions see \cite[Theorem 5.2]{Gheondea} and \cite[Theorem 6.1]{CBkernels}.
On a countable set, every reproducing kernel Hilbert space is separable. By
Lemma~\ref{l:hol-decomp} below, every nonvanishing holomorphic kernel on a
domain is decomposable. The point is that, for decomposable $k$ with linearly
independent kernel functions, weak and strong Agler decompositions coincide.

\begin{theorem}\label{Leech=wAg=sAg}
Suppose  $(k, \ell)$ is a pair of reproducing kernels  over a set $X$
 and  $1/k$ decomposes as 
 \begin{equation} \label{1/kLeech}
 1/k = VV^* -YY^*,
 \end{equation}
 where $V: X\to\mathcal{B}(\mathcal{K},\mathbb{C})$ and $Y:X\to\mathcal{B}(\mathcal{L},\mathbb{C})$ and  $\mathcal{K}, \mathcal{L}$
 are separable Hilbert spaces.

 Consider the following conditions.
\begin{enumerate}[(i)]\itemsep=6pt
    \item  \label{LAGi} $(k, \ell)$ is a complete pseudo-Leech pair; 

     \item  \label{LAGii} $(k, \ell)$ admits a weak Agler decomposition; 

   \item  \label{LAGiii} $(k, \ell)$ admits a strong Agler decomposition of the form,
\begin{equation}\label{1/k-Leechdecomp}
     1/k = V(I-\Psi\Psi^*)V^*,
\end{equation}
 for a contractive multiplier $\Psi\in\Mult(\cH_\ell\otimes\KL,\cH_\ell\otimes\LK)$ with $Y=V\Psi$.
\end{enumerate}
Then \ref{LAGi} implies \ref{LAGiii}, and \ref{LAGiii} implies \ref{LAGii}. If the
kernel functions $\{k_x\}$ are linearly independent, then \ref{LAGii} implies
\ref{LAGi}, so the three conditions are equivalent.

Conversely, if $(k,\ell)$ admits a strong Agler decomposition, then $k$ is decomposable.
\end{theorem}

\begin{proof}[Proof of Theorem~\ref{Leech=wAg=sAg}]
The implication \ref{LAGiii} implies \ref{LAGii} is Remark~\ref{r:strong-weak}, and,
when the kernel functions $\{k_x\}$ are linearly independent, \ref{LAGii} implies
\ref{LAGi} by Theorem~\ref{AglerdectoLeech}. It remains to show that \ref{LAGi}
implies \ref{LAGiii}. By assumption,
\[
  \frac1{k(z,w)}=V(z)V(w)^*-Y(z)Y(w)^*,\qquad z,w\in X,
\]
where $V:X\to\mathcal{B}(\mathcal{K},\CC)$ and
$Y:X\to\mathcal{B}(\mathcal{L},\CC)$, and $\mathcal{K},\mathcal{L}$ are
separable Hilbert spaces. It suffices to find a contractive multiplier
$\Psi\in\Mult(\mathcal{H}_\ell\otimes\mathcal{L},
               \mathcal{H}_\ell\otimes\mathcal{K})$
such that $Y(z)=V(z)\Psi(z)$ for all $z\in X$.

For a non-empty finite set $F\subseteq X$, let
\[
  \mathcal{K}_F=\operatorname{span}\{V(z)^*1:z\in F\},\qquad
  \mathcal{L}_F=\operatorname{span}\{Y(z)^*1:z\in F\}.
\]
For $z\in F$, we have $V(z)=V(z)P_{\mathcal{K}_F}$ and
$Y(z)=Y(z)P_{\mathcal{L}_F}$. The restrictions therefore act on
finite-dimensional spaces and satisfy
\[
  \big[(V(z)V(w)^*-Y(z)Y(w)^*)k(z,w)\big]_{z,w\in F}
  =[1]_{z,w\in F}\succeq0.
\]
Apply the complete pseudo-Leech property to these restrictions, padding with
zero coordinates if necessary. Extend the resulting multiplier $\Psi_F$ by
zero, as $\widetilde\Psi_F=\iota_{\mathcal{K}_F}\Psi_F P_{\mathcal{L}_F}$, where
$\iota_{\mathcal{K}_F}:\mathcal{K}_F\to\mathcal{K}$ is the inclusion.
This produces a contractive multiplier from
$\mathcal{H}_\ell\otimes\mathcal{L}$ to $\mathcal{H}_\ell\otimes\mathcal{K}$
with $Y(z)=V(z)\widetilde\Psi_F(z)$ for $z\in F$.

Let
\[
  Z=\Mult(\mathcal{H}_\ell\otimes\mathcal{L},
          \mathcal{H}_\ell\otimes\mathcal{K}),
\]
identifying multipliers that induce the same multiplication operator, and
equip $Z$ with the inherited weak operator topology. The unit ball of $Z$ is
WOT-compact by Lemma~\ref{noholmontel}. Let $A$ be the set of finite subsets of
$X$, and, for $F\in A$, let $Z_F$ be the set of contractive multipliers
$\Phi\in Z$ such that $Y(z)=V(z)\Phi(z)$ for every $z\in F$ with
$\ell(z,z)>0$. For non-empty $F$, the multiplier $\widetilde\Psi_F$ lies in $Z_F$; $Z_\emptyset$ is the whole unit ball.
If a net $(M_{\Phi_\alpha})$ from $Z_F$ converges WOT to $M_\Phi$ in the unit
ball, the identity $\langle\Phi_\alpha(z)e,f\rangle=\langle M_{\Phi_\alpha}(\ell_z\otimes e),\ell_z\otimes f\rangle/\ell(z,z)$
gives $\Phi_\alpha(z)\to\Phi(z)$ WOT whenever
$\ell(z,z)>0$. Thus $Y(z)=V(z)\Phi(z)$ at those points of $F$, so $Z_F$ is
WOT-closed in the unit ball and hence compact.

For finite sets $F_1,\dots,F_r\subseteq X$,
\[
  \bigcap_{j=1}^r Z_{F_j}=Z_{\bigcup_{j=1}^rF_j}\ne\emptyset.
\]
The finite intersection property therefore gives a contractive multiplier
\[
  \Psi\in\bigcap_{F\in A}Z_F.
\]
It satisfies $Y(z)=V(z)\Psi(z)$ whenever $\ell(z,z)>0$.
At each $x\in X$ with $\ell(x,x)=0$,
\[
  V(x)V(x)^*-Y(x)Y(x)^*=\frac1{k(x,x)}>0.
\]
By Douglas' lemma, choose a contraction $W_x:\mathcal{L}\to\mathcal{K}$ with
$Y(x)=V(x)W_x$ and redefine $\Psi(x)=W_x$. This leaves its multiplication
operator unchanged because all functions in $\mathcal{H}_\ell$ vanish there.
Consequently $Y=V\Psi$ on $X$, and \eqref{1/k-Leechdecomp} follows.

Finally, if $(k,\ell)$ admits a strong Agler decomposition as in
Definition~\ref{strongAgler}, then $1/k=QQ^*-(Q\Phi)(Q\Phi)^*$ with $\mathcal{G}$
and $\HSF$ separable, so $k$ is decomposable.
\end{proof}

We end this section by showing that decomposability is automatic for
nonvanishing holomorphic kernels. As a consequence, a weak Agler decomposition
of a pair satisfying Assumption~\ref{assume} can always be replaced by a strong
Agler decomposition with holomorphic $Q$. Lemma~\ref{l:hol-decomp} is a global counterpart of
results of Alpay and of Constantinescu and Gheondea on reproducing kernel Krein spaces of
holomorphic kernels on disks and balls; see \cite{gheondea2013survey}.

\begin{lemma}\label{l:hol-decomp}
  Let $\Omega\subseteq\CC^{\lcal{d}}$ be a domain and let $h\in\Her(\Omega)$
  satisfy $h(z,w)=\ol{h(w,z)}$ for $z,w\in\Omega$. There exist a separable
  Hilbert space $\LK$ and holomorphic functions $V,Y:\Omega\to\cB(\LK,\CC)$ such
  that
  \[
    h(z,w)=V(z)V(w)^*-Y(z)Y(w)^*,\qquad z,w\in\Omega.
  \]
\end{lemma}

\begin{proof}
  Choose compact sets $K_1\subseteq K_2\subseteq\cdots$ with union $\Omega$ such
  that each $K_j$ is contained in the interior of $K_{j+1}$, and set
  $K_0=\emptyset$. Define $c:\Omega\to[1,\infty)$ by
  \[
    c(u)=1+\max\{|h(v,v')|:v,v'\in K_{j+1}\}\qquad\text{for }u\in K_{j+1}\setminus K_j,\ j\ge0.
  \]
  If $u\in K_{i+1}\setminus K_i$ and $v\in K_{j+1}\setminus K_j$ with $j\le i$, then
  $u,v\in K_{i+1}$, so $|h(u,v)|\le c(u)\le c(u)c(v)$. Exchanging the roles of
  $u$ and $v$, we get $|h(u,v)|\le c(u)c(v)$ for all $u,v\in\Omega$. Let
  $g(u)=(1+|u|^2)^{-\lcal{d}-1}$, so that $\gamma=\int_\Omega g(u)\,du$ is
  finite, and let $\mu$ be the measure on $\Omega$ with density $\rho=g/c^2$ with
  respect to Lebesgue measure. Then
  \begin{equation}\label{e:hol-decomp-L2}
    \int_\Omega|h(u,w)|^2\,d\mu(u)\le\gamma\,c(w)^2\quad(w\in\Omega),
    \qquad
    \int_{\Omega\times\Omega}|h|^2\,d(\mu\times\mu)\le\gamma^2.
  \end{equation}

  Every compact subset of $\Omega$ is contained in some $K_j$, and
  $c\le1+\max\{|h(v,v')|:v,v'\in K_j\}$ on $K_j$, so $\rho$ is bounded below by
  a positive constant on each compact subset of $\Omega$. Hence, by the
  sub-mean-value property of the subharmonic function $|f|^2$, evaluation at
  points of
  $\Omega$ is bounded, uniformly on compact subsets of $\Omega$, on the space
  $A^2(\mu)$ of holomorphic functions on $\Omega$ that lie in $L^2(\mu)$. Thus
  $A^2(\mu)$ is a closed subspace of $L^2(\mu)$ and a reproducing kernel
  Hilbert space. Let $b_z\in A^2(\mu)$ be its kernel function at $z\in\Omega$,
  so that $f(z)=\langle f,b_z\rangle$ for $f\in A^2(\mu)$, and let $P$ be the
  orthogonal projection of $L^2(\mu)$ onto $A^2(\mu)$.

  By \eqref{e:hol-decomp-L2}, $h(\cdot,w)\in A^2(\mu)$ for each $w\in\Omega$.
  Using the reproducing property and $h(u,w)=\ol{h(w,u)}$,
  \[
    h(z,w)=\langle h(\cdot,w),b_z\rangle,
    \qquad
    h(u,w)=\ol{\langle h(\cdot,u),b_w\rangle}=\int_\Omega h(u,v)\,b_w(v)\,d\mu(v)
  \]
  for $u,w,z\in\Omega$. By \eqref{e:hol-decomp-L2}, the integral operator
  \[
    (Hf)(u)=\int_\Omega h(u,v)f(v)\,d\mu(v)
  \]
  is a Hilbert--Schmidt operator on $L^2(\mu)$, and it is self-adjoint because
  $h(u,v)=\ol{h(v,u)}$. The second identity says $Hb_w=h(\cdot,w)$, so the first
  gives
  \[
    h(z,w)=\langle Hb_w,b_z\rangle,\qquad z,w\in\Omega.
  \]
  Write $H=H_+-H_-$ with $H_\pm=\frac12(|H|\pm H)\succeq0$, let $\LK=L^2(\mu)$,
  which is separable, and define $V(z),Y(z)\in\cB(\LK,\CC)$ by
  $V(z)x=\langle x,H_+^{1/2}b_z\rangle$ and $Y(z)x=\langle x,H_-^{1/2}b_z\rangle$.
  Then $V(w)^*1=H_+^{1/2}b_w$, so $V(z)V(w)^*=\langle H_+b_w,b_z\rangle$;
  likewise $Y(z)Y(w)^*=\langle H_-b_w,b_z\rangle$, and the asserted identity
  follows. Finally, for $x\in\LK$, the function
  $z\mapsto V(z)x=\langle H_+^{1/2}x,b_z\rangle=(PH_+^{1/2}x)(z)$ is holomorphic,
  and $\|V(z)\|\le\|H\|^{1/2}\|b_z\|$ is bounded on compact subsets of $\Omega$,
  since $\|b_z\|$ is the norm of evaluation at $z$. Hence $V$ is holomorphic, and
  so is $Y$.
\end{proof}

\noindent In particular, if $k$ is a nonvanishing holomorphic kernel on a
domain, then $k$ is decomposable, and the functions $V$ and $Y$ in
\eqref{1/kLeech} can be taken holomorphic, with $\LK=\KL$ separable.

\begin{corollary} \label{strongAglertoAgler}
  Let $(k, \ell)$ be a pair of reproducing kernels over a domain
  $\Omega\subseteq\CC^{\lcal{d}}$ satisfying Assumption~\ref{assume}. If
  $(k,\ell)$ admits a weak Agler decomposition, then there exist separable
  Hilbert spaces $\HSF$ and $\HSG$, a holomorphic function
  $Q:\Omega\to\cB(\HSG,\CC)$ and a contractive multiplier
  $\Phi\in\Mult(\cH_\ell\otimes\HSF,\cH_\ell\otimes\HSG)$ such that
  \[
    \frac{1}{k(z,w)}=Q(z)\big(I-\Phi(z)\Phi(w)^*\big)Q(w)^*,\qquad z,w\in\Omega.
  \]
  Moreover, $\Phi$ is holomorphic and $\|\Phi(z)\|\le1$ for $z\in\Omega$. In
  particular, $(k,\ell)$ admits a strong Agler decomposition.
\end{corollary}

\begin{proof}
  By Definition~\ref{weakAgler}, $k$ is nonvanishing. Since $k$ is
  holomorphic, $1/k\in\Her(\Omega)$, and $1/k(z,w)=\ol{1/k(w,z)}$. Thus
  Lemma~\ref{l:hol-decomp} gives a separable Hilbert space $\LK$ and holomorphic
  functions $V,Y:\Omega\to\cB(\LK,\CC)$ with $1/k=VV^*-YY^*$.

  By Assumption~\ref{assume}, the polynomials are dense in $\cH_\ell$; hence
  $\cH_\ell$ contains the constant function $1$, so
  $\ell(z,z)>0$ for $z\in\Omega$. Also by Assumption~\ref{assume}, the
  polynomials are multipliers of $\cH_k$, and this implies that the kernel
  functions $\{k_z\}$ are linearly independent: if $\sum_{i=1}^nc_ik_{z_i}=0$
  for distinct points $z_1,\dots,z_n$, fix $j$ and let $p$ be a polynomial with
  $p(z_i)=\delta_{ij}$; then $0=M_p^*\sum_ic_ik_{z_i}=c_jk_{z_j}$, and
  $k_{z_j}\neq0$ because $k(z_j,z_j)\neq0$.

  By Theorem~\ref{AglerdectoLeech}, $(k,\ell)$ is a complete pseudo-Leech pair.
  Hence Theorem~\ref{Leech=wAg=sAg}, applied with $\KL=\LK$, gives a contractive
  multiplier $\Psi\in\Mult(\cH_\ell\otimes\LK)$ with $Y=V\Psi$, so that
  $1/k=VV^*-V\Psi\Psi^*V^*=V(I-\Psi\Psi^*)V^*$. Take $\HSF=\HSG=\LK$, $Q=V$
  and $\Phi=\Psi$. Finally, for $e\in\HSF$, the function
  $z\mapsto\Phi(z)e$ is the element $M_\Phi(1\otimes e)$ of $\cH_\ell\otimes\HSG$,
  hence holomorphic; and, since $M_\Phi^*(\ell_z\otimes g)=\ell_z\otimes\Phi(z)^*g$
  for $g\in\HSG$ and $\ell(z,z)>0$, contractivity of $M_\Phi$ gives
  $\|\Phi(z)\|\le1$. Being bounded and weakly holomorphic, $\Phi$ is
  holomorphic.
\end{proof}

\section{The co-extension property}\label{Heredsection}
 \label{ss:co-ext-pairs}
It is known that, for a complete Pick kernel $\ell$, a tuple $T$ of bounded commuting operators
with Taylor spectrum in $\Omega$ satisfies von Neumann's inequality with respect to
$\Mult(\mathcal{H}_{\ell})$ if and only if $T$ is a $1/\ell$ contraction; see
\cite{ArvesonIII} for the Drury--Arveson kernel, and \cite[Corollary~7]{Polynomialarazyenglis}
and \cite{AgMcClocal} in general. That this property characterizes the complete Pick property
is shown in \cite{AgMcClocal}; see also \cite{hyperigid}. For a general kernel $\ell,$  there is no a priori reason for a tuple $T$ satisfying the $\Mult(\mathcal{H}_{\ell})$-von Neumann inequality to be a $1/\ell$ contraction. It is, however, reasonable to expect that this property could be partially salvaged if we demand a less restrictive $1/k$-contractivity condition. These considerations motivate the notion of a co-extension pair. See Definition~\ref{coextprop}.  The highlight of this section is Theorem~\ref{COEXTISAGLER}, which says, under
our regularity hypotheses, a pair is a co-extension pair if and only if it admits an Agler decomposition.

Given a  kernel $\ell$, let \index{$\mathcal{A}(\mathcal{H}_{\ell})$}  $\mathcal{A}(\mathcal{H}_{\ell})\subseteq \cB(\cH_\ell)$ denote the norm closure of the polynomials as multipliers in $\text{Mult}(\mathcal{H}_{\ell}).$ 

\begin{definition}
    \label{def:mult-contraction}
A $\lcal{d}$-tuple of bounded commuting operators  $T=(T_1, \dots, T_{\lcal{d}})$ on Hilbert space $H$ is \df{$\Mult(\cH_\ell)$-contractive} if
\begin{equation} \label{vonNeumann}
  \|P(T)\| \le \|P\|_{\text{Mult}(\mathcal{H}_{\ell}\otimes \mathbb{C}^m)}  
\end{equation}
 for all $m$ and polynomials $P\in \CC[z_1, \dots, z_{\lcal{d}}]\otimes \CC^{m\times m}.$
 We then also call $T$ a \df{$\Mult(\cH_\ell)$-contraction}.
\qed \end{definition}

\begin{remark} \label{SpInclusion}
 The inequality of \eqref{vonNeumann} has the following reformulation. The map \[p(M^{\ell}_z)\mapsto p(T) \]
is a unital completely contractive homomorphism from $\mathcal{A}(\mathcal{H}_{\ell})$ to $\mathcal{B}(H)$. Thus, if $\specT(M^{\ell}_z)=\overline{\Omega}$ {is polynomially convex}, then $\specT(T)\subset \overline{\Omega}$ by Lemma \ref{Taylorinclusion}.  In particular, if $(k,\ell)$ is a regular pair over $\Omega,$
then $\specT(T)\subseteq \ol{\Omega}.$ 
\qed
\end{remark}

\begin{definition} \label{coextprop}
  A regular pair  $(k, \ell)$ is a \df{co-extension pair} if each $\Mult(\cH_\ell)$-contraction
  is a $1/k$ contraction.
\qed\end{definition}

 \begin{remark}\label{well-def:rem}
 If $(k, \ell)$ is a regular pair over $\Omega,$ then  $1/k\in\Her{(\Omega')}$ for some open set $\Omega'\supseteq\overline{\Omega}$. 
 In this case, by  Remark \ref{SpInclusion}, $1/k(T, T^*)$ is defined via the hereditary functional calculus whenever $T$ is a $\Mult(\cH_\ell)$-contraction. 
 \qed
 \end{remark}

The existence of a $(k,\ell)$-exhaustion in the definition of a regular pair permits a 
significant weakening in the definition of a co-extension pair.

\begin{lemma} \label{relax-coext}
Suppose  $(k, \ell)$ is a regular pair   on a domain $\Omega\subset\mathbb{C}^{\lcal{d}}$. If
 $T$ is a $\Mult(\cH_\ell)$-contraction and $\specT(T)\subseteq \Oml$ implies $T$ is a $1/k$ contraction,
 then $(k,\ell)$ is a co-extension pair. 
\end{lemma}

\begin{proof}
Assume $(k, \ell)$ is a regular pair and  $T$ is a $\Mult(\cH_\ell)$-contraction. By Remark \ref{SpInclusion},  $\specT(T)\subset\overline{\Omega}.$  Fix a $(k, \ell)$-exhaustion $(r_n,\Omega^\prime)$ of $\Omega$ (item~\ref{i:regular:exh} of Definition~\ref{def:regular}) and set $k_n(z,w)=k(r_n(z), r_n(w)).$  By item~\ref{i:exh:ucs} of Definition~\ref{kexhaust}, $1/k_n\to 1/k$ in $\Her(\Omega')$.  Thus, by item~\ref{i:her-calc:v} of Lemma~\ref{l:Her-calc}, 
\begin{equation}\label{1/k-approx}
1/k_n(T, T^*)\to 1/k(T, T^*)
\end{equation}
in operator norm.
Since $p(M^{\ell}_z)\mapsto p\circ r_n(M^{\ell}_z)$ is  completely contractive (by item~\ref{i:exh:cc}
of the exhaustion assumption, Definition~\ref{kexhaust}), 
  \[
\|P\circ r_n(T)\|\le \|P\circ r_n(M_z^\ell)\|
  \le \| P(M_z^\ell)\| = \|P \|_{\textup{Mult}(\mathcal{H}_{\ell}\otimes\mathbb{C}^m)}
  \]
  for {$P\in\mathbb{C}[z]\otimes\mathbb{C}^{m\times m}$}  and $m\ge 1.$ Thus, 
   $r_n(T)$ is a $\Mult(\cH_\ell)$-contraction and satisfies $\specT(r_n(T)) \subseteq r_n(\ol{\Omega}) \subseteq \Omega.$  Hence, by assumption and Lemma~\ref{l:her-compose}, applied with $G=\Omega''$ and $U=r_n^{-1}(\Omega'')\cap\Omega'$, where $\Omega''\supseteq\ol{\Omega}$ is an open set on which $1/k$ is a limit of hereditary polynomials (Remark~\ref{r:polycvx}), $1/k_n(T,T^*)=1/k(r_n(T),r_n(T)^*)\succeq0$. 
   The operator-norm convergence in equation~\eqref{1/k-approx} guarantees that $T$ is a $1/k$ contraction, as desired. 
\end{proof}
We will now employ a cone separation argument to show:
\begin{theorem}\label{COEXTISAGLER}
A regular pair   $(k, \ell)$ is a co-extension pair if and only if it admits an Agler decomposition. 
\end{theorem}

\begin{proof}   We will first show that $(k, \ell)$ having an Agler decomposition implies that it is a co-extension pair. The proof is standard. Assume $1/k\in \mathcal{AG}(\ell)$, so there exists a sequence $\{i_n\}\subset\mathbb{N}$, analytic functions $Q_n:\Omega\to \mathcal{B}(\mathbb{C}^{i_n}, \CC)$ and contractive polynomial multipliers $P_n\in\Mult(\mathcal{H}_{\ell}\otimes\mathbb{C}^{i_n})$ such that
\begin{equation}  \label{Aglerconv}
    h_n(z, w):=Q_n(z)\big[I_{i_n\times i_n}-P_n(z)P_n(w)^* \big]Q_n(w)^*\xrightarrow{\text{\hspace{0.1 cm} \hspace{0.1 cm}}}  1/k(z, w) 
\end{equation}
 uniformly on compact subsets of $\Omega\times\Omega.$
Assuming $T\in\mathcal{B}(H)^{\lcal{d}}$ is a $\Mult(\cH_\ell)$-contraction,
 the goal is to show that $1/k(T, T^*)\succeq 0.$ 
\par

In view of Lemma \ref{relax-coext}, it suffices to assume $\specT(T)\subset\Omega$. The convergence in equation~\eqref{Aglerconv} implies by continuity of the hereditary functional calculus, that
\begin{equation}\label{limit}
    h_n(T, T^*)=Q_n(T)\big[I-P_n(T)P_n(T)^*\big]Q_n(T)^*\xrightarrow{\text{\hspace{0.1 cm} \hspace{0.1 cm}}}1/k(T, T^*)
\end{equation}
in the operator norm. Also, in view of \eqref{vonNeumann}, $P_n$ being a contractive multiplier implies $I-P_n(T)P_n(T)^*\succeq 0$. Conjugating, we obtain
\begin{equation*} 
Q_n(T)\big[I-P_n(T)P_n(T)^*\big]Q_n(T)^*\succeq 0, \hspace{0.4 cm} \forall n\ge 1.
\end{equation*}
Hence equation~\eqref{limit} yields $T$ is a $1/k$ contraction, as desired.  \par 

Now, we prove that the Agler decomposition is necessary; that is, assuming $(k,\ell)$ is a co-extension pair,
we will show $1/k\in \mathcal{AG}(\ell)$ using a standard cone separation argument.
 
The set 
\[\mathcal{R}=\{h\in\text{Her}(\Omega):\ h(z, w)=\overline{h(w, z)}\}\]
 is a real subspace of $\text{Her}(\Omega)$. Observe that $1/k\in \mathcal{R}$ and that $\mathcal{AG}(\ell)$ is closed by definition, is a
 cone by Lemma~\ref{closed-under-sums}, and lies in $\mathcal{R}$, since each function of the
 form \eqref{aglerfact} does and $\mathcal{R}$ is closed. 
Thus, the assertion $1/k\in \mathcal{AG}(\ell)$ follows from the Hahn-Banach separation theorem once
it is shown that 
\begin{equation}\label{cone1}
L(1/k)\ge 0    
\end{equation}
whenever $L\in \mathcal{R}^*$ (where $\mathcal{R}^*$ is the dual of $\mathcal{R})$ and 
\begin{equation}\label{cone2}
 L(h)\ge 0
\end{equation}
for all  $h\in \mathcal{AG}(\ell).$
Accordingly, fix a continuous real linear functional $L$ so that \eqref{cone2} holds and note it extends
 (uniquely) to a continuous complex linear functional $\widetilde{L}$ defined on all of $\text{Her}(\Omega)$ via the formula $\widetilde{L}(h)=L\big(\tfrac{h+h^\sharp}{2}\big)+iL\big(\tfrac{h-h^\sharp}{2i}\big)$, where $h^\sharp(z,w)=\ol{h(w,z)}$. The functional $\widetilde{L}$  gives rise to a sesquilinear form on $\overline{\textup{Hol}(\Omega)}:=\{\bar{f}:\ f\in\textup{Hol}(\Omega)\}$:
\[
  \langle \bar{f}, \bar{g}\rangle_0=\widetilde{L}(g(z)\overline{f(w)}), 
\]
 for $f, g\in \textup{Hol}(\Omega).$
 For 
 $f\in \textup{Hol}(\Omega),$  choosing $Q=f$ and $P=0$ in Definition~\ref{Aglerwedge},  $f(z)\overline{f(w)}\in \mathcal{AG}(\ell)$   and thus
\[\langle \bar{f}, \bar{f}\rangle_0=\widetilde{L}(f(z)\overline{f(w)})=L(f(z)\overline{f(w)})\ge 0.\]
Consequently, $\langle \cdot, \cdot\rangle_0$ is positive semidefinite,   $\mathcal{N}=\{\bar{f}\in\overline{\textup{Hol}(\Omega)}: \langle \bar{f}, \bar{f}\rangle_0=0\}$ is a subspace of $\overline{\textup{Hol}(\Omega)} $ and 
\[\langle \overline{f}+\mathcal{N}, \overline{g}+\mathcal{N}\rangle_L:=\langle \overline{f}, \overline{g}\rangle_0, \hspace{0.4 cm} f, g\in\textup{Hol}(\Omega)\]
is an inner product on the quotient $\overline{\textup{Hol}(\Omega)}/\mathcal{N}.$
Finally, let $H_L$ denote the completion of $\overline{\textup{Hol}(\Omega)}/\mathcal{N}$ with respect to $\langle \cdot, \cdot\rangle_L$. \par 
Since $z_1, \dots, z_{\lcal{d}}\in\text{Mult}(\mathcal{H}_{\ell})$, there exists a constant $C$ such that
\[C^2-z_i\overline{w}_i\in \mathcal{AG}(\ell), \hspace{0.4 cm} 1\le i\le\lcal{d}.\]
Thus,
\[C^2\|\bar{f}+\mathcal{N}\|^2_L-\|\bar{z_i}\bar{f}+\mathcal{N}\|^2_L=L\big(f(z)(C^2-z_i\overline{w}_i)\overline{f(w)}\big)\ge 0, 
\]
 for all $f\in \textup{Hol}(\Omega),$
where the last inequality follows from \eqref{cone2}.
It follows that, for every $1\le i\le\lcal{d}$, the operator $T_i^*:H_L\to H_L$ defined by 
\[T^*_i\big(\bar{f}+\mathcal{N}\big)=\bar{z_i}\bar{f}+\mathcal{N}\]
is bounded. The operators $T_i$ commute since their adjoints do. Next, consider an arbitrary $P\in \CC[z_1, \dots, z_{\lcal{d}}]\otimes\mathbb{C}^{n\times n}$, with $n\ge 1.$ Without loss of generality assume $\|P\|_{\text{Mult}(\mathcal{H}_{\ell}\otimes\mathbb{C}^n)}\le 1$. Our goal is to show that $P(T)\in \mathcal{B}(H_L\otimes\mathbb{C}^n)$
is a contraction. Accordingly, let  $f_1, \dots, f_n\in \textup{Hol}(\Omega)$ and observe that
\begin{equation}\label{key}
\begin{bmatrix}
 f_1(z)& \cdots & f_n(z)   
\end{bmatrix}\big(I_{n\times n}-P(z)P(w)^* \big)\begin{bmatrix}
\overline{f_1(w)} \\ 
\vdots \\
\overline{f_n(w)} 
\end{bmatrix}\in  \mathcal{AG}(\ell).
\end{equation}
Next, set $P=[p_{ij}]$ and write
\begin{align*}
\Bigg|\Bigg|P(T)^*\begin{bmatrix}
\bar{f_1}+\mathcal{N} \\ \vdots \\\bar{f_n}+\mathcal{N}
\end{bmatrix}\Bigg|\Bigg|^2_{H_L\otimes\mathbb{C}^n}
=&\Bigg|\Bigg|\begin{bmatrix}
\sum_{j=1}^np_{j1}(T)^*(\bar{f_j}+\mathcal{N}) \\ \vdots \\ \sum_{j=1}^np_{jn}(T)^*(\bar{f_j}+\mathcal{N})
\end{bmatrix}\Bigg|\Bigg|^2_{H_L\otimes\mathbb{C}^n}
\\
=&\Bigg|\Bigg|\begin{bmatrix}
\sum_{j=1}^n\overline{p}_{j1}\bar{f}_j+\mathcal{N} \\ \vdots \\ \sum_{j=1}^n\overline{p}_{jn}\bar{f}_j+\mathcal{N}
\end{bmatrix}\Bigg|\Bigg|^2_{H_L\otimes\mathbb{C}^n}.
\end{align*}
Hence 
\begin{align*}
&\ \Bigg|\Bigg|\begin{bmatrix}
\bar{f_1}+\mathcal{N} \\ \vdots \\\bar{f_n}+\mathcal{N}
\end{bmatrix}\Bigg|\Bigg|^2_{H_L\otimes\mathbb{C}^n}-\Bigg|\Bigg|P(T)^*\begin{bmatrix}
\bar{f_1}+\mathcal{N} \\ \vdots \\\bar{f_n}+\mathcal{N}
\end{bmatrix}\Bigg|\Bigg|^2_{H_L\otimes\mathbb{C}^n}
\\
=&\ 
\sum_{i=1}^n\|\bar{f}_i+\mathcal{N}\|^2_{H_L}-\sum_{i=1}^n\Big|\Big|\sum_{j=1}^n\overline{p}_{ji}\bar{f}_j+\mathcal{N}\Big|\Big|^2_{H_L}
\\
=&\ L\Big(\sum_{i=1}^nf_i(z)\overline{f_i(w)}\Big)- L\Big(\sum_{i=1}^n\Big[\sum_{j=1}^np_{ji}(z)f_j(z)\Big]\Big[\sum_{j=1}^n\overline{p_{ji}(w)f_j(w)}\Big] \Big)\\
=&\ L\Bigg(\begin{bmatrix}
 f_1(z)& \cdots & f_n(z)   
\end{bmatrix}\big(I_{n\times n}-P(z)P(w)^* \big)\begin{bmatrix}
\overline{f_1(w)} \\ 
\vdots \\
\overline{f_n(w)} 
\end{bmatrix}\Bigg) \\
\ge&\ 0,
\end{align*}
where the last inequality follows from \eqref{cone2} and \eqref{key}. Thus, $P(T)$ is a contraction, and so $T$ is a  $\Mult(\cH_\ell)$-contraction.  Since $(k, \ell)$ is a co-extension pair by assumption,  $1/k(T, T^*)\succeq 0$ (recall that, by Remarks~\ref{SpInclusion} and  \ref{well-def:rem}, $1/k(T, T^*)$ is  a bounded operator defined
through the hereditary functional calculus).

By regularity, there exists an open set $\Omega'\supseteq \ol{\Omega}$ such that $1/k\in\Her(\Omega')$.  By Remark~\ref{r:polycvx}, there is an open
set $\ol{\Omega} \subseteq U \subseteq \Omega^\prime$  and a sequence
of hereditary polynomials $p_n$ that converge to $1/k$ uniformly
on compact subsets of $U\times U.$
The sequence $(p_n(T, T^*))_n$ converges to $1/k(T, T^*)$ in operator norm
by Lemma~\ref{l:Her-calc} item~\ref{i:her-calc:v}. Setting $p_n(z,w)=\sum p_{n,\alpha, \beta}z^{\alpha}\overline{w}^{\beta}$, 
\begin{align*}
0\le &\ \big\langle 1/k(T, T^*)(1+\mathcal{N}), 1+\mathcal{N}\big\rangle_{L}\\ 
=&\lim_n\big\langle p_n(T, T^*)(1+\mathcal{N}), 1+\mathcal{N}\big\rangle_L  
\\
=&\lim_n\sum p_{n,\alpha,\beta}\big\langle T^{\alpha}T^{*\beta}(1+\mathcal{N}), 1+\mathcal{N}\big\rangle_L  
\\
=&\lim_n\sum p_{n, \alpha,\beta}\ \widetilde{L}(z^{\alpha}\overline{w}^{\beta})  
\\
=&\lim_n {\widetilde{L}(p_n)} 
\\
=& L(1/k),
\end{align*}
where the last equality follows from continuity of $\widetilde{L}$, from $p_n\to1/k$ in $\Her(\Omega)$ and from $1/k\in\mathcal{R}$. Thus,
 \eqref{cone1} holds and our proof is complete.
\end{proof}

\section{The BLH and co-extension properties}\label{BLHandco-extSec}

The main result of this section, Theorem~\ref{BLHISCOEXT}, asserts that the BLH
and co-extension properties coincide for strongly regular pairs. A barebones
version of our argument is as follows. Fix a strongly regular pair $(k, \ell)$
and set $S=M^{\ell}_z$. If $(k,\ell)$ is either a BLH pair or co-extension pair,
then $S$ is a $1/k$ contraction and hence so is its compression to any
co-invariant subspace. However, the BLH property is equivalent to the assertion
that the restriction of any direct sum of copies of $S$ to an invariant subspace
is a $1/k$ contraction; see Lemma~\ref{recastBLH}. Combining these two facts, we
obtain that $(k, \ell)$ is BLH if and only if every operator that dilates to
copies of $S$ is a $1/k$ contraction; see Lemmas~\ref{recastBLH} and~\ref{intermediate}. By standard
operator model theory, this last property turns out to be equivalent to the
co-extension property, which concludes the proof.

\par We point out that strong regularity is only used when passing from the BLH
to the co-extension property. The issue is that arbitrary
$\Mult(\cH_\ell)$-contractions dilate to operators made up of copies of $S$ plus a
residual summand. The BLH property, however, only gives information about
invariant-subspace-restrictions of direct sums of $S$ that do not involve any
residual summands. The dilation property of an exhaustion
(Definition~\ref{def:regular}), which item~\ref{i:regular:dilation} of that
definition requires, says that the residual summand is absent once the
exhaustion maps $r_n$ are applied.
Szeg\H{o}'s kernel, $\ell(z,w)= (1-z\ol{w})^{-1},$ provides an illustrative example.
In this case, $S$ is the unilateral shift and a $1/\ell$ contraction is a
contraction. A version of the Sz.-Nagy Dilation Theorem says if $T$ is a
contraction, then there exists a cardinal $\omega$ and a unitary operator $U$ such
that $T^*$ is the restriction of $(\oplus_{j=1}^\omega S^*)\oplus U$ to an
invariant subspace. Here $U$ is the residual summand. If $0<t_n<1$ is a sequence
that converges to $1,$ then $r_n(T)=t_n T$ is a strict contraction, so the
residual unitary summand is absent in the Sz.-Nagy dilation of $r_n(T)$; thus
$r_n(z)=t_nz$ induces an exhaustion with the dilation property. The strategy is
now clear: verify that $r_n(T)$ satisfies the BLH hypothesis and take a limit.
Subsection~\ref{ss:dilation} gives two sufficient conditions for the dilation
property.

The proofs of the results in Subsection~\ref{ss:dilation} use the following two well-known tools: a Sz.-Nagy--Foia\c{s}-type model for $1/k$ contractions (Theorem \ref{arazyenglis}) and the Arveson Extension Theorem (Theorem \ref{Arvesonsext}). 

  \begin{theorem}[{\cite[Corollary~7]{Polynomialarazyenglis}}]
 \label{arazyenglis}
Assume $k$ is a nonvanishing kernel on $\Omega\subseteq\CC^{\lcal{d}}$ satisfying Assumption~\ref{assume}, and let $T\in\cB(H)^{\lcal{d}}$ be a tuple of commuting operators with $\specT(T)\subseteq\Omega$. If $(1/k)(T,T^*)\succeq0$, then there exists a Hilbert space $\HSE$ such that $T$ is unitarily equivalent to the compression of $M^k_z\otimes I_\HSE$ to a co-invariant subspace.
\end{theorem}

\begin{theorem}\label{Arvesonsext}
Let $\mathcal{A}$ be a unital subalgebra of the unital $C^*$-algebra $\mathcal{C}$ and $H$ a Hilbert space. If $\rho: \mathcal{A}\to\mathcal{B}(H)$ is a completely contractive unital homomorphism, then there exists a Hilbert space $K\supset H$ and a unital $*$-representation $\pi: \mathcal{C}\to\mathcal{B}(K)$ such that,
letting $V:H\to K$ denote the inclusion, 
\[
 \rho(a)=V^* \pi(a) V, 
\]
for all $a\in\mathcal{A}.$ 
\end{theorem}\noindent

The following general fact about $1/k$ contractions is standard. A proof is provided for the sake of completeness.

\begin{lemma}\label{compresstocoinv}
Assume $k$ is a kernel on a domain $\Omega\subset\mathbb{C}^{\lcal{d}}$ with polynomially convex closure. Assume also that $1/k\in\Her(\Omega')$ for some open set $\Omega'\supset\ol{\Omega}$ and let $T\in \mathcal{B}(H)^{\lcal{d}}$ be a  tuple of commuting operators with $\specT(T)\subset \overline{\Omega}.$ If $T$ is a $1/k$ contraction, $\cM\subset H$ is a co-invariant subspace for $T$ and
$V:\cM\to H$ is the inclusion,  then $V^*TV$ is a $1/k$ contraction.
\end{lemma}

\begin{proof}
Suppose $\cM$ is co-invariant for $T,$ let $V:\cM\to H$ denote the inclusion
and set $J=V^*T V$. We are to show $1/k(J,J^*)\succeq 0.$
Since $\mathcal{M}$ is co-invariant for $T$, it is invariant for $T^*$ and hence 
\[
   J^{\alpha}=V^* T^{\alpha} V,
\]
 for $\alpha\in\mathbb{N}^{\lcal{d}}.$ 
 It follows that the map $p(T) \mapsto p(J)$ for polynomials $p$ is a unital completely contractive
 representation. Since $\specT(T)\subset \overline{\Omega}$ and $\ol{\Omega}$ is polynomially convex, $\specT(J)\subseteq \ol{\Omega}$
 by Lemma~\ref{Taylorinclusion}. From Remark~\ref{r:polycvx}, 
 there are an open set $\Omega''$ with
  $\ol{\Omega}\subseteq\Omega''\subseteq\Omega'$ and a
   sequence of hereditary polynomials $(p_n)$ that converges to $1/k$ in
   $\Her(\Omega'').$  Hence, by the hereditary functional calculus,
   $p_n(T,T^*)$  and $p_n(J,J^*)$ converge in operator norm to $1/k(T,T^*)$ and $1/k(J,J^*)$, 
   respectively. On the other hand,
   since for all $\alpha, \beta\in\mathbb{N}^{\lcal{d}},$ 
\begin{equation*} V^* T^{\alpha}T^{*\beta} V =V^* T^{\alpha}P^2_{\mathcal{M}}T^{*\beta}V =J^{\alpha}J^{*{\beta}},
\end{equation*}
 for each $n,$
 \[
  p_n(J,J^*) =   V^* p_n(T,T^*) V.
 \]
 Hence,
\[
 1/k(J,J^*) = V^*( 1/k(T,T^*) ) V\succeq 0,
\]
 where the last inequality holds since $T$ is a $1/k$ contraction.
\end{proof}

We now prove another reformulation of the BLH property. 
 \begin{lemma}\label{recastBLH}
If $(k, \ell)$ is a regular pair, then  $(k, \ell)$ is a BLH pair if and only if, for any Hilbert space $\mathcal{E}$, the restriction of $M^{\ell}_z\otimes I_{\mathcal{E}}$ to any invariant subspace yields a $1/k$ contraction. 
 \end{lemma}

\begin{proof}
 By item~\ref{i:regular:include} of Definition~\ref{def:regular}, the subspaces of $\cH_\ell\otimes\HSE$ invariant for $M_z^\ell\otimes I_\HSE$ are exactly the $\Mult(\cH_k)$-invariant ones. Let $\cM$ be such a subspace. Let $T=P_\mathcal{M}(M^{\ell}_z\otimes I_{\mathcal{E}})|_{\mathcal{M}}=\big(P_{\mathcal{M}}(M^{\ell}_{z_1}\otimes I_{\mathcal{E}})|_{\mathcal{M}}, \dots, P_{\mathcal{M}}(M^{\ell}_{z_{\lcal{d}}}\otimes I_{\mathcal{E}})|_{\mathcal{M}}\big)$. Since $\specT(M^{\ell}_z)\subset \overline{\Omega} $ and $ \overline{\Omega} $ is polynomially convex,  $\specT(T)\subset \overline{\Omega}$ by Lemma \ref{Taylorinclusion}. {In view of Proposition \ref{l_m/k}}, it suffices to show that $\ell_{\cM}/k$ being positive semidefinite is equivalent to $1/k(T, T^*)\succeq 0.$

Since $\ol{\Omega}$ is polynomially convex and $1/k\in\Her(\Omega')$ for some open set $\Omega'\supset\ol{\Omega}$, Remark~\ref{r:polycvx} gives an open set $U\supseteq\ol{\Omega}$ and hereditary polynomials $p_n$ converging to $1/k$ in $\Her(U)$.
Let $g_{w,\eta}\in \mathcal{M}$ denote the kernel vector at $w\in\Omega$ with coefficient
$\eta\in \mathcal{E}$, so that
$g_{w,\eta}(\cdot)=\ell_{\mathcal{M}}(\cdot,w)\eta$
and
$\langle f, g_{w,\eta}\rangle = \langle f(w), \eta\rangle_{\mathcal{E}}$
for $f\in \mathcal{M}.$ Since $\mathcal{M}$ is invariant for $M^{\ell}_z\otimes I_{\mathcal{E}}$, each adjoint $T_j^*$ acts on kernel vectors by
$
T_j^* g_{w,\eta} = \overline{w_j}\, g_{w,\eta}.$
For every hereditary polynomial $p$, linearity gives
\[
\langle p(T,T^*) g_{w,\eta},\, g_{z,\xi}\rangle
=
\left\langle p(z,w)\,\ell_{\mathcal{M}}(z,w)\eta,\xi \right\rangle,
\]
for all $z,w\in\Omega$ and $\eta,\xi\in\mathcal{E}$. Applying this identity to $p_n$ and using item~\ref{i:her-calc:v} of Lemma~\ref{l:Her-calc}, we obtain
\[
\big\langle (1/k)(T,T^*)g_{w,\eta},g_{z,\xi}\big\rangle
=\left\langle \frac{\ell_{\mathcal{M}}(z,w)}{k(z,w)}\eta,\xi\right\rangle.
\]
Since finite linear combinations of kernel vectors are dense in $\mathcal{M}$, this identity gives
\begin{equation}\label{operatorhered}
  (1/k)(T,T^*)\succeq0\quad\Longleftrightarrow\quad\ell_{\mathcal{M}}/k\succeq0,
\end{equation}
as required.
\end{proof}

 We will require the notion, due to Sarason, of a \df{semi-invariant subspace}; these are subspaces that are differences of two invariant subspaces. More precisely,  given a tuple $T\in\mathcal{B}(H)^{\lcal{d}}$ of commuting operators, a  subspace $\mathcal{M}$ of $H$ is a semi-invariant subspace of $T$ if there exists a decomposition 
$H=H_{-}\oplus\mathcal{M}\oplus H_{+}$ with respect to which 
\[T_i=\begin{bmatrix}
  * & * & * \\
  0 & P_{\mathcal{M}}T_i|_{\mathcal{M}} & * \\
  0 & 0 & *
\end{bmatrix}, \hspace*{0.4 cm} 1\le i \le\lcal{d}. \]
The following result is due to Sarason (see e.g. \cite[Theorem 7.6]{Paulsenbook}).

\begin{lemma}\label{dilation-decomp}
If $T\in\mathcal{B}(H)^{\lcal{d}}$ and  $S\in\mathcal{B}(K)^{\lcal{d}}$ are tuples of commuting operators and  $H$ is a  subspace of $K$ and $V:H\to K$ is the inclusion, then 
\[p(T)=V^* p(S) V \]
  $\text{ for all }p\in\mathbb{C}[z_1, \dots, z_{\lcal{d}}]$
if and only if $H$ is semi-invariant for $S$ and $T_j=V^*S_jV$ for $1\le j\le\lcal{d}$.
\end{lemma}

Now, we show that $(k, \ell)$ being BLH actually implies a stronger restriction property.

\begin{lemma} \label{intermediate}
If $(k, \ell)$ is a regular BLH pair and $\mathcal{E}$ a Hilbert space,
then the compression of $M^{\ell}_z\otimes I_{\mathcal{E}}$ 
to any of its semi-invariant subspaces is a $1/k$ contraction.
\end{lemma}

\begin{proof}
Set $T=M^{\ell}_z\otimes I_{\mathcal{E}}$ and $\mathcal{N}=\mathcal{H}_{\ell}\otimes \mathcal{E}$ and let $\mathcal{M}\subset\mathcal{N}$ be a semi-invariant subspace of $T$. Thus, there exists a decomposition  $\mathcal{N}=\mathcal{N}_{-}\oplus \mathcal{M}\oplus \mathcal{N}_{+}$ so that 
\[T_i=\begin{bmatrix}
    * & * & * \\
    0 &  P_{\mathcal{M}}T_i|_{\mathcal{M}} & * \\
    0 & 0 & * 
\end{bmatrix}, \]
for all $1\le i\le\lcal{d}.$ Our goal is to show that $P_{\mathcal{M}}T|_{\mathcal{M}}$ is a $1/k$ contraction. Since $\mathcal{N}_{-}\oplus\mathcal{M}$ is an invariant subspace of $T$, Lemma \ref{recastBLH} implies that the restriction of $T$ to $\mathcal{N}_{-}\oplus\mathcal{M}$, call it $J$, is a $1/k$ contraction. But $\mathcal{M}$ is a co-invariant subspace of $J$, and so, by Lemma~\ref{compresstocoinv}, $P_{\mathcal{M}}J|_{\mathcal{M}}=P_{\mathcal{M}}T|_{\mathcal{M}}$ is a $1/k$ contraction, as desired.
\end{proof}

We now show that the BLH and co-extension properties coincide. 

\begin{theorem} \label{BLHISCOEXT}
If $(k,\ell)$ is a regular co-extension pair, then $(k,\ell)$  is a BLH pair.

If $(k,\ell)$ is a strongly regular BLH pair, then it is a co-extension pair. 
\end{theorem}

\begin{proof}
We prove the easy direction first. Assume $(k, \ell)$ is a co-extension pair. To prove $(k,\ell)$ is BLH, let $\mathcal{E}$ be a Hilbert space and $\mathcal{M}$ a $ \Mult(\cH_k)$-invariant subspace of $\mathcal{H}_{\ell}\otimes {\mathcal{E}}$. By the regularity hypothesis,
$\cM$ is invariant for $M_z^\ell \otimes I_\HSE.$  Set $T=P_{\mathcal{M}}(M^{\ell}_z\otimes I_{\mathcal{E}})|_{\mathcal{M}}$ and let $Q\in\mathbb{C}[z_1, \dots, z_{\lcal{d}}]\otimes\mathbb{C}^{m\times m}$,  $m\ge 1.$ Then,
\begin{align*}
\|Q(T)\|=&\ \|(P_{\mathcal{M}}\otimes I_{m\times m})Q(M^{\ell}_z\otimes I_{\mathcal{E}})|_{\mathcal{M}\otimes \mathbb{C}^m}\| \\
\le& \ \|Q(M^{\ell}_z\otimes I_{\mathcal{E}})\|\\
=&\ \|Q(M^{\ell}_z)\otimes I_{\mathcal{E}}\| \\
=&\ \|Q(M^{\ell}_z)\| \\
=&\ \|Q\|_{\textup{Mult}(\mathcal{H}_{\ell}\otimes\mathbb{C}^m)},
\end{align*}
and thus $T$ is  $\textup{Mult}(\mathcal{H}_{\ell})$-contractive. Since $(k, \ell)$ is a co-extension pair, $T$ is a $1/k$ contraction. By Lemma \ref{recastBLH} $(k, \ell)$ is a BLH pair, as desired. 

\par 
 
To prove the converse, assume $(k, \ell)$ is a strongly regular BLH pair, fix a
$(k, \ell)$-exhaustion $(r_n,\Omega^\prime)$ with the dilation property and set
$k_n(z, w)=k(r_n(z), r_n(w))$. Let  $T=(T_1, \dots, T_{\lcal{d}})$ be a
$\Mult(\cH_\ell)$-contractive tuple of operators on a Hilbert space $H$ and fix
$n\ge1$. By the dilation property, there exist a Hilbert space $\HSF$ and an
isometry $V:H\to\cH_\ell\otimes\HSF$ such that
\[
  p(r_n(T))=V^*\big(p(M_z^\ell)\otimes I_\HSF\big)V=V^*p(M_z^\ell\otimes I_\HSF)V
\]
for all polynomials $p$. By Lemma~\ref{dilation-decomp}, $VH$ is a semi-invariant
subspace of $M_z^\ell\otimes I_\HSF$, and $r_n(T)$ is unitarily equivalent to the
compression of $M_z^\ell\otimes I_\HSF$ to $VH$. Lemma~\ref{intermediate} then
yields that $r_n(T)$ is a $1/k$ contraction. By Remark~\ref{r:polycvx}, there is an open set
$\Omega''$ with $\ol{\Omega}\subseteq\Omega''\subseteq\Omega'$ such
that $1/k$ is the limit in $\Her(\Omega'')$ of a sequence of hereditary
polynomials. Since $\specT(T)\subseteq\ol{\Omega}$ (Remark~\ref{SpInclusion})
and $r_n(\ol{\Omega})\subseteq\Omega$, Lemma~\ref{l:her-compose}, applied with
$G=\Omega''$ and $U=r_n^{-1}(\Omega'')\cap\Omega'$, gives
$0\preceq1/k(r_n(T),r_n(T)^*)=1/k_n(T,T^*)$. Since $1/k_n\to1/k$ in
$\Her(\Omega')$ (as part of the exhaustion hypothesis),
item~\ref{i:her-calc:v} of Lemma~\ref{l:Her-calc} shows that $1/k_n(T,T^*)$
converges to $1/k(T,T^*).$ Thus, $T$ is a $1/k$ contraction, and the proof is
complete.
\end{proof}

\subsection{Sufficient conditions for the dilation property}
\label{ss:dilation}
Our first sufficient condition asks that $\ell$ be nonvanishing, that $1/\ell$ be
approximable by hereditary polynomials and that the exhaustion satisfy
item~\ref{i:annoying} of Definition~\ref{kexhaust}. It does not involve the BLH
property.

\begin{lemma}\label{l:dilation-criterion}
  Let $(k,\ell)$ be a regular pair over $\Omega$. Suppose that $\ell$ is
  nonvanishing and that $1/\ell$ is the limit in $\Her(\Omega)$ of a sequence of
  hereditary polynomials. Then every $(k,\ell)$-exhaustion satisfying
  item~\ref{i:annoying} of Definition~\ref{kexhaust} has the dilation property.
\end{lemma}

\begin{proof}
  Let $(r_n,\Omega')$ be a $(k,\ell)$-exhaustion satisfying
  item~\ref{i:annoying} of Definition~\ref{kexhaust}, let $T$ be a
  $\Mult(\cH_\ell)$-contractive tuple on a Hilbert space $H$, and set
  $S=M_z^\ell$. Fix $n\ge1$ and set $S_n=r_n(S)$ and
  $\ell_n(z,w)=\ell(r_n(z),r_n(w))$.

  By Remark~\ref{SpInclusion} there exists a unital completely contractive
  homomorphism $\rho:\mathcal{A}(\cH_\ell)\to\cB(H)$ with $\rho(S)=T$. Letting
  $C^*(\mathcal{A}(\cH_\ell))$ denote the $C^*$-algebra generated by
  $\mathcal{A}(\cH_\ell)$, Theorem~\ref{Arvesonsext} produces a Hilbert space $K$,
  an isometry $W:H\to K$ and a $*$-representation
  $\pi:C^*(\mathcal{A}(\cH_\ell))\to\cB(K)$ such that $p(T)=W^*\pi(p(S))W$ for all
  polynomials $p$. Hence
  \begin{equation}\label{e:dilation-rep}
    p(r_n(T))=W^*p(\pi(S_n))W\qquad\text{for all polynomials }p.
  \end{equation}

  By the spectral mapping theorem for the Taylor spectrum and item (ii) of
  Definition~\ref{kexhaust},
  $\specT(S_n)=r_n(\specT(S))=r_n(\ol{\Omega})\subseteq\Omega$. A commuting tuple
  $X$ is Taylor invertible if and only if $D_X+D_X^*$ is invertible, where $D_X$
  is the boundary map of the Koszul complex of $X$, a
  $2^{\lcal{d}}\times2^{\lcal{d}}$ matrix whose entries are $0$ and $\pm X_j$
  \cite{Curtosurveyref}. As $C^*$-algebras are inverse-closed and $\pi$ preserves
  invertibility, it follows that
  $\specT(\pi(S_n))\subseteq\specT(S_n)\subseteq\Omega$.

  Since $\ell$ is nonvanishing, $1/\ell\in\Her(\Omega)$; let $(q_m)$ be hereditary
  polynomials with $q_m\to1/\ell$ in $\Her(\Omega)$. As
  $q_m(\pi(S_n),\pi(S_n)^*)=\pi\big(q_m(S_n,S_n^*)\big)$ and $\pi$ is norm
  continuous, letting $m\to\infty$ and using item~\ref{i:her-calc:v} of
  Lemma~\ref{l:Her-calc} twice gives
  \[
    \frac1\ell\big(\pi(S_n),\pi(S_n)^*\big)=\pi\Big(\frac1\ell\big(S_n,S_n^*\big)\Big).
  \]
  Let $B$ be an open ball containing $\ol{\Omega}$, and let $U$ be the connected
  component of $r_n^{-1}(\Omega)\cap B$ that contains the connected set
  $\ol{\Omega}$. Thus $U$ is a bounded domain with $\ol{\Omega}\subseteq U$ and
  $r_n(U)\subseteq\Omega$. By Lemma~\ref{l:her-compose}, applied with $G=\Omega$,
  $\frac1\ell(S_n,S_n^*)=\frac1{\ell_n}(S,S^*)$. Since $\ol{\Omega}=\specT(S)$ is
  polynomially convex, Lemma~\ref{l:her-calc+}, applied to $1/\ell_n\in\Her(U)$,
  gives
  \[
    \Big\langle \frac1{\ell_n}(S,S^*)\ell_b,\ell_a\Big\rangle=\frac{\ell(a,b)}{\ell_n(a,b)},
    \qquad a,b\in\Omega.
  \]
  By item~\ref{i:annoying} of Definition~\ref{kexhaust}, $\ell/\ell_n=g_n\succeq0$.
  As the kernel functions span a dense subspace of $\cH_\ell$, it follows that
  $\frac1{\ell_n}(S,S^*)\succeq0$. Hence
  \[
    \frac1\ell\big(\pi(S_n),\pi(S_n)^*\big)=\pi\Big(\frac1{\ell_n}(S,S^*)\Big)\succeq0 .
  \]

  Thus $\specT(\pi(S_n))\subseteq\Omega$ and
  $\frac1\ell(\pi(S_n),\pi(S_n)^*)\succeq0$, so Theorem~\ref{arazyenglis},
  applied with the nonvanishing kernel $\ell$, gives a
  Hilbert space $\HSF$, a subspace $\mathcal{N}\subseteq\cH_\ell\otimes\HSF$ that is
  co-invariant for $S\otimes I_\HSF$, and a unitary operator $\Gamma:K\to\mathcal{N}$
  with $\Gamma\pi(S_n)\Gamma^*=P_\mathcal{N}(S\otimes I_\HSF)|_\mathcal{N}$. Since $\mathcal{N}$ is
  co-invariant, $p(\pi(S_n))=\Gamma^*P_\mathcal{N}\,p(S\otimes I_\HSF)|_\mathcal{N}\Gamma$ for
  all polynomials $p$. Let $V:H\to\cH_\ell\otimes\HSF$ be the isometry
  $W$ followed by $\Gamma$ and the inclusion of $\mathcal{N}$. By \eqref{e:dilation-rep},
  $p(r_n(T))=V^*\big(p(S)\otimes I_\HSF\big)V$ for all polynomials $p$.
\end{proof}

\begin{remark}\rm
  \label{r:dilation-criterion}
  The hypothesis on $1/\ell$ in Lemma~\ref{l:dilation-criterion} holds if $1/\ell$
  extends to an element of $\Her(\Omega'')$ for an open set
  $\Omega''\supseteq\ol{\Omega}$, by Remark~\ref{r:polycvx}, since $\ol{\Omega}$ is
  polynomially convex. Example~\ref{eg:certifiably-annoying} provides kernels
  $\ell$ and polynomial maps $r_n$ for which item~\ref{i:annoying} of
  Definition~\ref{kexhaust} holds.
  \qed
\end{remark}

Our second sufficient condition is circular symmetry of $\ell$. It requires
neither that $\ell$ be nonvanishing nor any factorization of $\ell$. The proof is
given in Appendix~\ref{append-radial}.

\begin{theorem}\label{t:radial-dilation}
  Let $\ell$ be a circularly symmetric kernel on a domain
  $\Omega\subseteq\CC^{\lcal{d}}$ satisfying Assumption~\ref{assume}, and let $T$
  be a $\Mult(\cH_\ell)$-contractive $\lcal{d}$-tuple on a Hilbert space $H$. For
  every $0<t<1$, there exist a Hilbert space $\HSF$ and an isometry
  $V:H\to\cH_\ell\otimes\HSF$ such that
  \[
    p(tT)=V^*\big(p(M_z^\ell)\otimes I_\HSF\big)V\quad\text{for all }
    p\in\CC[z_1,\dots,z_{\lcal{d}}].
  \]
\end{theorem}

\noindent The proof applies, with the changes indicated at the end of
Appendix~\ref{append-radial}, when $\Omega$ and $\ell$ are invariant under a
weighted action
$z\mapsto(\tau^{m_1}z_1,\dots,\tau^{m_{\lcal{d}}}z_{\lcal{d}})$, $\tau\in\T$, with
positive integers $m_1,\dots,m_{\lcal{d}}$; the conclusion then holds for
$(t^{m_1}T_1,\dots,t^{m_{\lcal{d}}}T_{\lcal{d}})$ in place of $tT$. Domains
invariant under such an action are called quasi-circular.

A domain $\Omega$ is \df{balanced} if $\lambda\Omega\subseteq\Omega$
whenever $|\lambda|\le1$. The ball and the polydisk are bounded convex balanced
domains.

\begin{corollary}\label{c:radial-strong}
  Let $\Omega\subseteq\CC^{\lcal{d}}$ be a bounded convex balanced domain and let
  $(k,\ell)$ be a pair of analytic kernels on $\Omega$ satisfying
  Assumption~\ref{assume}, such that $\ell$ is circularly symmetric and
  $\specT(M_z^\ell)=\ol{\Omega}$. Suppose that $k$ is nonvanishing, that $1/k$
  extends to an element of $\Her(\Omega')$ for some open set
  $\Omega'\supseteq\ol{\Omega}$, and that $(k,\ell)$ satisfies
  item~\ref{i:regular:include} of Definition~\ref{def:regular}. Then $(k,\ell)$
  is strongly regular. More precisely, there is $\varepsilon>0$ such that, for
  every sequence $0<t_n<1$ converging to $1$, the pair
  $(r_n,(1+\varepsilon)\Omega)$ with $r_n(z)=t_nz$ is a $(k,\ell)$-exhaustion
  with the dilation property.
\end{corollary}

\begin{proof}
  Let $\mu$ be the Minkowski functional of $\Omega$. Since $\Omega$ is bounded,
  open, convex and balanced, $\mu$ is a norm on $\CC^{\lcal{d}}$,
  $\Omega=\{\mu<1\}$ and $\ol{\Omega}=\{\mu\le1\}$. By compactness of
  $\ol{\Omega}$, there is $\varepsilon>0$ with $\{\mu\le1+\varepsilon\}\subseteq\Omega'$.
  Set $\Omega_\varepsilon=(1+\varepsilon)\Omega=\{\mu<1+\varepsilon\}$, and let
  $0<t_n<1$ converge to $1$.

  We check Definition~\ref{kexhaust} for $(r_n,\Omega_\varepsilon)$. Item (i)
  holds because $k$ is nonvanishing and $1/k$ extends to an element of
  $\Her(\Omega_\varepsilon)$. Item (ii) holds because
  $r_n(\ol{\Omega})=\{\mu\le t_n\}\subseteq\Omega$ and
  $r_n(\Omega_\varepsilon)\subseteq\Omega_\varepsilon$. For item~\ref{i:exh:ucs},
  let $h$ be the holomorphic function on
  $\Omega_\varepsilon\times\Omega_\varepsilon^*$ with $h(z,\ol{w})=1/k(z,w)$. If
  $C\subseteq\Omega_\varepsilon\times\Omega_\varepsilon^*$ is compact, then so is
  $\{(sz,s\zeta):0\le s\le1,\ (z,\zeta)\in C\}$, since $\Omega_\varepsilon$ and
  $\Omega_\varepsilon^*$ are balanced, and the uniform continuity of $h$ on this
  set shows that $h(t_nz,t_n\zeta)\to h(z,\zeta)$ uniformly on $C$. For
  item~\ref{i:exh:cc}, apply Theorem~\ref{t:radial-dilation} to the
  $\Mult(\cH_\ell)$-contractive tuple $T=M_z^\ell$ and use its identity
  entrywise: for $m\ge1$ and
  $P\in\CC[z_1,\dots,z_{\lcal{d}}]\otimes\CC^{m\times m}$,
  $P(t_nM_z^\ell)=(I_m\otimes V)^*\big(P(M_z^\ell)\otimes I_\HSF\big)(I_m\otimes V)$,
  so that
  \[
    \|P(t_nM_z^\ell)\|\le\|P(M_z^\ell)\otimes I_\HSF\|=\|P(M_z^\ell)\|.
  \]

  Now consider Definition~\ref{def:regular}. Items~\ref{i:regular:specT}
  and~\ref{i:regular:include} hold by hypothesis, item~\ref{i:regular:poly-cvx}
  holds because compact convex sets are polynomially convex, and
  item~\ref{i:regular:exh} was just verified. Finally,
  Theorem~\ref{t:radial-dilation} shows that $(r_n,\Omega_\varepsilon)$ has the
  dilation property, which gives item~\ref{i:regular:dilation}.
\end{proof}

The class of strongly regular pairs contains pairs with a vanishing kernel
$\ell$, to which Lemma~\ref{l:dilation-criterion} does not apply.

\begin{example}\rm
  \label{eg:strong-not-old}
  On $\DD$, let $k(z,w)=(1-z\ol{w})^{-1}$ and
  \[
    \ell(z,w)=\frac{1+2z\ol{w}}{1-z\ol{w}}=1+3\sum_{j\ge1}(z\ol{w})^j.
  \]
  The norms of $\cH_\ell$ and $H^2(\DD)$ are equivalent, so $M_z^\ell$ is similar
  to the unilateral shift and $\specT(M_z^\ell)=\ol{\DD}$. As
  $\ell/k=1+2z\ol{w}\succeq0$, Lemma~\ref{l:cert-cont-include} gives
  $\Mult(\cH_k)\subseteq\Mult(\cH_\ell)$. Since $(k,\ell)$ is a circularly
  symmetric pair, Corollary~\ref{c:allequiv} gives item~\ref{i:regular:include}
  of Definition~\ref{def:regular}, and, since $1/k$ is a polynomial,
  Corollary~\ref{c:radial-strong} shows that $(k,\ell)$ is strongly regular. As
  item~\ref{i:regular:include} implies the invariance hypothesis of
  Proposition~\ref{CPinbetween}, that proposition, applied with $s=k$, shows that
  $(k,\ell)$ is a BLH pair. However, $\ell(z,w)=0$ whenever $z\ol{w}=-1/2$.
  \qed
\end{example}

\section{The characterization} \label{mastersec}

The following theorem summarizes the main results of this paper. 

\begin{theorem}
 \label{DerSatz}
    The relations among the following conditions on a pair $(k, \ell)$ of reproducing kernels with $k$ nonvanishing,
    \begin{enumerate}[(a)]  \itemsep=6pt
        \item \label{i:weakAgler}  $(k, \ell)$ admits a weak Agler decomposition; 
        \item \label{i:strongAgler}  $(k, \ell)$ admits a strong Agler decomposition;
    \item \label{i:Leech}   $(k, \ell)$ is a complete pseudo-Leech pair; 
        \item \label{i:BLH}  $(k, \ell)$ is BLH,
    \end{enumerate}

\noindent are summarized in the diagram below.

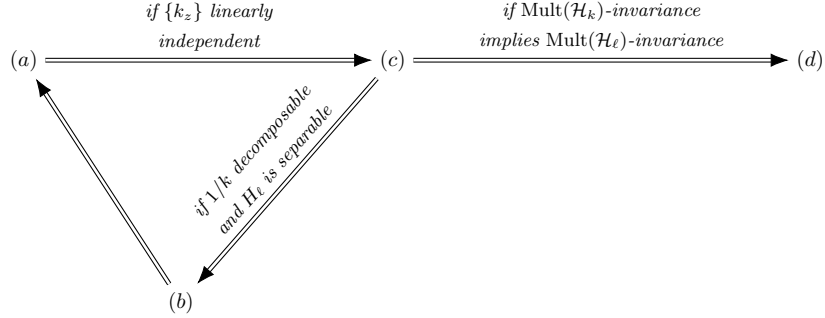
\begin{center}
\begin{tikzpicture}[
    scale=0.80,
    every node/.style={scale=0.80},
    >=Latex,
    implication/.style={
        double distance=1.0pt,
        line width=0.4pt,
        -{Latex[length=2.4mm,width=1.8mm]}
    }
]

\node (a) at (0,0) {$ (a) $};
\node (c) at (6.1,0) {$ (c) $};
\node (d) at (13.0,0) {$ (d) $};
\node (b) at (2.6,-4) {$ (b) $};

\draw[implication]
    (a) -- node[midway, above, font=\small, align=center, text width=3.8cm, yshift=1pt]
    {if $\{k_z\}$ linearly independent}
    (c);

\draw[implication]
    (c) -- node[midway, above, font=\small, align=center, text width=7.0cm, yshift=1pt]
    {if $\Mult(\mathcal{H}_k)$-invariance $\text{implies}$ $\Mult(\mathcal{H}_{\ell})$-invariance}
    (d);

\draw[implication]
    (c) -- node[midway, sloped, above, font=\small, align=center, text width=4.4cm]
    {if $k$ decomposable}
    (b);

\draw[implication]
    (b) -- (a);

\draw[implication]
    (a.north) -- ++(0,1.9) -- node[midway, above, font=\small, align=center, text width=7.0cm, yshift=1pt]
    {if $\Mult(\mathcal{H}_k)$-invariance $\text{implies}$ $\Mult(\mathcal{H}_{\ell})$-invariance}
    ($(d.north)+(0,1.9)$) -- (d.north);

\end{tikzpicture}
\end{center}

\noindent Suppose now that $k, \ell$ satisfy Assumption \ref{assume}. The relations between the conditions above and those below, 

\begin{enumerate}[(a)] \itemsep=6pt
\setcounter{enumi}{4}
 \item    \label{i:regularAg}   $(k, \ell)$ admits an Agler decomposition; 
           \item   \label{i:co-extpair}  $(k, \ell)$ is a co-extension pair, 
\end{enumerate}

\noindent are summarized in the diagram below.

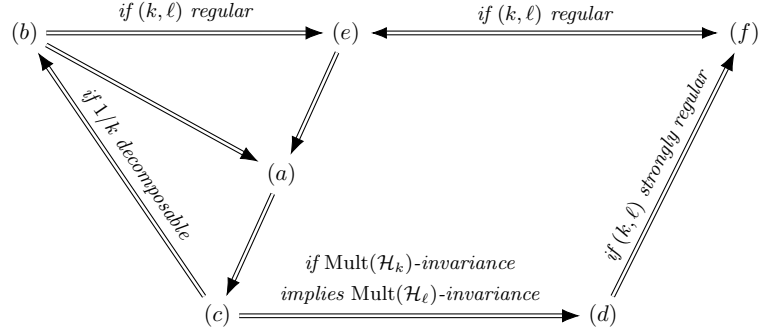
\begin{center}
\begin{tikzpicture}[
    scale=0.85,
    every node/.style={scale=0.85},
    >=Latex,
    implication/.style={
        double distance=1.0pt,
        line width=0.4pt,
        -{Latex[length=2.4mm,width=1.8mm]}
    }
]

\node (b) at (0,0) {$ (b) $};
\node (e) at (5.0,0) {$ (e) $};
\node (f) at (11.2,0) {$ (f) $};

\node (a) at (4,-2.2) {$ (a) $};
\node (c) at (3.0,-4.4) {$ (c) $};
\node (d) at (9.0,-4.4) {$ (d) $};


\draw[implication] (e) to (a);
\draw[implication] (a) to (c);

\draw[implication] (c) -- (b);
    
\draw[implication]
    (b) -- (a);
\draw[implication]
    (c) -- node[midway, above, font=\small, align=center, text width=7.0cm, yshift=1pt]
    {if $\Mult(\mathcal{H}_k)$-invariance $\text{implies}$ $\Mult(\mathcal{H}_{\ell})$-invariance}
    (d);

\draw[implication]
    (d) -- node[midway, sloped, above, font=\small, align=center, text width=4.2cm]
    {if $(k,\ell)$ strongly regular}
    (f);

\draw[double distance=1.0pt, line width=0.4pt, <->, shorten <=1pt, shorten >=1pt]
    (e.east) -- node[midway, above, font=\small]
    {if $ (k,\ell) \text{ regular} $}
    (f.west);

\end{tikzpicture}
\end{center}

\noindent In particular, if $(k, \ell)$ is strongly regular, 
then all items from (a) to (f) are equivalent. 
\end{theorem}

\begin{remark}
Recall that item~\ref{i:regular:include} of Definition~\ref{def:regular} holds under mild additional assumptions; see Remark~\ref{r:regular}(2), Proposition~\ref{allequiv} and Corollary~\ref{c:allequiv}.
\qed\end{remark}

\begin{proof}[Proof of Theorem \ref{DerSatz}]
Consider the first diagram. The implication \ref{i:Leech}$\Rightarrow$\ref{i:BLH} is Proposition \ref{LeechtoBLH}, and \ref{i:weakAgler}$\Rightarrow$\ref{i:BLH} is Proposition~\ref{gensuffBLH};   \ref{i:weakAgler}$\Rightarrow$\ref{i:Leech} is Theorem \ref{AglerdectoLeech}; and  \ref{i:strongAgler} $\Rightarrow$ \ref{i:weakAgler} is trivial. Finally, \ref{i:Leech} $\Rightarrow$ \ref{i:strongAgler} follows from Theorem \ref{Leech=wAg=sAg}. 
Now, consider the second diagram. The equivalence \ref{i:regularAg}$\Leftrightarrow$\ref{i:co-extpair} is Theorem~\ref{COEXTISAGLER}. The implication \ref{i:regularAg}$\Rightarrow$\ref{i:weakAgler} is trivial. The implications \ref{i:weakAgler}$\Rightarrow$\ref{i:Leech} and \ref{i:Leech}$\Rightarrow$\ref{i:BLH} are proved as in the first diagram; the linear independence of $\{k_z\}$ needed for the former follows from Assumption~\ref{assume}. The implication \ref{i:Leech}$\Rightarrow$\ref{i:strongAgler} follows from Theorem~\ref{Leech=wAg=sAg}, since $k$ is decomposable by Lemma~\ref{l:hol-decomp}. Finally, \ref{i:BLH}$\Rightarrow$\ref{i:co-extpair} follows from Theorem~\ref{BLHISCOEXT}.

If $(k,\ell)$ is strongly regular, then \ref{i:strongAgler}$\Rightarrow$\ref{i:regularAg} follows from \ref{i:strongAgler}$\Rightarrow$\allowbreak\ref{i:weakAgler}$\Rightarrow$\allowbreak\ref{i:Leech}$\Rightarrow$\allowbreak\ref{i:BLH}$\Rightarrow$\allowbreak\ref{i:co-extpair}$\Rightarrow$\allowbreak\ref{i:regularAg}; see also Remark~\ref{r:CPA}.

Under strong regularity, item~\ref{i:im:exact} of Theorem~\ref{intromain} implies \ref{i:strongAgler} trivially, and \ref{i:weakAgler} implies item~\ref{i:im:exact} by Corollary~\ref{strongAglertoAgler}; hence item~\ref{i:im:exact} is equivalent to the other items of Theorem~\ref{intromain}.
\end{proof}

\begin{remark}\rm
  \label{r:CPA}
  Strong regularity is more than the implication
  \ref{i:strongAgler}$\Rightarrow$\ref{i:regularAg} requires. Say that $\ell$ has
  \df{contractive polynomial approximation} if, for every $m\ge1$, every
  contractive multiplier of $\cH_\ell\otimes\CC^m$ is a strong-$*$ limit of
  contractive polynomial multipliers. If $k$ and $\ell$ satisfy
  Assumption~\ref{assume} and $\ell$ has contractive polynomial approximation,
  then \ref{i:strongAgler} implies \ref{i:regularAg}. Indeed, by
  \ref{i:strongAgler}$\Rightarrow$\ref{i:weakAgler} and
  Corollary~\ref{strongAglertoAgler}, $1/k=Q(I-\Phi\Phi^*)Q^*$ with
  $Q:\Omega\to\cB(\HSG,\CC)$ holomorphic and
  $\Phi\in\Mult(\cH_\ell\otimes\HSF,\cH_\ell\otimes\HSG)$ contractive. Compressing
  $\Phi$ to increasing finite-dimensional subspaces of $\HSF$ and $\HSG$, padded
  with zeros to equal dimension, expresses $1/k$ as a limit of functions
  $Q_m(I-\Phi_m\Phi_m^*)Q_m^*$, with $Q_m$ holomorphic and $\Phi_m$ a contractive
  multiplier of $\cH_\ell\otimes\CC^m$; the convergence is pointwise and locally
  bounded, hence in $\Her(\Omega)$. By contractive polynomial approximation, each
  $Q_m(I-\Phi_m\Phi_m^*)Q_m^*$ is, in the same way, a limit of functions of the
  form \eqref{aglerfact}, so it lies in $\mathcal{AG}(\ell)$. As $\mathcal{AG}(\ell)$
  is closed, $1/k\in\mathcal{AG}(\ell)$. Contractive polynomial approximation
  holds, for instance, when $\ell$ is circularly symmetric: the Fej\'er means of
  a contractive multiplier are contractive polynomial multipliers converging to
  it in the strong-$*$ topology.
  \qed
\end{remark}

\section{Examples}\label{Exampsec}
  This section contains seven examples. Examples~\ref{bidiskex} and~\ref{polydiskex} concern
  polydisks. Example~\ref{bidiskex} shows that a strongly regular pair $(k,\ell)$ on $\DD^2$ with
  $\Mult(\cH_\ell)=H^\infty(\DD^2)$ isometrically is a BLH pair if and only if
  $1/k=(1-z_1\ol{w_1})g_1+(1-z_2\ol{w_2})g_2$ for some holomorphic kernels $g_1,g_2$.
  Example~\ref{polydiskex} shows that the characterization of Example~\ref{bidiskex} does not extend to
  $\DD^{\lcal{d}}$ for $\lcal{d}\ge3$: for the Szeg\H{o} kernel $\ell_{\lcal{d}}$ of
  $\DD^{\lcal{d}}$, there is a strongly regular BLH pair $(k,\ell_{\lcal{d}})$ such that $1/k$ is
  not of the form
  $\sum_i(1-z_i\ol{w_i})g_i$ with holomorphic kernels $g_i$.

  In Example~\ref{symbidiskex}, $\ell_G$ is the reproducing kernel of the Hardy space of the
  symmetrized bidisk $G=\{(\lambda_1+\lambda_2,\lambda_1\lambda_2):\lambda\in\DD^2\}$; it is not a
  complete Pick kernel. If $k$ is a nonvanishing analytic kernel on $G$ that satisfies
  Assumption~\ref{assume} and $1/k(z,\ol w)$ is holomorphic on a neighborhood of
  $\ol G\times\ol G$, then $(k,\ell_G)$ is strongly regular, and $(k,\ell_G)$ is a BLH
  pair if and only if
  \[
    \frac{1}{k(z,w)}=Q(z)\Big[(1-p_1\ol{p_2})I-\tfrac12\big(s_1-p_1\ol{s_2}\big)U-\tfrac12\big(\ol{s_2}-s_1\ol{p_2}\big)U^*\Big]Q(w)^*
  \]
  for $z=(s_1,p_1)$ and $w=(s_2,p_2)$ in $G$, some unitary operator $U$ on a Hilbert space
  $\mathcal{K}$ and some holomorphic $Q:G\to\cB(\mathcal{K},\CC)$.

  Example~\ref{regularballex} concerns kernels $k,\ell$ on the ball $\mathbb{B}_{\lcal{d}}$ that
  are invariant under the unitary group and satisfy $\lim_n\ell_{n+1}/\ell_n=1$,
  $\ell/\mathfrak{s}\succeq0$ and $\Mult(\cH_k)\subseteq\Mult(\cH_\ell)$, where $\mathfrak{s}$ is
  the Drury--Arveson kernel. For such kernels, $(k,\ell)$ is a BLH pair if and only if
  $\mathfrak{s}/k\succeq0$, and, if $(k,\ell)$ is regular, if and only if every commuting row
  contraction is a $1/k$ contraction. Neither equivalence extends to kernels $k$ that are only
  circularly symmetric.

  Examples~\ref{eg:recursivel} and~\ref{eg:periodicl}, on $\DD$, and
  Example~\ref{bidiskfamilyex}, on $\DD^2$, give strongly regular BLH pairs
  $(k,\ell)$ of diagonal holomorphic kernels for which
  no complete Pick kernel $s$ satisfies $\ell/s\succeq0$ and $s/k\succeq0$. Thus the hypotheses
  of Proposition~\ref{CPinbetween} are not necessary for the BLH property, and these pairs are
  complete pseudo-Leech pairs but not complete Pick pairs. On $\DD$, such a pair requires
  $\|M_z^\ell\|>1$, and Example~\ref{eg:periodicl} shows that $\|M_z^\ell\|$ can be arbitrarily
  close to $1$, whereas in Example~\ref{bidiskfamilyex}, $\ell$ is the Szeg\H{o} kernel of
  $\DD^2$. See Subsection~\ref{ss:intro-CP} and the discussion at the beginning of
  Subsection~\ref{CPsubsec}.

\subsection{Polydisks}\label{ss:ex-polydisks}

\begin{example}\label{bidiskex}
  Assume $(k,\ell)$ is a strongly regular pair on $\DD^2$ such that
  $\Mult(\cH_\ell)=H^\infty(\DD^2)$ isometrically. Every holomorphic
  $\Phi:\DD^2\to\CC^{m\times m}$ with $\sup_{\DD^2}\|\Phi\|\le1$ satisfies
  \[
    I-\Phi(z)\Phi(w)^*=(1-z_1\ol{w_1})G_1(z,w)+(1-z_2\ol{w_2})G_2(z,w)
  \]
  for holomorphic $\CC^{m\times m}$-valued kernels $G_1,G_2$ \cite{Hellinger} (see also
  \cite{BickelKnese}). Let $\mathcal{S}_2$ denote the set of functions
  $(1-z_1\ol{w_1})g_1+(1-z_2\ol{w_2})g_2$, where $g_1,g_2$ are holomorphic kernels on $\DD^2$.
  Recall that $\mathcal{AG}(\ell)$ is the closure in $\Her(\DD^2)$ of the set of functions of
  the form \eqref{aglerfact}, where $P$ is a contractive polynomial multiplier of
  $\cH_\ell\otimes\CC^m$ and $Q:\DD^2\to\cB(\CC^m,\CC)$ is holomorphic
  (Definition~\ref{Aglerwedge}). Since $\sup_{\DD^2}\|P\|\le1$, each function of the form
  \eqref{aglerfact} lies in $\mathcal{S}_2$. Conversely, $z_1$ and $z_2$ are contractive
  multipliers of $\cH_\ell$, so each function $(1-z_j\ol{w_j})f(z)\ol{f(w)}$, with $f$
  holomorphic, is of the form \eqref{aglerfact}, with $m=1$, $P=z_j$ and $Q=f$; as every
  holomorphic kernel is a limit in $\Her(\DD^2)$ of finite sums of functions $f(z)\ol{f(w)}$,
  Lemma~\ref{closed-under-sums} gives $\mathcal{S}_2\subseteq\mathcal{AG}(\ell)$. Finally,
  $\mathcal{S}_2$ is closed in $\Her(\DD^2)$. Indeed, if
  $h_n=(1-z_1\ol{w_1})g_{1,n}+(1-z_2\ol{w_2})g_{2,n}$ converges to $h$ in $\Her(\DD^2)$, then
  $g_{j,n}(z,z)\le h_n(z,z)/(1-|z_j|^2)$, so, as $|g_{j,n}(z,w)|^2\le g_{j,n}(z,z)\,g_{j,n}(w,w)$,
  the kernels $g_{j,n}$ are uniformly bounded on compact subsets of $\DD^2\times\DD^2$; by
  Montel's theorem, a subsequence of $(g_{1,n},g_{2,n})$ converges in $\Her(\DD^2)$ to a pair
  of holomorphic kernels $(g_1,g_2)$, and $h=(1-z_1\ol{w_1})g_1+(1-z_2\ol{w_2})g_2$. Hence
  \begin{equation}\label{e:ex-AG-bidisk}
    \mathcal{AG}(\ell)=\mathcal{S}_2=\big\{(1-z_1\ol{w_1})g_1+(1-z_2\ol{w_2})g_2\ \big|\ g_1,g_2\text{ holomorphic kernels}\big\}.
  \end{equation}
  By Theorem~\ref{DerSatz} and \eqref{e:ex-AG-bidisk}, $(k,\ell)$ is a BLH pair if and only if
  $1/k=(1-z_1\ol{w_1})g_1+(1-z_2\ol{w_2})g_2$ for some holomorphic kernels $g_1,g_2$ on $\DD^2$.
  \qed
\end{example}

\begin{example}\label{polydiskex}
  Let $\lcal{d}\ge2$ and let $\ell_{\lcal{d}}(z,w)=\prod_{i=1}^{\lcal{d}}(1-z_i\ol{w_i})^{-1}$ be the
  Szeg\H{o} kernel of $\DD^{\lcal{d}}$, so that multiplier norms from $\cH_{\ell_{\lcal{d}}}\otimes\CC^n$ to
  $\cH_{\ell_{\lcal{d}}}\otimes\CC^m$ are supremum norms on $\DD^{\lcal{d}}$. Let $\mathcal{S}_{\lcal{d}}$ be
  the set of functions $\sum_{i=1}^{\lcal{d}}(1-z_i\ol{w_i})g_i$, where
  $g_1,\dots,g_{\lcal{d}}$ are holomorphic kernels on $\DD^{\lcal{d}}$. As in
  Example~\ref{bidiskex}, $\mathcal{S}_{\lcal{d}}$ is a closed subset of
  $\mathcal{AG}(\ell_{\lcal{d}})$, and $\mathcal{S}_2=\mathcal{AG}(\ell_2)$ by \eqref{e:ex-AG-bidisk}. If
  $h\in\mathcal{S}_{\lcal{d}}$ and $X$ is a $\lcal{d}$-tuple of commuting contractions with
  $\specT(X)\subset\DD^{\lcal{d}}$, then $h(X,X^*)\succeq0$: for $f$ holomorphic,
  Lemma~\ref{l:Her-calc} gives
  $\big((1-z_i\ol{w_i})f(z)\ol{f(w)}\big)(X,X^*)=f(X)(I-X_iX_i^*)f(X)^*$, and the general case
  follows by taking sums and limits, using item~\ref{i:her-calc:v} of Lemma~\ref{l:Her-calc}.
  For $\lcal{d}\ge3$ we construct a strongly regular BLH pair $(k,\ell_{\lcal{d}})$ with
  $1/k\notin\mathcal{S}_{\lcal{d}}$, so that the characterization of Example~\ref{bidiskex} does
  not extend to $\DD^{\lcal{d}}$.

  Let
  \[
    p_V(z_1,z_2,z_3)=z_1^2+z_2^2+z_3^2-2z_1z_2-2z_1z_3-2z_2z_3
  \]
  be the Kaijser--Varopoulos polynomial. Then $\sup_{\DD^3}|p_V|=5$
  \cite[Proposition~2]{Holbrook}, and there are commuting contractions
  $T_1,T_2,T_3\in\CC^{5\times5}$ with $\|p_V(T_1,T_2,T_3)\|>5$ \cite{Varopoulos}. Fix
  $0<\rho,t<1$ with $\rho t^2\|p_V(T_1,T_2,T_3)\|>5$ and $c>0$ with $\rho^2+\lcal{d}c^2\le1$,
  and let $k=(1-\Psi\Psi^*)^{-1}$, where
  \begin{equation}\label{e:ex-PsiV}
    \Psi(z)=\big[\,\Psi_0(z)\ \ \Psi_1(z)\ \cdots\ \Psi_{\lcal{d}}(z)\,\big]
    =\big[\,\tfrac{\rho}{5}\,p_V(z_1,z_2,z_3)\ \ cz_1\ \cdots\ cz_{\lcal{d}}\,\big].
  \end{equation}
  Since $\|\Psi(z)\|<1$ on $\DD^{\lcal{d}}$, $k$ is a nonvanishing complete Pick kernel and
  $1/k$ is a hereditary polynomial. The polynomials
  $\Psi^\alpha=\Psi_0^{\alpha_0}\cdots\Psi_{\lcal{d}}^{\alpha_{\lcal{d}}}$,
  $\alpha\in\NN^{\lcal{d}+1}$, span a dense subspace of $\cH_k$ \cite{PaulsenRagh}, and every
  monomial is a scalar multiple of one of them, so polynomials are dense in $\cH_k$; since
  $(1-\Psi\Psi^*)k=1$, each $z_i$ is a multiplier of $\cH_k$. As $\sup\|\Psi\|\le1$,
  $\ell_{\lcal{d}}/k=(1-\Psi\Psi^*)\ell_{\lcal{d}}\succeq0$, so
  $\Mult(\cH_k)\subseteq\Mult(\cH_{\ell_{\lcal{d}}})$ by Lemma~\ref{l:cert-cont-include}. The pair
  $(k,\ell_{\lcal{d}})$ is strongly regular by Corollary~\ref{c:radial-strong}:
  $\specT(M_z^{\ell_{\lcal{d}}})=\ol{\DD}^{\lcal{d}}$ (by Remark~\ref{r:assume-Omega-bounded}, as each
  $M_{z_j}^{\ell_{\lcal{d}}}$ is an isometry), $1/k$ is a polynomial, and, since $p_V$ is
  homogeneous, $(k,\ell_{\lcal{d}})$ is a circularly symmetric pair, so Corollary~\ref{c:allequiv}
  gives item~\ref{i:regular:include} of Definition~\ref{def:regular}. As this item implies the
  invariance hypothesis of Proposition~\ref{CPinbetween}, that proposition, applied with $s=k$,
  shows that $(k,\ell_{\lcal{d}})$ is a BLH pair. Also
  $1/k=Q(I-PP^*)Q^*\in\mathcal{AG}(\ell_{\lcal{d}})$, where $P$ is the $(\lcal{d}+1)\times(\lcal{d}+1)$
  matrix with first row $\Psi$ and all other rows zero, and $Q=\big[\,1\ \ 0\ \cdots\ 0\,\big]$.

  The $\lcal{d}$-tuple $T'=(tT_1,tT_2,tT_3,0,\dots,0)$ consists of commuting contractions, and
  $\specT(T')\subset\DD^{\lcal{d}}$. Let $A=p_V(tT_1,tT_2,tT_3)$, so that $A=t^2p_V(T_1,T_2,T_3)$
  as $p_V$ is homogeneous of degree two, and choose a unit vector $v$ with $\|A^*v\|=\|A\|$.
  Since $(1/k)(T',T'^*)=I-\frac{\rho^2}{25}AA^*-c^2\sum_iT_i'T_i'^*$,
  \[
    \big\langle (1/k)(T',T'^*)v,v\big\rangle\le1-\frac{\rho^2}{25}\|A^*v\|^2
    =1-\Big(\frac{\rho t^2}{5}\|p_V(T_1,T_2,T_3)\|\Big)^2<0,
  \]
  so $T'$ is not a $1/k$ contraction, and therefore $1/k\notin\mathcal{S}_{\lcal{d}}$.
  \qed
\end{example}

\subsection{The symmetrized bidisk}\label{ss:ex-symbidisk}

\begin{example}\label{symbidiskex}
  Let $\pi(\lambda)=(\lambda_1+\lambda_2,\lambda_1\lambda_2)$ for $\lambda\in\CC^2$, and let
  $G=\pi(\DD^2)$ and $\Gamma=\pi(\ol{\DD}^2)$ be the open and the closed \df{symmetrized bidisk}.
  The map $\pi$ is open, $G$ is a domain, and $\Gamma=\ol G$; moreover, $\Gamma$ is
  polynomially convex \cite[Lemma~2.1]{AYcommutant}. Points of $G$ are written $(s,p)$, the arguments of a hereditary
  function on $G$ are written $z=(s_1,p_1)$ and $w=(s_2,p_2)$, and $\delta_t(s,p)=(ts,t^2p)$, so
  that $\delta_t\circ\pi(\lambda)=\pi(t\lambda)$. For $j\ge1$ let $\chi_j$ be the polynomial with
  $\chi_j\circ\pi(\lambda)=\lambda_1^j+\lambda_2^j$. The reproducing kernel $\ell_G$ of the Hardy
  space of $G$ \cite{BDStoeplitz} is determined by
  \begin{equation}\label{e:ex-ellG-prod}
    \ell_G(\pi(\lambda),\pi(\mu))=\prod_{i,j=1}^{2}\frac{1}{1-\lambda_i\ol{\mu_j}}
    =\exp\Big(\sum_{j\ge1}\frac1j\,\chi_j(\pi(\lambda))\ol{\chi_j(\pi(\mu))}\Big),
  \end{equation}
  so $\ell_G$ is a nonvanishing positive kernel, and
  \begin{equation}\label{e:ex-1/ellG}
    \frac{1}{\ell_G(z,w)}=(1-p_1\ol{p_2})^2-(s_1-p_1\ol{s_2})(\ol{s_2}-s_1\ol{p_2}).
  \end{equation}
  As $\ell_G(\cdot,0)=1$ and the coefficient of $p_1^2\ol{p_2}^{\,2}$ in $1-1/\ell_G$ is $-1$,
  the kernel $1-1/\ell_G$ is not positive semidefinite, so $\ell_G$ is not a complete Pick
  kernel \cite{Pickbook}. The map
  $f\mapsto2^{-1/2}(\lambda_1-\lambda_2)(f\circ\pi)$ is a unitary operator from $\cH_{\ell_G}$
  onto the antisymmetric part of $H^2(\DD^2)$, mapping polynomials onto antisymmetric
  polynomials \cite{BDStoeplitz}. Hence polynomials are dense in $\cH_{\ell_G}$, and, since the
  multiplication operator $M_\Phi$ corresponds to multiplication by $\Phi\circ\pi$, for every
  bounded holomorphic $\Phi:G\to\cB(\HSF,\HSE)$,
  \begin{equation}\label{e:ex-G-ci}
    \|\Phi\|_{\Mult(\cH_{\ell_G}\otimes\HSF,\,\cH_{\ell_G}\otimes\HSE)}=\sup_{z\in G}\|\Phi(z)\| .
  \end{equation}
  In particular, $\ell_G$ satisfies Assumption~\ref{assume}, and, since
  $1=\ell_G(\cdot,0)\in\cH_{\ell_G}$, every multiplier of $\cH_{\ell_G}$ is holomorphic, so that
  $\Mult(\cH_{\ell_G})=H^\infty(G)$ isometrically.

  Let $k$ be a nonvanishing analytic kernel on $G$ that satisfies Assumption~\ref{assume} and
  \begin{equation}\label{e:ex-G-hyp}
    1/k(z,\ol w)\in\hol(W\times W)\quad\text{for an open set }W\supseteq\Gamma.
  \end{equation}
  Then $(k,\ell_G)$ is strongly regular, with the exhaustion described below. Indeed,
  $\specT(M_s,M_p)=\Gamma$ by \eqref{e:ex-G-ci},
  polynomial convexity and Remark~\ref{r:assume-Omega-bounded}, and $1/\ell_G$ is a hereditary
  polynomial by \eqref{e:ex-1/ellG}. With $r_n=\delta_{t_n}$, where $0<t_n<1$ increases to $1$,
  and $\Omega'=\pi(R\DD^2)$, where $R>1$ is such that $\Omega'\subseteq W$, items~(i)--(iii) of
  Definition~\ref{kexhaust} are immediate, item~\ref{i:exh:cc} follows from \eqref{e:ex-G-ci}, as
  $\delta_t(G)\subseteq G$, and item~\ref{i:annoying} follows from \eqref{e:ex-ellG-prod} as in
  Example~\ref{eg:certifiably-annoying}, since $\chi_j\circ\delta_t=t^j\chi_j$; thus
  $(\delta_{t_n},\Omega')$ is a strong exhaustion. Next, every
  multiplier $\varphi$ of $\cH_k$ is holomorphic, as $1\in\cH_k$, and satisfies
  $|\varphi(z)|\le\|\varphi\|_{\Mult(\cH_k)}$, as $k(z,z)>0$; hence
  $\Mult(\cH_k)\subseteq H^\infty(G)=\Mult(\cH_{\ell_G})$. The maps
  $(s,p)\mapsto(e^{i\theta}s,e^{2i\theta}p)$ preserve $G$ and $\ell_G$, so, as in the proof of
  \cite[Proposition~2.5]{Beurlingfactor}, the Fej\'er means of each
  $\varphi\in\Mult(\cH_{\ell_G})$ with respect to this action of the circle are polynomials of
  multiplier norm at most $\|\varphi\|_{\Mult(\cH_{\ell_G})}$ that converge to $\varphi$ in the
  strong operator topology (the Fourier coefficients of
  $\theta\mapsto\varphi(e^{i\theta}s,e^{2i\theta}p)$ are polynomials, as their compositions with
  $\pi$ are the homogeneous components of $\varphi\circ\pi$, which are symmetric polynomials).
  In particular, polynomials are weak-$*$ dense in $\Mult(\cH_{\ell_G})$, and
  Proposition~\ref{allequiv} gives item~\ref{i:regular:include} of
  Definition~\ref{def:regular}. Hence $(k,\ell_G)$ is a regular pair. Finally, $\ell_G$ is
  nonvanishing, $1/\ell_G$ is a hereditary polynomial and $(\delta_{t_n},\Omega')$ is a strong
  exhaustion, so Lemma~\ref{l:dilation-criterion} shows that $(\delta_{t_n},\Omega')$ has the
  dilation property.

  For a unitary operator $U$ on a Hilbert space $\mathcal{K}$ let
  \[
  \begin{aligned}
    \Phi_U(s,p)&=(2pU-sI)(2I-sU)^{-1},\\
    H_U(z,w)&=(1-p_1\ol{p_2})I-\tfrac12\big(s_1-p_1\ol{s_2}\big)U-\tfrac12\big(\ol{s_2}-s_1\ol{p_2}\big)U^*.
  \end{aligned}
  \]
  For $U=\omega\in\T$, the function $\Phi_\omega$ maps $G$ into $\DD$
  \cite[Theorem~2.1]{AYhyperbolic}; hence, by the spectral theorem, $\|\Phi_U(z)\|<1$ for
  $z\in G$. Since all operators involved are functions of $U$, a direct computation gives
  \begin{equation}\label{e:ex-HU}
    H_U(z,w)=\big(I-\tfrac{s_1}{2}U\big)\big(I-\Phi_U(z)\Phi_U(w)^*\big)\big(I-\tfrac{s_2}{2}U\big)^*.
  \end{equation}
  Then $(k,\ell_G)$ is a BLH pair if and only if there are a
  Hilbert space $\mathcal{K}$, a unitary operator $U$ on $\mathcal{K}$ and a holomorphic function
  $Q:G\to\cB(\mathcal{K},\CC)$ such that
  \begin{equation}\label{e:ex-G-exact}
    \frac{1}{k(z,w)}=Q(z)\,H_U(z,w)\,Q(w)^*,\qquad z,w\in G.
  \end{equation}

  Indeed, suppose that \eqref{e:ex-G-exact} holds. Then, by \eqref{e:ex-HU},
  $1/k=\widetilde Q(I-\Phi_U\Phi_U^*)\widetilde Q^*$ with
  $\widetilde Q(s,p)=Q(s,p)(I-\frac{s}{2}U)$, and $\Phi_U$ is a contractive multiplier of
  $\cH_{\ell_G}\otimes\mathcal{K}$ by \eqref{e:ex-G-ci}; so, for every
  $\Mult(\cH_{\ell_G})$-invariant subspace $\mathcal{M}$, Lemma~\ref{ellM/k} gives
  $(\ell_G)_{\mathcal{M}}/k\succeq0$, and $(k,\ell_G)$ is BLH by Proposition~\ref{l_m/k} and
  item~\ref{i:regular:include} of Definition~\ref{def:regular}. Conversely, suppose that
  $(k,\ell_G)$ is a BLH pair. By Theorem~\ref{DerSatz}, $(k,\ell_G)$ admits a weak Agler
  decomposition, so Corollary~\ref{strongAglertoAgler} gives $1/k=V(I-\Psi\Psi^*)V^*$, where
  $V:G\to\cB(\HSF,\CC)$ is holomorphic, $\HSF$ and $\mathcal{L}$ are separable, and
  $\Psi:G\to\cB(\mathcal{L},\HSF)$ is a contractive multiplier; by the same corollary, $\Psi$
  is holomorphic and $\|\Psi(z)\|\le1$ on $G$. By \cite[Theorem~2.5]{AYrealization}, every holomorphic function $F$ on $G$ with
  $\|F(z)\|\le1$ admits a Hilbert space $\mathcal{K}$, a unitary operator $U_0$ on
  $\mathcal{K}$ and a holomorphic function $u$ on $G$ such that
  $I-F(w)^*F(z)=u(w)^*\big(I-\Phi_{U_0}(w)^*\Phi_{U_0}(z)\big)u(z)$. Apply this to
  $F(z)=\Psi(\ol z)^*$, which is holomorphic on the conjugation-invariant domain $G$. Since
  $\Phi_{U_0}(\ol z)^*=\Phi_{U_0^*}(z)$, the unitary operator $U=U_0^*$ on $\mathcal{K}$ and the
  holomorphic function $v:G\to\cB(\mathcal{K},\HSF)$, $v(z)=u(\ol z)^*$, satisfy
  \[
    I-\Psi(z)\Psi(w)^*=v(z)\big(I-\Phi_U(z)\Phi_U(w)^*\big)v(w)^*.
  \]
  Hence \eqref{e:ex-HU} gives
  \eqref{e:ex-G-exact} with ${Q(s,p)=V(s,p)\,v(s,p)\,(I-\frac{s}{2}U)^{-1}}$.
  \qed
\end{example}

\subsection{The ball}\label{ss:ex-ball}

\begin{example}\label{regularballex}
  Let $k,\ell$ be kernels on the ball $\mathbb{B}_{\lcal{d}}$ that are invariant under the
  unitary group of $\CC^{\lcal{d}}$ and satisfy Assumption~\ref{assume}, with $k$ nonvanishing
  and $\Mult(\cH_k)\subseteq\Mult(\cH_\ell)$, and write
  \[
    \ell(z,w)=\sum_{n=0}^{\infty}\ell_n\langle z,w\rangle^n,\qquad
    \frac{1}{k(z,w)}=\sum_{n=0}^{\infty}b_n\langle z,w\rangle^n .
  \]
  Assume that $\lim_n\ell_{n+1}/\ell_n=1$. This assumption is fairly common in the study of
  spaces of functions on $\mathbb{B}_{\lcal{d}}$ whose kernels are invariant under the unitary
  group (see for example \cite[Section 4]{GreeneRichtSund} or
  \cite[Proposition 8.5]{HartzIsomorphism}), and it implies
  $\specT(M_z^\ell)=\ol{\mathbb{B}}_{\lcal{d}}$: the inclusion $\supseteq$ holds by
  Remark~\ref{r:assume-Omega-bounded}, and for $\subseteq$ it suffices, by invariance under the
  unitary group and the spectral mapping theorem, that $M_{z_1}^\ell$ have spectral radius at
  most $1$, which follows from $\lim_n\ell_{n+1}/\ell_n=1$, since the monomials are orthogonal
  in $\cH_\ell$ with $\|z^\alpha\|^2=\alpha!/(|\alpha|!\,\ell_{|\alpha|})$. Let
  $\mathfrak{s}(z,w)=(1-\langle z,w\rangle)^{-1}$ be the Drury--Arveson kernel and assume
  $\ell/\mathfrak{s}\succeq0$ (equivalently, $M_z^\ell$ is a row contraction). As
  $(k,\ell)$ is a circularly symmetric pair, Corollary~\ref{c:allequiv} shows that, for every
  Hilbert space $\HSE$, the shift-invariant, the $\Mult(\cH_\ell)$-invariant and the
  $\Mult(\cH_k)$-invariant subspaces of $\cH_\ell\otimes\HSE$ coincide.

  Suppose that $(k,\ell)$ is a BLH pair. Given a positive integer $m$, the subspace
  $\mathcal{M}=\{f\in\cH_\ell:\ \langle f,z^\alpha\rangle=0,\ |\alpha|\le m-1\}$ is
  shift-invariant and $\ell_{\mathcal{M}}=\sum_{n\ge m}\ell_n\langle z,w\rangle^n$, so
  $\ell_{\mathcal{M}}/k\succeq0$ by Proposition~\ref{l_m/k}, that is,
  $\sum_{n=0}^pb_{p-n}\ell_{m+n}\ge0$ for $p\ge0$. Dividing by $\ell_m$ and letting $m\to\infty$
  gives
  \begin{equation}\label{bcoeffs}
    \beta_p:=\sum_{n=0}^pb_{n}\ge0,\qquad p\ge0,
  \end{equation}
  that is, $\mathfrak{s}/k=\sum_p\beta_p\langle z,w\rangle^p\succeq0$. Conversely, by
  Proposition~\ref{CPinbetween}, $\ell/\mathfrak{s}\succeq0$ and $\mathfrak{s}/k\succeq0$ imply
  that $(k,\ell)$ is BLH. Hence $(k,\ell)$ is BLH if and only if $\mathfrak{s}/k\succeq0$.

  If, in addition, $(k,\ell)$ is regular, then $(k,\ell)$ is BLH if and only if every
  commuting row contraction is a $1/k$ contraction. Indeed, applying \eqref{vonNeumann} to the
  $\lcal{d}\times\lcal{d}$ matrix polynomial with first row
  $\big[\,z_1\ \cdots\ z_{\lcal{d}}\,\big]$ and all other rows zero, which is a
  contractive multiplier since $\ell/\mathfrak{s}\succeq0$, shows that every
  $\Mult(\cH_\ell)$-contraction is a commuting row contraction. Hence, if every commuting row
  contraction is a $1/k$ contraction, then $(k,\ell)$ is a co-extension pair, and so a BLH pair by
  Theorem~\ref{BLHISCOEXT}. Conversely, if $(k,\ell)$ is BLH, then $\mathfrak{s}/k\succeq0$;
  let $T$ be a commuting row contraction, so that $\specT(T)\subseteq\ol{\mathbb{B}}_{\lcal{d}}$,
  and let
  $\varphi_T(X)=\sum_{i=1}^{\lcal{d}}T_iXT_i^*$. As $1/k$ is holomorphic on a neighborhood of
  $\ol{\mathbb{B}}_{\lcal{d}}\times\ol{\mathbb{B}}_{\lcal{d}}$ (item~(i) of
  Definition~\ref{kexhaust}), the $b_n$ decay geometrically
  and $1/k(T,T^*)=\sum_nb_n\varphi_T^n(I)$, and summation by parts gives
  \[
    \sum_{n=0}^Nb_n\varphi_T^n(I)=\sum_{n=0}^{N-1}\beta_n\,\varphi_T^n\big(I-\varphi_T(I)\big)+\beta_N\,\varphi_T^N(I)\succeq0,
  \]
  by \eqref{bcoeffs}; letting $N\to\infty$ gives $1/k(T,T^*)\succeq0$.

  Neither equivalence persists if invariance of $k$ under the unitary group is weakened to
  circular symmetry. Let $\lcal{d}\ge2$
  and let $\ell=(1-\langle z,w\rangle)^{-\lcal{d}}$ be the Szeg\H{o} kernel of the ball, so that
  multiplier norms are supremum norms on $\mathbb{B}_{\lcal{d}}$,
  $\ell_n=\binom{n+\lcal{d}-1}{n}$ and
  $\ell/\mathfrak{s}=(1-\langle z,w\rangle)^{-(\lcal{d}-1)}\succeq0$. Fix $1/\sqrt2<\rho<1$, let
  $c=(1-\rho^2)^{1/2}$, and let $k=(1-\Psi\Psi^*)^{-1}$ with
  $\Psi(z)=\big[\,2\rho z_1z_2\ \ cz_1\ \cdots\ cz_{\lcal{d}}\,\big]$. As
  $4|z_1z_2|^2\le|z|^4$, we have $\|\Psi(z)\|<1$ on $\mathbb{B}_{\lcal{d}}$, and, as in
  Example~\ref{polydiskex}, $(k,\ell)$ is a circularly symmetric pair and a BLH pair, and it is
  strongly regular by Corollary~\ref{c:radial-strong}, since
  $\specT(M_z^\ell)=\ol{\mathbb{B}}_{\lcal{d}}$ and $1/k$ is a polynomial. However, $\Psi$ is not
  a contractive multiplier of the
  Drury--Arveson space $H^2_{\lcal{d}}=\cH_{\mathfrak{s}}$, since
  $\|2\rho z_1z_2\|_{\Mult(H^2_{\lcal{d}})}\ge\|2\rho z_1z_2\|_{H^2_{\lcal{d}}}=\sqrt2\,\rho>1$
  \cite{ArvesonIII}. Thus $\mathfrak{s}/k=(1-\Psi\Psi^*)\mathfrak{s}$ is not positive
  semi-definite, and, by Lemma~\ref{l:her-calc+} applied to $\mathfrak{s}$, the
  $\lcal{d}$-shift $M_z^{\mathfrak{s}}$, which is a row contraction, is not a $1/k$ contraction.
  \qed
\end{example}

\subsection{BLH pairs not covered by Proposition~\ref{CPinbetween}}\label{ss:ex-disk}

The following proposition records the disk criterion used in the introduction.

\begin{proposition}\label{p:disk-universal}
  Let $(k,\ell)$ be a strongly regular pair on $\DD$, and let
  $\mathfrak{s}_1(z,w)=(1-z\ol w)^{-1}$. If $(k,\ell)$ is a BLH pair, then
  $\mathfrak{s}_1/k\succeq0$. If, in addition, $M_z^\ell$ is a contraction, then
  $(k,\ell)$ is a BLH pair if and only if $\mathfrak{s}_1/k\succeq0$.
\end{proposition}

\begin{proof}
  Suppose $(k,\ell)$ is BLH. By Theorem~\ref{DerSatz}, $1/k\in\mathcal{AG}(\ell)$. A
  contractive polynomial multiplier $P$ of $\cH_\ell\otimes\CC^m$ satisfies
  $\sup_{\DD}\|P\|\le1$, so
  $I-P(z)P(w)^*=(1-z\ol w)G(z,w)$ for a positive $\CC^{m\times m}$-valued holomorphic
  kernel $G$ \cite{PaulsenRagh}. Multiplication on the left by $Q(z)$ and on the right by
  $Q(w)^*$ shows that every generator of $\mathcal{AG}(\ell)$ has the form
  $(1-z\ol w)g$ with $g\succeq0$. Since division by the nonvanishing function $1-z\ol w$
  is continuous in $\Her(\DD)$, passage to the defining limit gives
  $\mathfrak{s}_1/k\succeq0$.

  If $M_z^\ell$ is a contraction, then
  $\ell/\mathfrak{s}_1=(1-z\ol w)\ell\succeq0$. Thus the converse follows from
  Proposition~\ref{CPinbetween}, applied with $s=\mathfrak{s}_1$, because regularity
  supplies its invariance hypothesis.
\end{proof}

Consequently, a strongly regular BLH pair $(k,\ell)$ on $\DD$ for which no complete Pick
kernel $s$ satisfies $\ell/s\succeq0$ and $s/k\succeq0$ requires $\|M_z^\ell\|>1$.

  The examples of this subsection share the following two steps.

\begin{lemma}\label{l:ex-normalize}
  Let $k$ and $\ell$ be holomorphic kernels on a domain $\Omega\ni0$ such that $k$ is
  nonvanishing, $k(z,0)=1$ for $z\in\Omega$, and $1\in\cH_\ell$.
  \begin{enumerate}[(a)]\itemsep=4pt
    \item If $s$ is a nonvanishing complete Pick kernel with $\ell/s\succeq0$ and $s/k\succeq0$,
      then so is $\widetilde s(z,w)=s(0,0)s(z,w)/\big(s(z,0)s(0,w)\big)$, and
      $\widetilde s(z,0)=1$. If, moreover, $s(z,0)=1$, then $B:=1-1/s$ and $P:=s/k$ are
      holomorphic kernels on $\Omega$ with $B(z,0)=0$ and $P(z,0)=1$.
    \item If $(k,\ell)$ is a normalized diagonal holomorphic pair with strictly positive Taylor
      coefficients and no nonvanishing complete
      Pick kernel $s$ satisfies $\ell/s\succeq0$ and $s/k\succeq0$, then $(k,\ell)$ is not a
      complete Pick pair.
  \end{enumerate}
\end{lemma}

\begin{proof}
  (a) Since $s$ is nonvanishing, $s(0,0)>0$, and $\widetilde s(z,w)=s(z,w)/\big(f(z)\ol{f(w)}\big)$
  with $f(z)=s(z,0)s(0,0)^{-1/2}$. Rescaling by a nonvanishing function preserves the complete
  Pick property \cite{Pickbook}, so $\widetilde s$ is a nonvanishing complete Pick kernel;
  moreover $\ell/\widetilde s=(\ell/s)\,f(z)\ol{f(w)}\succeq0$,
  $\widetilde s/k=(s/k)/\big(f(z)\ol{f(w)}\big)\succeq0$, and $\widetilde s(z,0)=1$. If $s(z,0)=1$, then
  $B=1-1/s$ is a positive kernel with $B(z,0)=0$ \cite{Pickbook}. Writing
  $B(z,w)=\Psi(z)\Psi(w)^*$, Lemma~\ref{l:l/s} shows that $\Psi$ is a contractive multiplier of
  $\cH_\ell$; since $1\in\cH_\ell$, each function $\Psi(\cdot)v=M_\Psi(1\otimes v)$ lies in $\cH_\ell$, so
  $\Psi$ is holomorphic. Hence $B$,
  $s=1/(1-B)$ and $P=s/k$ are holomorphic, and $P(z,0)=s(z,0)/k(z,0)=1$.

  (b) By Theorem~\ref{CPpaircharact}, a normalized diagonal holomorphic complete Pick pair
  $(k,\ell)$ has a strong Shimorin certificate: a complete Pick kernel $s$ with $k=(1-h)s$ and
  $\ell=sg$ for positive kernels $h,g$. Then $\ell/s=g\succeq0$, and $s$ is nonvanishing
  because $k$ is. As $k(z,z)=(1-h(z,z))s(z,z)>0$ and $s(z,z)\ge0$, we have $h(z,z)<1$, so
  $|h(z,w)|^2\le h(z,z)h(w,w)<1$ and $s/k=1/(1-h)=\sum_{j\ge0}h^j\succeq0$.
\end{proof}

\begin{example}\label{eg:recursivel}
  Write $x=z\ol w$, and let $k$ and $\ell$ be the diagonal kernels on $\DD$ determined by
  \begin{equation}\label{e:ex-rec-k}
    \frac{1}{k}=\Big(1-\frac23x-\frac13x^2\Big)+\frac{x^3}{2}\Big(1-\frac13x-\frac23x^2\Big)
    =1-\frac23x-\frac13x^2+\frac12x^3-\frac16x^4-\frac13x^5
  \end{equation}
  and
  \begin{equation}\label{e:ex-rec-l}
    \ell=\frac{3(3-x)}{(1-x)(9-2x^2)}=\sum_{n\ge0}\ell_nx^n,\qquad
    \ell_{2j}=\frac67+\frac17\Big(\frac29\Big)^j,\quad
    \ell_{2j+1}=\frac67-\frac4{21}\Big(\frac29\Big)^j .
  \end{equation}
  Thus $2/3\le\ell_n\le1$ for all $n$, so $\cH_\ell$ is $H^2(\DD)$ with an equivalent norm,
  $M_z^\ell$ is similar to the unilateral shift, and $\specT(M_z^\ell)=\ol\DD$. Moreover,
  \begin{equation}\label{e:ex-rec-factors}
    \Big(1-\frac23x-\frac13x^2\Big)\ell=\frac{9-x^2}{9-2x^2},\qquad
    \Big(1-\frac13x-\frac23x^2\Big)\ell=1+\frac{3x}{9-2x^2},
  \end{equation}
  and both functions have nonnegative Taylor coefficients. Hence
  \[
    \Phi(z)=\big[\,\sqrt{2/3}\,z\ \ \sqrt{1/3}\,z^2\,\big]\quad\text{and}\quad
    \Phi'(z)=\big[\,\sqrt{1/3}\,z\ \ \sqrt{2/3}\,z^2\,\big]
  \]
  are contractive multipliers from $\cH_\ell\otimes\CC^2$ to $\cH_\ell$, and
  \eqref{e:ex-rec-k} reads
  \begin{equation}\label{e:ex-rec-cert}
    \begin{gathered}
      \frac{1}{k(z,w)}=Q(z)\big(I-\Theta(z)\Theta(w)^*\big)Q(w)^*,\\
      Q(z)=\big[\,1\ \ 2^{-1/2}z^3\,\big],\qquad
      \Theta=\operatorname{diag}(\Phi,\Phi').
    \end{gathered}
  \end{equation}
  In particular $\ell/k\succeq0$, so $\Mult(\cH_k)\subseteq\Mult(\cH_\ell)$ by
  Lemma~\ref{l:cert-cont-include}.

  The kernel $k$ is positive. Indeed, with
  $P_0(x)=\frac1{54}\big(13x^3+5x^4+3x^5+19x^6+14x^7\big)$, direct multiplication gives
  \[
    \frac1k\Big(1+\frac23x+\frac79x^2\Big)=1-P_0 .
  \]
  As $P_0$ has nonnegative coefficients and $P_0(1)=1$, we have $|P_0(x)|<1$ for $|x|<1$, and
  $k=\big(1+\frac23x+\frac79x^2\big)\sum_{j\ge0}P_0^j=\sum_n\kappa_nx^n$. Writing $n=3q+r$ with
  $r\in\{0,1,2\}$, the coefficient of $x^r$ in $1+\frac23x+\frac79x^2$ times
  $(13/54)^q$ is at most $\kappa_n$, so all $\kappa_n$ are positive, and $k$ is a nonvanishing
  normalized diagonal holomorphic kernel on $\DD$ in whose space polynomials are dense.
  Moreover, $1/k=(1-x)Q_0(x)/6$ with $Q_0(x)=6+2x+3x^3+2x^4$. If $Q_0(x_0)=0$ with
  $|x_0|\le1$, then $P_0(x_0)=1$, which forces $|x_0|=1$ and $x_0^3=x_0^4=\dots=x_0^7=1$, so
  $x_0=1$; but $Q_0(1)=13$. Hence $6/Q_0$ is holomorphic on a disk of radius greater than $1$,
  so $\kappa_n\to6/13$, and, as all $\kappa_n$ are positive, $z$ is a multiplier of $\cH_k$.

  The pair $(k,\ell)$ is strongly regular by Corollary~\ref{c:radial-strong}:
  $\specT(M_z^\ell)=\ol\DD$, $1/k$ is a polynomial, and $(k,\ell)$ is a circularly symmetric
  pair, so Corollary~\ref{c:allequiv} gives item~\ref{i:regular:include} of
  Definition~\ref{def:regular}. Adding two zero rows to $\Theta$ and two zero entries to $Q$
  turns \eqref{e:ex-rec-cert} into a weak Agler decomposition, so, by
  Corollary~\ref{c:AD+invariant}, $(k,\ell)$ is a BLH pair, and by
  Theorem~\ref{AglerdectoLeech} it is a complete pseudo-Leech
  pair (the kernel functions of $k$ are linearly independent, as $z$ is a multiplier of
  $\cH_k$).

  Suppose that $s$ is a nonvanishing complete Pick kernel with $\ell/s\succeq0$ and
  $s/k\succeq0$. By Lemma~\ref{l:ex-normalize}(a) we may assume $s(z,0)=1$, so that
  $B=1-1/s=\sum_{i,j\ge1}B_{ij}z^i\ol w^j$ and $P=s/k$ are holomorphic kernels with $P(z,0)=1$.
  Let $R=\ell/s=\ell(1-B)$, and write $P_{ij}$ and $R_{ij}$ for the coefficients of $P$ and
  $R$. The matrices $[B_{ij}]$, $[P_{ij}]$ and $[R_{ij}]$ are positive semidefinite, $P_{00}=1$ and $P_{i0}=0$ for $i\ge1$; since $\ell$
  is diagonal,
  \begin{equation}\label{e:ex-rec-Rnn}
    R_{nn}=\ell_n-\sum_{j=1}^n\ell_{n-j}B_{jj}\ge0,\qquad n\ge1 .
  \end{equation}
  We compare coefficients in $P=1/k+PB$. At $z\ol w$ this gives $P_{11}=-\frac23+B_{11}$, so
  $B_{11}\ge\frac23$, while $R_{11}=\frac23-B_{11}\ge0$; hence $B_{11}=\frac23$ and $P_{11}=0$.
  Then the first row of $[P_{ij}]$ vanishes, and comparing coefficients at $z\ol w^j$, $j\ge2$,
  gives $B_{1j}=P_{1j}=0$. At $z^2\ol w^2$ we get $P_{22}=-\frac13+B_{22}$, so
  $B_{22}\ge\frac13$, while \eqref{e:ex-rec-l} and \eqref{e:ex-rec-Rnn} give
  $0\le R_{33}=\frac{22}{27}-\frac{16}{27}-\frac23B_{22}-B_{33}$, so $B_{22}\le\frac13$; hence
  $B_{22}=\frac13$. By \eqref{e:ex-rec-factors}, the function $(1-\frac23x-\frac13x^2)\ell$ has
  only even powers of $x$, that is, $\ell_n=\frac23\ell_{n-1}+\frac13\ell_{n-2}$ for odd
  $n\ge3$. For odd $n\ge3$, \eqref{e:ex-rec-Rnn} therefore reads
  $R_{nn}=-\sum_{j=3}^n\ell_{n-j}B_{jj}\ge0$, and, as $\ell_m>0$ and $B_{jj}\ge0$, we get
  $B_{jj}=0$ for all $j\ge3$; the corresponding rows of $[B_{ij}]$ vanish. Thus
  $B=\frac23x+\frac13x^2$, $s=\big((1-x)(1+\frac x3)\big)^{-1}$, and
  \[
    \frac sk=\frac{Q_0(x)}{6+2x}=1+\frac{x^3(3+2x)}{2(3+x)}
  \]
  has coefficient $-\frac1{18}$ at $x^5$, a contradiction. Hence no complete Pick kernel $s$
  satisfies $\ell/s\succeq0$ and $s/k\succeq0$, and $(k,\ell)$ is not a complete Pick pair by
  Lemma~\ref{l:ex-normalize}(b).
  \qed
\end{example}

\begin{example}\label{eg:periodicl}
  Let $0<a<1$ and $0<c<a^2/2$, write $x=z\ol w$, and let
  \[
    \ell_a(z,w)=\frac{1+ax}{1-x^2},\qquad
    k_{a,c}(z,w)=\frac{1}{1-ax+cx^2-cx^4}.
  \]
  The coefficients of $\ell_a$ are $1$ at even powers of $x$ and $a$ at odd powers, so
  $\cH_{\ell_a}$ is $H^2(\DD)$ with the equivalent norm
  $\|\sum_nf_nz^n\|^2=\sum_n|f_{2n}|^2+a^{-1}\sum_n|f_{2n+1}|^2$; in particular
  $\|M_z^{\ell_a}\|=a^{-1/2}>1$, and $(M_z^{\ell_a})^2$ is an isometry.

  Let $u(x)=1+ax+(a^2-c)x^2+a(a^2-2c)x^3$ and
  \[
    v(x)=\big((a^2-c)^2+c(1-a^2)\big)x^4+ac(1-a^2+2c)x^5+c(a^2-c)x^6+ac(a^2-2c)x^7 .
  \]
  Direct multiplication gives $(1-ax+cx^2-cx^4)\,u=1-v$. The coefficients of $u$ and $v$ are
  positive and $v(1)=1-(1-a)u(1)<1$, so $k_{a,c}=u\sum_{j\ge0}v^j=\sum_n\kappa_nx^n$ is
  holomorphic for $|x|<\rho$, for some $\rho>1$. Here $\kappa_0,\dots,\kappa_3$ are the
  coefficients of $u$, and $\kappa_n=\sum_{j=4}^7v_j\kappa_{n-j}$ for $n\ge4$, where $v_j$ is the
  coefficient of $x^j$ in $v$ and $\kappa_m:=0$ for $m<0$; hence all $\kappa_n$ are positive, $k_{a,c}$ is a nonvanishing
  diagonal holomorphic kernel on $\DD$, and polynomials are dense in $\cH_{k_{a,c}}$. With
  $\varepsilon=\min_{1\le n\le7}\kappa_n/\kappa_{n-1}$, the recurrence and induction give
  $\kappa_n\ge\varepsilon\kappa_{n-1}$ for all $n\ge1$; that is,
  $(1-\varepsilon x)k_{a,c}\succeq0$, and $z$ is a multiplier of $\cH_{k_{a,c}}$.

  Since $(1-ax)\ell_a=(1-a^2x^2)/(1-x^2)\succeq0$ and $(1-x^2)\ell_a=1+ax\succeq0$, the
  functions $\sqrt a\,z$ and $z^2$ are contractive multipliers of $\cH_{\ell_a}$, and
  \begin{equation}\label{e:ex-periodic-cert}
    \frac{1}{k_{a,c}(z,w)}=(1-az\ol w)+c\,z^2\big(1-z^2\ol w^2\big)\ol w^2
    =Q(z)\big(I-\Phi(z)\Phi(w)^*\big)Q(w)^*,
  \end{equation}
  with $Q(z)=\big[\,1\ \ \sqrt c\,z^2\,\big]$ and
  $\Phi(z)=\operatorname{diag}(\sqrt a\,z,\,z^2)$. In particular
  $\ell_a/k_{a,c}\succeq0$, so $\Mult(\cH_{k_{a,c}})\subseteq\Mult(\cH_{\ell_a})$ by
  Lemma~\ref{l:cert-cont-include}.

  The pair $(k_{a,c},\ell_a)$ is strongly regular by Corollary~\ref{c:radial-strong}:
  $M_z^{\ell_a}$ is similar to the unilateral shift, so $\specT(M_z^{\ell_a})=\ol\DD$;
  $1/k_{a,c}$ is a polynomial; and $(k_{a,c},\ell_a)$ is a circularly symmetric pair, so
  Corollary~\ref{c:allequiv} gives item~\ref{i:regular:include} of
  Definition~\ref{def:regular}. By
  Corollary~\ref{c:AD+invariant} and \eqref{e:ex-periodic-cert}, $(k_{a,c},\ell_a)$ is a BLH
  pair, and by Theorem~\ref{AglerdectoLeech} it is a complete pseudo-Leech pair (the kernel
  functions of $k_{a,c}$ are linearly independent, as $z$ is a multiplier of $\cH_{k_{a,c}}$).

  Suppose that $s$ is a nonvanishing complete Pick kernel with $\ell_a/s\succeq0$ and
  $s/k_{a,c}\succeq0$. By Lemma~\ref{l:ex-normalize}(a) we may assume $s(z,0)=1$, so that
  $B=1-1/s=\sum_{i,j\ge1}B_{ij}z^i\ol w^j$ and $P=s/k_{a,c}$ are holomorphic kernels with
  $P(z,0)=1$. Let $R=\ell_a/s=\ell_a(1-B)$, and write $P_{ij}$ and $R_{ij}$ for the coefficients of $P$
  and $R$. The
  matrices $[B_{ij}]$, $[P_{ij}]$ and $[R_{ij}]$ are positive semidefinite, $P_{00}=1$, and
  $P_{i0}=0$ for $i\ge1$. Comparing the coefficients of $z\ol w$ in $1/k_{a,c}=P(1-B)$ gives
  $-a=P_{11}-B_{11}$, so $B_{11}\ge a$, while $R_{11}=a-B_{11}\ge0$; hence $B_{11}=a$. For odd
  $N\ge3$, as the coefficient of $x^{N-1}$ in $\ell_a$ is $1$,
  \[
    R_{NN}=a-\sum_{j=1}^{N}[x^{N-j}]\ell_a\,B_{jj}=-\sum_{j=2}^{N}[x^{N-j}]\ell_a\,B_{jj}\ge0 ,
  \]
  where $[x^n]\ell_a$ denotes the coefficient of $x^n$ in $\ell_a$. As these coefficients are
  positive and $B_{jj}\ge0$, it follows that $B_{jj}=0$ for $j\ge2$, and then the corresponding
  rows of $[B_{ij}]$ vanish. Thus $B=ax$ and $s=(1-ax)^{-1}$. But
  \[
    \frac{s}{k_{a,c}}=1+\frac{cx^2(1-x^2)}{1-ax}
  \]
  has coefficient $c(a^2-1)<0$ at $x^4$, a contradiction. Hence no complete Pick kernel $s$
  satisfies $\ell_a/s\succeq0$ and $s/k_{a,c}\succeq0$, and $(k_{a,c},\ell_a)$ is not a complete
  Pick pair by Lemma~\ref{l:ex-normalize}(b). Since $\|M_z^{\ell_a}\|=a^{-1/2}$ and $a$ can be
  taken arbitrarily close to $1$, the inequality $\|M_z^\ell\|>1$ obtained at the beginning of
  Subsection~\ref{ss:ex-disk} cannot be improved.
  \qed
\end{example}

\begin{example}\label{bidiskfamilyex}
  Let $\ell(z,w)=(1-z_1\ol{w_1})^{-1}(1-z_2\ol{w_2})^{-1}$ be the Szeg\H{o} kernel of $\DD^2$,
  write $x=z_1\ol{w_1}$ and $y=z_2\ol{w_2}$, and for $0<\tau\le\tau_*:=3-2\sqrt2$ let $k_\tau$
  be determined by
  \begin{equation}\label{e:ex-ktau}
    \frac{1}{k_\tau}=(1-x)\,\frac{1}{2(1-\tau y)}+(1-y)\,\frac{1}{2(1-\tau x)}.
  \end{equation}
  We show that $(k_\tau,\ell)$ is a strongly regular BLH pair of diagonal holomorphic kernels
  and that no nonvanishing complete Pick kernel $s$ satisfies $\ell/s\succeq0$ and $s/k_\tau\succeq0$. (For
  $\tau=0$, the kernel $k_0=\big(1-\frac{x+y}{2}\big)^{-1}$ is a complete Pick kernel. The bound
  $\tau_*$ is sharp: for $\tau_*<\tau<1$, the restriction
  $(1-\tau x)/\big(1-\frac{1+\tau}{2}x+\frac\tau2x^2\big)$ of $k_\tau$ to $z_2=w_2=0$ has a
  denominator with non-real roots, so, by Pringsheim's theorem, it has a negative Taylor
  coefficient, and $k_\tau$ is not positive semidefinite.)

  Combining the fractions in \eqref{e:ex-ktau} gives
  \begin{equation}\label{e:ex-ktau-rational}
    k_\tau=\frac{(1-\tau x)(1-\tau y)}{1-\frac{1+\tau}{2}(x+y)+\frac{\tau}{2}(x^2+y^2)}.
  \end{equation}
  Let $\gamma=\frac{1-\tau}{2}$ and let $\alpha\ge\beta$ be the numbers
  $\frac14\big(1+\tau\pm\sqrt{1-6\tau+\tau^2}\big)$, so that
  $\alpha+\beta=\frac{1+\tau}{2}$, $\alpha\beta=\frac\tau2$, $\alpha^2+\beta^2=\gamma^2$ and
  $\beta>\tau$. With $A_1=1-\alpha x-\beta y$ and $A_2=1-\beta x-\alpha y$, the
  denominator in \eqref{e:ex-ktau-rational} equals $A_1A_2-\gamma^2xy$, so
  \[
    k_\tau=F_1F_2\sum_{n\ge0}\Big(\frac{\gamma^2xy}{A_1A_2}\Big)^n,
  \]
  where
  \[
    F_1=\frac{1-\tau x}{A_1}=1+\frac{(\alpha-\tau)x+\beta y}{A_1},\qquad
    F_2=\frac{1-\tau y}{A_2}=1+\frac{\beta x+(\alpha-\tau)y}{A_2}.
  \]
  For $|x|,|y|<1$ we have $|A_1|,|A_2|>1-\alpha-\beta=\gamma$, so the series converges; every
  factor has nonnegative Taylor coefficients in $x,y$, and those of $F_1,F_2$ are positive.
  Hence $k_\tau=\sum_{m,n}c_{mn}x^my^n$ with all $c_{mn}>0$: $k_\tau$ is a diagonal holomorphic
  kernel on $\DD^2$, nonvanishing by \eqref{e:ex-ktau-rational}, and polynomials are dense in
  $\cH_{k_\tau}$. With ${\eta=\alpha-\tau>0}$, one has $(1-\eta x)F_1=1+\big(\beta y+\tau\eta x^2\big)/A_1$, so
  $(1-\eta x)k_\tau\succeq0$; similarly $(1-\eta y)k_\tau\succeq0$, and $z_1,z_2$ are
  multipliers of $\cH_{k_\tau}$.

  The pair $(k_\tau,\ell)$ is strongly regular by Corollary~\ref{c:radial-strong}:
  $\specT(M_z^\ell)=\ol{\DD}^2$; $1/k_\tau(z,\ol w)$ is holomorphic on
  $\varrho\DD^2\times\varrho\DD^2$ for $1<\varrho<\tau^{-1/2}$; and
  \[
    \frac{\ell}{k_\tau}=\frac12\Big(\frac{1}{(1-x)(1-\tau x)}+\frac{1}{(1-y)(1-\tau y)}\Big)\succeq0,
  \]
  so $\Mult(\cH_{k_\tau})\subseteq\Mult(\cH_\ell)$ by Lemma~\ref{l:cert-cont-include}, and, since
  $(k_\tau,\ell)$ is a circularly symmetric pair, Corollary~\ref{c:allequiv} gives
  item~\ref{i:regular:include} of Definition~\ref{def:regular}.
  By \eqref{e:ex-ktau} and Example~\ref{bidiskex}, $1/k_\tau\in\mathcal{S}_2\subseteq\mathcal{AG}(\ell)$;
  thus $(k_\tau,\ell)$ admits an Agler decomposition, hence a weak one, and it is a BLH pair. By
  Theorem~\ref{AglerdectoLeech}, it is also a complete pseudo-Leech
  pair (the kernel functions of $k_\tau$ are linearly independent, as $z_1,z_2$ are multipliers
  of $\cH_{k_\tau}$).

  Now suppose that a nonvanishing complete Pick kernel $s$ satisfies ${\ell/s\succeq0}$ and
  ${s/k_\tau\succeq0}$. By Lemma~\ref{l:ex-normalize}(a) we may assume $s(z,0)=1$, so that
  $B:=1-1/s$ and $P:=s/k_\tau$ are holomorphic kernels, $B(z,0)=0$, and
  \begin{equation}\label{e:ex-PB}
    P=\frac{1}{k_\tau}+BP.
  \end{equation}
  Write $B_{\mu\nu}$ and $P_{\mu\nu}$, for $\mu,\nu\in\NN^2$, for the coefficients of
  $z^\mu\ol w^\nu$; the matrices $[B_{\mu\nu}]$ and $[P_{\mu\nu}]$ are positive semidefinite,
  $B_{\mu0}=0$, $P_{00}=1$, and $P_{\mu0}=0$ for $\mu\ne0$. As every coefficient of $\ell$
  equals $1$, the coefficient of $z^\mu\ol w^\mu$ in $\ell/s=\ell(1-B)$ is
  $1-\sum_{0<\nu\le\mu}B_{\nu\nu}\ge0$, whence
  \begin{equation}\label{e:ex-budget}
    \sum_{\nu\ne0}B_{\nu\nu}\le1.
  \end{equation}
  Expanding \eqref{e:ex-ktau} in powers of $x$ and $y$ gives
  \begin{equation}\label{e:ex-ktau-expand}
    \frac{1}{k_\tau}=1-\gamma(x+y)-\tau xy+R,\qquad
    R=\frac12\sum_{n\ge2}\tau^n\big(x^n+y^n-x^ny-xy^n\big).
  \end{equation}
  Let $e_1=(1,0)$, $e_2=(0,1)$, $\iota=(1,1)$, $b_j=B_{e_je_j}$, $u=B_{e_1e_2}$ and
  $d=B_{\iota\iota}$. By \eqref{e:ex-ktau-expand}, comparing the coefficients of
  $z^{e_j}\ol w^{e_j}$, $z^{e_1}\ol w^{e_2}$ and $z^\iota\ol w^\iota$ in \eqref{e:ex-PB} gives
  \[
    P_{e_je_j}=b_j-\gamma,\qquad P_{e_1e_2}=u,\qquad
    P_{\iota\iota}=d-\tau-\gamma(b_1+b_2)+2b_1b_2+2|u|^2.
  \]

  We show that $B=\gamma(x+y)+\tau xy$. Let $\varepsilon_j=b_j-\gamma\ge0$ and
  $\delta=\varepsilon_1+\varepsilon_2$. As $|u|^2\le\varepsilon_1\varepsilon_2\le\delta^2/4$,
  $P_{\iota\iota}\ge0$ gives $d\ge\tau-\gamma\delta-\delta^2$, while \eqref{e:ex-budget} gives
  $0\le d\le1-2\gamma-\delta=\tau-\delta$. Thus $\delta\le\tau$ and
  $\delta(1-\gamma-\delta)\le0$, and since $1-\gamma-\tau=\gamma>0$, it follows that
  $\delta=0$, so $b_1=b_2=\gamma$, $u=0$ and $d=\tau$. Then equality holds in
  \eqref{e:ex-budget}, so, for $\nu\notin\{e_1,e_2,\iota\}$, $B_{\nu\nu}=0$ and the $\nu$-th row
  of $[B_{\mu\nu}]$ vanishes. Finally, the $e_j$-th row of $[P_{\mu\nu}]$ vanishes, as
  $P_{e_je_j}=0$, and comparing the coefficients of $z^{e_j}\ol w^\iota$ in \eqref{e:ex-PB}
  gives $B_{e_j\iota}=P_{e_j\iota}=0$.

  Hence, by \eqref{e:ex-ktau-expand}, $P=(1/k_\tau)/(1-B)=1+R/(1-B)$. The coefficient of
  $x^2y$ in $R/(1-B)$ comes from the terms $-\frac{\tau^2}{2}x^2y$ and $\frac{\tau^2}{2}x^2$ of
  $R$ and the term $\gamma y$ of $1/(1-B)$, and equals
  $-\frac{\tau^2}{2}+\gamma\frac{\tau^2}{2}=-\frac{\tau^2(1+\tau)}{4}<0$, contradicting
  $P\succeq0$. Hence no such $s$ exists, and $(k_\tau,\ell)$ is not a complete Pick pair by
  Lemma~\ref{l:ex-normalize}(b).
  \qed
\end{example}

\section{Questions and problems} \label{questions}

As the title suggests, we list here several questions and a problem.

\subsection{The dilation property}\label{ss:q-dilation}
The proof of Theorem~\ref{BLHISCOEXT} uses strong regularity only through the
dilation property. That property cannot be asked of a $\Mult(\cH_\ell)$-contractive
tuple $T$ itself, even when $\specT(T)\subseteq\Omega$. Let
$\ell(z,w)=(1-z\ol{w})^{-2}$ be the Bergman kernel of $\DD$, so that $M_z^\ell$ is
the Bergman shift and, for each $m$, the norm of $\Mult(\cH_\ell\otimes\CC^m)$ is
the supremum norm over $\DD$. Let
$J=\left(\begin{smallmatrix}0&1\\0&0\end{smallmatrix}\right)$. Since $J$ is a
contraction, the Sz.-Nagy dilation theorem shows that $J$ is
$\Mult(\cH_\ell)$-contractive, and $\sigma(J)=\{0\}$. The monomials are orthogonal
in $\cH_\ell$ with $\|z^j\|^2=1/(j+1)$, so
\begin{equation}\label{e:bergman-strict}
  \|(M_z^\ell\otimes I_\HSF)x\|<\|x\|
\end{equation}
for every Hilbert space $\HSF$ and every non-zero $x\in\cH_\ell\otimes\HSF$. Hence, if
$V:\CC^2\to\cH_\ell\otimes\HSF$ were an isometry with
$J=V^*(M_z^\ell\otimes I_\HSF)V$, then
$1=\|Je_2\|\le\|(M_z^\ell\otimes I_\HSF)Ve_2\|<1$, a contradiction.

On a balanced domain, the radial maps $r_n(z)=t_nz$ are the natural candidates for
an exhaustion. When $\Omega$ is moreover convex and $\ell$ is circularly
symmetric, they give an exhaustion with the dilation property
(Corollary~\ref{c:radial-strong}). This leads to the following question.

\begin{question}\label{q:dilation-scaling}
  Let $\ell$ be a kernel on a bounded balanced domain
  $\Omega\subseteq\CC^{\lcal{d}}$ that satisfies Assumption~\ref{assume}, with
  $\ol{\Omega}$ polynomially convex and $\specT(M_z^\ell)=\ol{\Omega}$. Let $T$ be
  a $\Mult(\cH_\ell)$-contractive tuple on a Hilbert space $H$ and $0<t<1$. Is
  there an isometry $V:H\to\cH_\ell\otimes\HSF$, for some Hilbert space $\HSF$,
  such that $p(tT)=V^*\big(p(M_z^\ell)\otimes I_\HSF\big)V$ for all polynomials
  $p$?
\end{question}

\noindent On domains that are not balanced, multiplication by $t$ is not the right
operation; for instance, on $\Omega=2+\DD$ the map $z\mapsto tz$ does not map
$\ol{\Omega}$ into $\Omega$ for any $0<t<1$. Let $\Omega=2+\DD$ and $\ell(z,w)=\big(1-(z-2)(\ol{w}-2)\big)^{-2}$, so
that $M_z^\ell-2I$ is unitarily equivalent to the Bergman shift. The tuples
$2I+J$ and $T=\frac34(2I+J)$ are $\Mult(\cH_\ell)$-contractive, since $J$ and
$T-2I=-\frac12I+\frac34J$ are contractions, and $\specT(T)=\{3/2\}\subseteq\Omega$.
If $V$ were an isometry with $T=V^*(M_z^\ell\otimes I_\HSF)V$, then
$T-2I=V^*\big((M_z^\ell-2I)\otimes I_\HSF\big)V$; as $\|T-2I\|=1$ is attained at a
unit vector, this contradicts \eqref{e:bergman-strict} for the operator
$M_z^\ell-2I$, which is unitarily equivalent to the Bergman shift. Thus $T=tT'$
with $t=3/4$, where $T'=2I+J$ is $\Mult(\cH_\ell)$-contractive, and
$\specT(T)\subseteq\Omega$, and yet $T$ has no such dilation. In general, the
dilation property is asked of an exhaustion.

\begin{question}\label{q:dilation-exhaustion}
  Does every regular pair admit an exhaustion with the dilation property?
\end{question}

\noindent Corollary~\ref{c:radial-strong} answers
Question~\ref{q:dilation-exhaustion} affirmatively when $\Omega$ is a bounded
convex balanced domain and $\ell$ is circularly symmetric.

\subsection{The pseudo-Pick property}\label{ss:q-pseudo-Pick}

Definition~\ref{Leechdef} motivates the following alternative Pick property for a pair of reproducing kernels.

\begin{definition}
A pair $(k, \ell)$ on a set $X$ is a \textit{complete pseudo-Pick} pair if, whenever positive integers $n$ and $m$, points  $z_1, \dots, z_n\in X$ and  matrices $W_1, \dots, W_n\in M_m(\CC)$ satisfy
\[\big[(I-W_iW_j^*)k(z_i,z_j)\big]_{i,j=1}^n\succeq0,\]
there is a contractive multiplier $\Phi\in\Mult(\cH_\ell\otimes\CC^m)$ with
$\Phi(z_i)=W_i$ for all $i$.
\qed\end{definition}
Clearly, complete pseudo-Leech implies complete pseudo-Pick. Also, $(k, \ell)$ being complete pseudo-Pick implies the same for $(k, \ell_{\mathcal{M}}),$ for any $\Mult(\mathcal{H}_{\ell})$-invariant $\mathcal{M}.$ Thus, in view of Proposition \ref{l_m/k}, to prove that complete pseudo-Pick implies BLH for pairs with $k$ nonvanishing in which every $\Mult(\cH_k)$-invariant subspace is $\Mult(\cH_\ell)$-invariant (for instance, regular pairs), it suffices to show that complete pseudo-Pick implies $\ell/k\succeq 0$. This implication
fails. Ma and Zhang \cite{MaZhang} found a complete pseudo-Pick pair on a three-point set, with
$k$ nonvanishing of rank two, for which $\ell/k$ is not positive semidefinite; there
$\Mult(\cH_k)$ consists of the constant functions, so the invariance hypothesis above fails. In
the following example, $k$ is positive definite, so that $\Mult(\cH_k)=\Mult(\cH_\ell)$ consists
of all functions on $X$ and the invariance hypothesis holds.

\begin{example}\label{eg:pseudo-Pick}
  Let $X=\{x_1,x_2,x_3\}$, and let $k$ and $\ell$ be the kernels on $X$ whose Gram matrices
  $K=\big[k(x_i,x_j)\big]_{i,j=1}^3$ and $L=\big[\ell(x_i,x_j)\big]_{i,j=1}^3$ are
  \begin{gather*}
    K=\frac{5}{1056}\begin{pmatrix}848&1820&980\\1820&4557&2975\\980&2975&2597\end{pmatrix},\qquad
    L=\begin{pmatrix}1&1&1\\1&5/3&3/4\\1&3/4&5/2\end{pmatrix}.
  \end{gather*}
  Let $n_1=(1,0)$, $n_2=(4/5,3/5)$ and $n_3=(-4/5,3/5)$ in $\CC^2$. Then
  \[
    K^{-1}=\frac{1}{315}\begin{pmatrix}1211&-735&385\\-735&504&-300\\385&-300&224\end{pmatrix},
  \]
  and $(L^{-1})_{ij}=(K^{-1})_{ij}\,\langle n_j,n_i\rangle$ for $1\le i,j\le3$.
  The matrices $K$ and $L$ are positive definite with positive entries; in particular $k$ is
  nonvanishing. Moreover, $\ell(x_i,x_j)=1/(1-B_{ij})$, where $B$ is the positive
  semidefinite matrix with rows $(0,0,0)$, $(0,2/5,-1/3)$ and $(0,-1/3,3/5)$, so $\ell$ is a CP
  kernel by \eqref{univCP1}.

  As the kernel functions of $k$ form a basis of $\cH_k$, every function on $X$ is a multiplier
  of $\cH_k$, a function $W:X\to M_m(\CC)$ is a contractive multiplier of $\cH_k\otimes\CC^m$ if
  and only if $\big[(I-W(x_i)W(x_j)^*)k(x_i,x_j)\big]_{i,j=1}^3\succeq0$, and
  $\|f\|^2=\sum_{i,j}(K^{-1})_{ij}\langle f(x_j),f(x_i)\rangle$ for $f\in\cH_k\otimes\HSE$; the
  same holds for $\ell$ and $L$. Fix $m\ge1$ and define
  $V:\cH_\ell\otimes\CC^m\to\cH_k\otimes\CC^2\otimes\CC^m$ by $(Vg)(x_j)=n_j\otimes g(x_j)$. By the
  relation between $K^{-1}$ and $L^{-1}$, $V$ is an isometry. If $W:X\to M_m(\CC)$, then
  $V(Wg)=(I_{\CC^2}\otimes W)Vg$, so multiplication by $W$ on $\cH_\ell\otimes\CC^m$ equals
  $V^*YV$, where $Y$ is multiplication by $I_{\CC^2}\otimes W$ on
  $\cH_k\otimes\CC^2\otimes\CC^m$, and $\|Y\|$ is the norm of multiplication by $W$ on
  $\cH_k\otimes\CC^m$. Hence every contractive multiplier of $\cH_k\otimes\CC^m$ is a contractive
  multiplier of $\cH_\ell\otimes\CC^m$.

  We show that $(k,\ell)$ is complete pseudo-Pick. Let $z_1,\dots,z_n$ be points of $X$, and
  let $W_1,\dots,W_n\in M_m(\CC)$ satisfy
  \[
    P=\big[(I-W_iW_j^*)k(z_i,z_j)\big]_{i,j=1}^n\succeq0.
  \]
  If
  $z_i=z_j$ with $i\ne j$, then evaluating the quadratic form of $P$ at the vector with entries
  $u$ in position $i$, $-u$ in position $j$ and $0$ elsewhere gives
  $-k(z_i,z_i)\|W_i^*u-W_j^*u\|^2\ge0$ for every $u\in\CC^m$, so $W_i=W_j$. We may therefore
  assume that $z_1,\dots,z_n$ are distinct. If $n=1$, the constant function $W_1$ is a
  contractive multiplier of $\cH_\ell\otimes\CC^m$. If $n=2$, the matrix
  $\big[\ell(z_i,z_j)/k(z_i,z_j)\big]_{i,j=1}^2$ is a $2\times2$ principal submatrix of the matrix
  $R$ below and is positive definite, so the Schur product theorem gives
  $\big[(I-W_iW_j^*)\ell(z_i,z_j)\big]_{i,j=1}^2\succeq0$, and, as $\ell$ is a CP kernel, there
  is a contractive multiplier $\Phi$ of $\cH_\ell\otimes\CC^m$ with $\Phi(z_i)=W_i$. If $n=3$,
  then $W(z_i)=W_i$ defines a contractive multiplier $W$ of $\cH_k\otimes\CC^m$, hence of
  $\cH_\ell\otimes\CC^m$.

  However, $\ell/k$ is not positive semidefinite: the matrix
  \[
    R=\Big[\frac{\ell(x_i,x_j)}{k(x_i,x_j)}\Big]_{i,j=1}^3=\begin{pmatrix}\tfrac{66}{265}&\tfrac{264}{2275}&\tfrac{264}{1225}\\[2pt]\tfrac{264}{2275}&\tfrac{352}{4557}&\tfrac{792}{14875}\\[2pt]\tfrac{264}{1225}&\tfrac{792}{14875}&\tfrac{528}{2597}\end{pmatrix}
  \]
  has positive definite $2\times2$ principal submatrices, but
  \[
    v^*Rv=-\frac{180204574}{6672017625}<0\qquad\text{for }v=(-1,1,1).
  \] By Proposition~\ref{l_m/k}, applied with $\cM=\cH_\ell$, the pair $(k,\ell)$
  is not BLH. Since $\Mult(\cH_k)=\Mult(\cH_\ell)$ consists of all functions on $X$,
  Proposition~\ref{LeechtoBLH} shows that $(k,\ell)$ is not complete pseudo-Leech, and, since
  $\Mult(\cH_k)\subseteq\Mult(\cH_\ell)$ completely contractively, the converse of
  Lemma~\ref{l:cert-cont-include} fails.
  \qed
\end{example}

We do not know whether such pairs exist among regular pairs.

\begin{question}\label{q:pseudo-Pick}
    Let $(k, \ell)$ be a regular pair. If $(k, \ell)$ is complete pseudo-Pick, is $(k,\ell)$ a BLH pair?
\end{question}

Example~\ref{bidiskfamilyex} shows that the complete pseudo-Pick property does not imply the
complete Pick property, even for strongly regular diagonal holomorphic pairs; see also
\cite[Proposition~7.4]{mcctsik}. What more can be said? Can one compare the definitions directly?

\begin{problem}
    Determine the relationship between the complete pseudo-Pick
    and complete Pick properties for pairs.
\end{problem}

\subsection{Truncating the Agler wedge}
A kernel is a complete Pick kernel if and only if it admits solutions to every
row-valued interpolation problem; in other words, the ``row-valued Pick
property'' implies the complete Pick property (see \cite{Pickbook} or
\cite{Gregsimple}). This motivates comparing the Agler wedge with its
row-valued analogue. Let $\mathcal{AG}_{\mathrm{row}}(\ell)$ denote the closure in
$\Her(\Omega)$ of the set of finite sums of hereditary functions of the form
\[
  q(z)\big[1-\Phi(z)\Phi(w)^*\big]\ol{q(w)},
\]
where $q:\Omega\to\CC$ is holomorphic and
$\Phi\in\Mult(\cH_\ell\otimes\CC^m,\cH_\ell)$ is a contractive polynomial
multiplier, $m\ge1$. Padding $\Phi$ with zero rows and $q$ with zero entries
shows that $\mathcal{AG}_{\mathrm{row}}(\ell)\subseteq\mathcal{AG}(\ell)$. By
Theorems~\ref{COEXTISAGLER} and~\ref{BLHISCOEXT}, a regular pair $(k,\ell)$ with
$1/k\in\mathcal{AG}(\ell)$ is a BLH pair.

\begin{question}
 \label{q:not-row}
  Is there a regular pair $(k,\ell)$ such that
  $1/k\in\mathcal{AG}(\ell)\setminus\mathcal{AG}_{\mathrm{row}}(\ell)$?
\end{question}

\section*{Declaration of generative AI and AI-assisted technologies in the manuscript preparation process}
The versions of this article posted on arXiv in July and August 2026 were prepared without the
use of generative AI. During the preparation of the present version, the authors used Claude
Opus~5.5 and Claude Fable~5.1 (Anthropic) and GPT~6.1 (OpenAI) in order to proofread the
manuscript, audit its proofs and develop parts of the revision. These tools found typographical
errors, missing hypotheses and gaps in several proofs, none affecting the main results, and
clarified where analyticity and separability hypotheses are needed. They also suggested the
formulation of strong regularity through the dilation property, and provided the following:
Theorem~\ref{t:radial-dilation}, proved in Appendix~\ref{append-radial}, and
Corollary~\ref{c:radial-strong}; Lemma~\ref{l:hol-decomp}; Proposition~\ref{p:disk-universal};
Examples~\ref{polydiskex}, \ref{symbidiskex}, \ref{eg:periodicl} and~\ref{bidiskfamilyex}; and
the examples in Subsection~\ref{ss:q-dilation}, the first of which answers negatively a question
posed in the arXiv versions. They substantially simplified Example~\ref{eg:recursivel}, which in
the arXiv versions relied on an appendix. Question~\ref{q:not-row} replaces two questions posed
in the arXiv versions; these tools answered one of them negatively and suggested the present
formulation of the other. Example~\ref{eg:pseudo-Pick} was generated by Claude. Wenlin Zhang and
Haobo Ma found the example of \cite{MaZhang} discussed in Subsection~\ref{ss:q-pseudo-Pick} with
the help of trureturing, their AI orchestration system.
After using these tools, the authors reviewed and edited the content as needed and take full
responsibility for the content of the published article.

{\emergencystretch=1em
\printbibliography}

@book {Paulsenbook,
    AUTHOR = {Paulsen, V. I.},
     TITLE = {Completely bounded maps and operator algebras},
    SERIES = {Cambridge Studies in Advanced Mathematics},
    VOLUME = {78},
 PUBLISHER = {Cambridge University Press, Cambridge},
      YEAR = {2002},
     PAGES = {xii+300},
}

@article {mcctrent,
    AUTHOR = {McCullough, S. and Trent, T. T.},
     TITLE = {Invariant subspaces and {N}evanlinna-{P}ick kernels},
   JOURNAL = {J. Funct. Anal.},
  FJOURNAL = {Journal of Functional Analysis},
    VOLUME = {178},
      YEAR = {2000},
    NUMBER = {1},
     PAGES = {226--249},
}

@article {Beurlingfactor,
    AUTHOR = {Clou\^atre, R. and Hartz, M. and Schillo, D.},
     TITLE = {A {B}eurling-{L}ax-{H}almos theorem for spaces with a complete
              {N}evanlinna-{P}ick factor},
   JOURNAL = {Proc. Amer. Math. Soc.},
  FJOURNAL = {Proceedings of the American Mathematical Society},
    VOLUME = {148},
      YEAR = {2020},
    NUMBER = {2},
     PAGES = {731--740},
}

@article {Charadvances,
    AUTHOR = {Bhattacharyya, T. and Jindal, A.},
     TITLE = {Complete {N}evanlinna-{P}ick kernels and the characteristic
              function},
   JOURNAL = {Adv. Math.},
  FJOURNAL = {Advances in Mathematics},
    VOLUME = {426},
      YEAR = {2023},
     PAGES = {Paper No. 109089, 25},
}

@article {mcctsik,
    AUTHOR = {McCullough, S. and Tsikalas, G.},
     TITLE = {The complete {P}ick property for pairs of kernels and
              {S}himorin's factorization},
   JOURNAL = {J. Lond. Math. Soc. (2)},
  FJOURNAL = {Journal of the London Mathematical Society. Second Series},
    VOLUME = {112},
      YEAR = {2025},
    NUMBER = {4},
     PAGES = {Paper No. e70325, 71},
}

@article {hyperigid,
    AUTHOR = {Clou\^atre, R. and Hartz, M.},
     TITLE = {Multiplier algebras of complete {N}evanlinna-{P}ick spaces:
              dilations, boundary representations and hyperrigidity},
   JOURNAL = {J. Funct. Anal.},
  FJOURNAL = {Journal of Functional Analysis},
    VOLUME = {274},
      YEAR = {2018},
    NUMBER = {6},
     PAGES = {1690--1738},
}

@incollection {AgMcClocal,
    AUTHOR = {Agler, J. and McCarthy, J. E.},
     TITLE = {Nevanlinna-{P}ick kernels and localization},
 BOOKTITLE = {Operator theoretical methods ({T}imi\c soara, 1998)},
     PAGES = {1--20},
 PUBLISHER = {Theta Found., Bucharest},
      YEAR = {2000},
}

@article {Polynomialarazyenglis,
    AUTHOR = {Ambrozie, C.-G. and Engli\v s, M. and M\"uller, V.},
     TITLE = {Operator tuples and analytic models over general domains in
              {$\mathbb C^n$}},
   JOURNAL = {J. Operator Theory},
  FJOURNAL = {Journal of Operator Theory},
    VOLUME = {47},
      YEAR = {2002},
    NUMBER = {2},
     PAGES = {287--302},
}

@article {HartzJEMS,
    AUTHOR = {Davidson, K. R. and Hartz, M.},
     TITLE = {Interpolation and duality in algebras of multipliers on the
              ball},
   JOURNAL = {J. Eur. Math. Soc. (JEMS)},
  FJOURNAL = {Journal of the European Mathematical Society (JEMS)},
    VOLUME = {25},
      YEAR = {2023},
    NUMBER = {6},
     PAGES = {2391--2434},
}

@book {mullerbook,
    AUTHOR = {M\"{u}ller, V.},
     TITLE = {Spectral theory of linear operators and spectral systems in
              {B}anach algebras},
    SERIES = {Operator Theory: Advances and Applications},
    VOLUME = {139},
   EDITION = {Second},
 PUBLISHER = {Birkh\"{a}user Verlag, Basel},
      YEAR = {2007},
     PAGES = {x+439},
      ISBN = {978-3-7643-8264-3},
}

@incollection {Curtosurveyref,
    AUTHOR = {Curto, R. E.},
     TITLE = {Applications of several complex variables to multiparameter
              spectral theory},
 BOOKTITLE = {Surveys of some recent results in operator theory, {V}ol.\
              {II}},
    SERIES = {Pitman Res. Notes Math. Ser.},
    VOLUME = {192},
     PAGES = {25--90},
 PUBLISHER = {Longman Sci. Tech., Harlow},
      YEAR = {1988},
      ISBN = {0-582-00518-3},
}

@article {BallBolotnikovBLH,
    AUTHOR = {Ball, J. A. and Bolotnikov, V.},
     TITLE = {Multivariable {B}eurling-{L}ax representations: the
              commutative and free noncommutative settings},
   JOURNAL = {Acta Sci. Math. (Szeged)},
  FJOURNAL = {Acta Universitatis Szegediensis. Acta Scientiarum
              Mathematicarum},
    VOLUME = {88},
      YEAR = {2022},
    NUMBER = {1-2},
     PAGES = {5--52},
}

@article {LuoForwBackw,
    AUTHOR = {Gu, C. and Luo, S.},
     TITLE = {Invariant subspaces of the direct sum of forward and backward
              shifts on vector-valued {H}ardy spaces},
   JOURNAL = {J. Funct. Anal.},
  FJOURNAL = {Journal of Functional Analysis},
    VOLUME = {282},
      YEAR = {2022},
    NUMBER = {9},
     PAGES = {Paper No. 109419, 31},
}

@article {BeurlingModulesSarkar,
    AUTHOR = {Bhattacharjee, M. and Das, B. K. and Debnath, R.
              and Sarkar, J.},
     TITLE = {Beurling quotient modules on the polydisc},
   JOURNAL = {J. Funct. Anal.},
  FJOURNAL = {Journal of Functional Analysis},
    VOLUME = {282},
      YEAR = {2022},
    NUMBER = {1},
     PAGES = {Paper No. 109258, 18},
}

@article {RaulBLH,
    AUTHOR = {Curto, R. E. and Hwang, I. S. and Lee, W. Y.},
     TITLE = {The {B}eurling-{L}ax-{H}almos theorem for infinite
              multiplicity},
   JOURNAL = {J. Funct. Anal.},
  FJOURNAL = {Journal of Functional Analysis},
    VOLUME = {280},
      YEAR = {2021},
    NUMBER = {6},
     PAGES = {Paper No. 108884, 101},
}

@article {LuoRecentInvariant,
    AUTHOR = {Gu, C. and Luo, S. and Ma, P.},
     TITLE = {Beurling type invariant subspaces on {H}ardy and {B}ergman
              spaces of the unit ball or polydisk},
   JOURNAL = {Canad. Math. Bull.},
  FJOURNAL = {Canadian Mathematical Bulletin. Bulletin Canadien de
              Math\'ematiques},
    VOLUME = {68},
      YEAR = {2025},
    NUMBER = {1},
     PAGES = {232--245},
}

@article {Beuroriginal,
    AUTHOR = {Beurling, A.},
     TITLE = {On two problems concerning linear transformations in {H}ilbert
              space},
   JOURNAL = {Acta Math.},
  FJOURNAL = {Acta Mathematica},
    VOLUME = {81},
      YEAR = {1949},
     PAGES = {239--255},
}

@article {Halmosoriginal,
    AUTHOR = {Halmos, P. R.},
     TITLE = {Shifts on {H}ilbert spaces},
   JOURNAL = {J. Reine Angew. Math.},
  FJOURNAL = {Journal f\"ur die Reine und Angewandte Mathematik. [Crelle's
              Journal]},
    VOLUME = {208},
      YEAR = {1961},
     PAGES = {102--112},
}

@article {LaxOriginal,
    AUTHOR = {Lax, P. D.},
     TITLE = {Translation invariant spaces},
   JOURNAL = {Acta Math.},
  FJOURNAL = {Acta Mathematica},
    VOLUME = {101},
      YEAR = {1959},
     PAGES = {163--178},
}

@book {Helson,
    AUTHOR = {Helson, H.},
     TITLE = {Lectures on invariant subspaces},
 PUBLISHER = {Academic Press, New York-London},
      YEAR = {1964},
     PAGES = {xi+130},
}

@book {Phillipsscattering,
    AUTHOR = {Lax, P. D. and Phillips, R. S.},
     TITLE = {Scattering theory},
    SERIES = {Pure and Applied Mathematics},
    VOLUME = {26},
   EDITION = {Second},
      NOTE = {With appendices by Cathleen S. Morawetz and Georg Schmidt},
 PUBLISHER = {Academic Press, Inc., Boston, MA},
      YEAR = {1989},
     PAGES = {xii+309},
      ISBN = {0-12-440051-5},
}

@book {Nagybook,
    AUTHOR = {Sz.-Nagy, B. and Foias, C. and Bercovici, H. and
              K\'{e}rchy, L.},
     TITLE = {Harmonic analysis of operators on {H}ilbert space},
    SERIES = {Universitext},
   EDITION = {Second},
 PUBLISHER = {Springer, New York},
      YEAR = {2010},
     PAGES = {xiv+474},
}

@article {BeurBergman,
    AUTHOR = {Aleman, A. and Richter, S. and Sundberg, C.},
     TITLE = {Beurling's theorem for the {B}ergman space},
   JOURNAL = {Acta Math.},
  FJOURNAL = {Acta Mathematica},
    VOLUME = {177},
      YEAR = {1996},
    NUMBER = {2},
     PAGES = {275--310},
}

@article {ShimBer1,
    AUTHOR = {Shimorin, S.},
     TITLE = {Wold-type decompositions and wandering subspaces for operators
              close to isometries},
   JOURNAL = {J. Reine Angew. Math.},
  FJOURNAL = {Journal f\"ur die Reine und Angewandte Mathematik. [Crelle's
              Journal]},
    VOLUME = {531},
      YEAR = {2001},
     PAGES = {147--189},
}

@article {ShimBer2,
    AUTHOR = {Shimorin, S.},
     TITLE = {On {B}eurling-type theorems in weighted {$l^2$} and {B}ergman
              spaces},
   JOURNAL = {Proc. Amer. Math. Soc.},
  FJOURNAL = {Proceedings of the American Mathematical Society},
    VOLUME = {131},
      YEAR = {2003},
    NUMBER = {6},
     PAGES = {1777--1787},
}

@article {ScottBergman,
    AUTHOR = {McCullough, S. and Richter, S.},
     TITLE = {Bergman-type reproducing kernels, contractive divisors, and
              dilations},
   JOURNAL = {J. Funct. Anal.},
  FJOURNAL = {Journal of Functional Analysis},
    VOLUME = {190},
      YEAR = {2002},
    NUMBER = {2},
     PAGES = {447--480},
}

@book {Nikeasy1,
    AUTHOR = {Nikolski, N. K.},
     TITLE = {Operators, functions, and systems: an easy reading. {V}ol. 1},
    SERIES = {Mathematical Surveys and Monographs},
    VOLUME = {92},
      NOTE = {Hardy, Hankel, and Toeplitz,
              Translated from the French by Andreas Hartmann},
 PUBLISHER = {American Mathematical Society, Providence, RI},
      YEAR = {2002},
     PAGES = {xiv+461},
      ISBN = {0-8218-1083-9},
}

@article {GreeneRichtSund,
    AUTHOR = {Greene, D. C. V. and Richter, S. and Sundberg, C.},
     TITLE = {The structure of inner multipliers on spaces with complete
              {N}evanlinna-{P}ick kernels},
   JOURNAL = {J. Funct. Anal.},
  FJOURNAL = {Journal of Functional Analysis},
    VOLUME = {194},
      YEAR = {2002},
    NUMBER = {2},
     PAGES = {311--331},
}

@book {Pickbook,
    AUTHOR = {Agler, J. and McCarthy, J. E.},
     TITLE = {Pick interpolation and {H}ilbert function spaces},
    SERIES = {Graduate Studies in Mathematics},
    VOLUME = {44},
 PUBLISHER = {American Mathematical Society, Providence, RI},
      YEAR = {2002},
     PAGES = {xx+308},
      ISBN = {0-8218-2898-3},
}

@article {ArvesonIII,
    AUTHOR = {Arveson, W.},
     TITLE = {Subalgebras of {$C^*$}-algebras. {III}. {M}ultivariable
              operator theory},
   JOURNAL = {Acta Math.},
  FJOURNAL = {Acta Mathematica},
    VOLUME = {181},
      YEAR = {1998},
    NUMBER = {2},
     PAGES = {159--228},
}

@article {ShimorinDirichletCP,
    AUTHOR = {Shimorin, S.},
     TITLE = {Complete {N}evanlinna-{P}ick property of {D}irichlet-type
              spaces},
   JOURNAL = {J. Funct. Anal.},
  FJOURNAL = {Journal of Functional Analysis},
    VOLUME = {191},
      YEAR = {2002},
    NUMBER = {2},
     PAGES = {276--296},
}

@incollection {BesovPick,
    AUTHOR = {Aleman, A. and Hartz, M. and McCarthy, J. E. and
              Richter, S.},
     TITLE = {Radially weighted {B}esov spaces and the {P}ick property},
 BOOKTITLE = {Analysis of operators on function spaces},
    SERIES = {Trends Math.},
     PAGES = {29--61},
 PUBLISHER = {Birkh\"auser/Springer, Cham},
      YEAR = {2019},
}

@incollection {Hartzsurvey,
    AUTHOR = {Hartz, M.},
     TITLE = {An invitation to the {D}rury-{A}rveson space},
 BOOKTITLE = {Lectures on analytic function spaces and their applications},
    SERIES = {Fields Inst. Monogr.},
    VOLUME = {39},
     PAGES = {347--413},
 PUBLISHER = {Springer, Cham},
      YEAR = {2023},
}

@incollection{shalit2014operator,
    AUTHOR = {Shalit, O. M.},
     TITLE = {Operator theory and function theory in {D}rury--{A}rveson
              space and its quotients},
 BOOKTITLE = {Operator theory},
    EDITOR = {Alpay, D.},
     PAGES = {1125--1180},
 PUBLISHER = {Springer, Basel},
      YEAR = {2015},
}

@article {AHMRfact,
    AUTHOR = {Aleman, A. and Hartz, M. and McCarthy, J. E. and
              Richter, S.},
     TITLE = {Factorizations induced by complete {N}evanlinna-{P}ick
              factors},
   JOURNAL = {Adv. Math.},
  FJOURNAL = {Advances in Mathematics},
    VOLUME = {335},
      YEAR = {2018},
     PAGES = {372--404},
}

@article {AHMRinterpol,
    AUTHOR = {Aleman, A. and Hartz, M. and McCarthy, J. E. and
              Richter, S.},
     TITLE = {Interpolating sequences in spaces with the complete {P}ick
              property},
   JOURNAL = {Int. Math. Res. Not. IMRN},
  FJOURNAL = {International Mathematics Research Notices. IMRN},
    VOLUME = {2019},
      YEAR = {2019},
    NUMBER = {12},
     PAGES = {3832--3854},
}

@article {Interpoltsik,
    AUTHOR = {Tsikalas, G.},
     TITLE = {Interpolating sequences for pairs of spaces},
   JOURNAL = {J. Funct. Anal.},
  FJOURNAL = {Journal of Functional Analysis},
    VOLUME = {285},
      YEAR = {2023},
    NUMBER = {7},
     PAGES = {Paper No. 110059, 43},
}

@article {Shimorin,
    AUTHOR = {Shimorin, S.},
     TITLE = {Commutant lifting and factorization of reproducing kernels},
   JOURNAL = {J. Funct. Anal.},
  FJOURNAL = {Journal of Functional Analysis},
    VOLUME = {224},
      YEAR = {2005},
    NUMBER = {1},
     PAGES = {134--159},
}

@article {Timotinsarkar,
    AUTHOR = {Deepak, K. D. and Pradhan, D. K. and Sarkar, J. and
              Timotin, D.},
     TITLE = {Commutant lifting and {N}evanlinna-{P}ick interpolation in
              several variables},
   JOURNAL = {Integral Equations Operator Theory},
  FJOURNAL = {Integral Equations and Operator Theory},
    VOLUME = {92},
      YEAR = {2020},
    NUMBER = {3},
     PAGES = {Paper No. 27, 15},
}

@article {ColRow,
    AUTHOR = {Hartz, M.},
     TITLE = {Every complete {P}ick space satisfies the column-row property},
   JOURNAL = {Acta Math.},
  FJOURNAL = {Acta Mathematica},
    VOLUME = {231},
      YEAR = {2023},
    NUMBER = {2},
     PAGES = {345--386},
}

@article {freeouter,
    AUTHOR = {Aleman, A. and Hartz, M. and McCarthy, J. E. and
              Richter, S.},
     TITLE = {Free outer functions in complete {P}ick spaces},
   JOURNAL = {Trans. Amer. Math. Soc.},
  FJOURNAL = {Transactions of the American Mathematical Society},
    VOLUME = {376},
      YEAR = {2023},
    NUMBER = {3},
     PAGES = {1929--1978},
}

@article {BB1,
    AUTHOR = {Ball, J. A. and Bolotnikov, V.},
     TITLE = {A {B}eurling type theorem in weighted {B}ergman spaces},
   JOURNAL = {C. R. Math. Acad. Sci. Paris},
  FJOURNAL = {Comptes Rendus Math\'ematique. Acad\'emie des Sciences. Paris},
    VOLUME = {351},
      YEAR = {2013},
    NUMBER = {11-12},
     PAGES = {433--436},
}

@article {BB2,
    AUTHOR = {Ball, J. A. and Bolotnikov, V.},
     TITLE = {Weighted {B}ergman spaces: shift-invariant subspaces and
              input/state/output linear systems},
   JOURNAL = {Integral Equations Operator Theory},
  FJOURNAL = {Integral Equations and Operator Theory},
    VOLUME = {76},
      YEAR = {2013},
    NUMBER = {3},
     PAGES = {301--356},
}

@article {SarkarInvI,
    AUTHOR = {Sarkar, J.},
     TITLE = {An invariant subspace theorem and invariant subspaces of
              analytic reproducing kernel {H}ilbert spaces. {I}},
   JOURNAL = {J. Operator Theory},
  FJOURNAL = {Journal of Operator Theory},
    VOLUME = {73},
      YEAR = {2015},
    NUMBER = {2},
     PAGES = {433--441},
}

@article {SarkarInvII,
    AUTHOR = {Sarkar, J.},
     TITLE = {An invariant subspace theorem and invariant subspaces of
              analytic reproducing kernel {H}ilbert spaces---{II}},
   JOURNAL = {Complex Anal. Oper. Theory},
  FJOURNAL = {Complex Analysis and Operator Theory},
    VOLUME = {10},
      YEAR = {2016},
    NUMBER = {4},
     PAGES = {769--782},
}

@incollection{Shields,
    AUTHOR = {Shields, A. L.},
     TITLE = {Weighted shift operators and analytic function theory},
 BOOKTITLE = {Topics in operator theory},
    EDITOR = {Pearcy, C.},
    SERIES = {Math. Surveys},
    VOLUME = {13},
     PAGES = {49--128},
 PUBLISHER = {Amer. Math. Soc., Providence, RI},
      YEAR = {1974},
}

@article {Quiggin,
    AUTHOR = {Quiggin, P.},
     TITLE = {For which reproducing kernel {H}ilbert spaces is {P}ick's
              theorem true?},
   JOURNAL = {Integral Equations Operator Theory},
  FJOURNAL = {Integral Equations and Operator Theory},
    VOLUME = {16},
      YEAR = {1993},
    NUMBER = {2},
     PAGES = {244--266},
}

@incollection {McCulloughlocal,
    AUTHOR = {McCullough, S.},
     TITLE = {The local de {B}ranges-{R}ovnyak construction and complete
              {N}evanlinna-{P}ick kernels},
 BOOKTITLE = {Algebraic methods in operator theory},
     PAGES = {15--24},
 PUBLISHER = {Birkh\"{a}user Boston, Boston, MA},
      YEAR = {1994},
}

@article{McCulloughcarath,
    AUTHOR = {McCullough, S.},
     TITLE = {Carath\'{e}odory interpolation kernels},
   JOURNAL = {Integral Equations Operator Theory},
    VOLUME = {15},
    NUMBER = {1},
      YEAR = {1992},
     PAGES = {43--71},
       DOI = {10.1007/BF01193766},
}

@book {PaulsenRagh,
    AUTHOR = {Paulsen, V. I. and Raghupathi, M.},
     TITLE = {An introduction to the theory of reproducing kernel {H}ilbert
              spaces},
    SERIES = {Cambridge Studies in Advanced Mathematics},
    VOLUME = {152},
 PUBLISHER = {Cambridge University Press, Cambridge},
      YEAR = {2016},
     PAGES = {x+182},
      ISBN = {978-1-107-10409-9},
}

@misc{Boas,
    AUTHOR = {Boas, H. P.},
     TITLE = {Lecture notes on several complex variables},
      NOTE = {Course notes, Math 650, Texas A\&M University},
      YEAR = {2013},
       URL = {https://haroldpboas.gitlab.io/courses/650-2013c/notes.pdf},
}

@article {Chalmsimply,
    AUTHOR = {Chalmoukis, N. and Dayan, A. and Hartz, M.},
     TITLE = {Simply interpolating sequences in complete {P}ick spaces},
   JOURNAL = {Trans. Amer. Math. Soc.},
  FJOURNAL = {Transactions of the American Mathematical Society},
    VOLUME = {377},
      YEAR = {2024},
    NUMBER = {5},
     PAGES = {3261--3286},
}

@article {weakPickprod,
    AUTHOR = {Jury, M. T. and Martin, R. T. W.},
     TITLE = {Factorization in weak products of complete {P}ick spaces},
   JOURNAL = {Bull. Lond. Math. Soc.},
  FJOURNAL = {Bulletin of the London Mathematical Society},
    VOLUME = {51},
      YEAR = {2019},
    NUMBER = {2},
     PAGES = {223--229},
}

@article {Wickcorona,
    AUTHOR = {Costea, \c{S}. and Sawyer, E. T. and Wick, B. D.},
     TITLE = {The corona theorem for the {D}rury-{A}rveson {H}ardy space and
              other holomorphic {B}esov-{S}obolev spaces on the unit ball in
              {$\mathbb{C}^n$}},
   JOURNAL = {Anal. PDE},
  FJOURNAL = {Analysis \& PDE},
    VOLUME = {4},
      YEAR = {2011},
    NUMBER = {4},
     PAGES = {499--550},
}

@article{DBrtoPick,
    AUTHOR = {Ahmed, H. and Das, B. K. and Panja, S.},
     TITLE = {de {B}ranges--{R}ovnyak spaces which are complete
              {N}evanlinna--{P}ick spaces},
   JOURNAL = {J. Geom. Anal.},
    VOLUME = {35},
    NUMBER = {2},
      YEAR = {2025},
       EID = {61},
}

@article {Gheondea,
    AUTHOR = {Constantinescu, T. and Gheondea, A.},
     TITLE = {On {L}. {S}chwartz's boundedness condition for kernels},
   JOURNAL = {Positivity},
  FJOURNAL = {Positivity. An International Mathematics Journal Devoted to
              Theory and Applications of Positivity},
    VOLUME = {10},
      YEAR = {2006},
    NUMBER = {1},
     PAGES = {65--86},
}

@article {CBkernels,
    AUTHOR = {Bhattacharyya, T. and Dritschel, M. A. and Todd,
              C. S.},
     TITLE = {Completely bounded kernels},
   JOURNAL = {Acta Sci. Math. (Szeged)},
  FJOURNAL = {Acta Universitatis Szegediensis. Acta Scientiarum
              Mathematicarum},
    VOLUME = {79},
      YEAR = {2013},
    NUMBER = {1-2},
     PAGES = {191--217},
}

@book {RKPontryagin,
    AUTHOR = {Alpay, D. and Dijksma, A. and Rovnyak, J. and de Snoo,
              H.},
     TITLE = {Schur functions, operator colligations, and reproducing kernel
              {P}ontryagin spaces},
    SERIES = {Operator Theory: Advances and Applications},
    VOLUME = {96},
 PUBLISHER = {Birkh\"auser Verlag, Basel},
      YEAR = {1997},
     PAGES = {xii+229},
}

@incollection {gheondea2013survey,
    AUTHOR = {Gheondea, A.},
     TITLE = {Reproducing kernel {K}rein spaces},
 BOOKTITLE = {Operator theory},
    EDITOR = {Alpay, D.},
     PAGES = {311--343},
 PUBLISHER = {Springer, Basel},
      YEAR = {2015},
       DOI = {10.1007/978-3-0348-0667-1_40},
}

@incollection{Hellinger,
    AUTHOR = {Agler, J.},
     TITLE = {On the representation of certain holomorphic functions
              defined on a polydisc},
 BOOKTITLE = {Topics in operator theory: {E}rnst {D}. {H}ellinger memorial
              volume},
    SERIES = {Oper. Theory Adv. Appl.},
    VOLUME = {48},
     PAGES = {47--66},
 PUBLISHER = {Birkh\"auser, Basel},
      YEAR = {1990},
}

@article {BickelKnese,
    AUTHOR = {Bickel, K. and Knese, G.},
     TITLE = {Canonical {A}gler decompositions and transfer function
              realizations},
   JOURNAL = {Trans. Amer. Math. Soc.},
  FJOURNAL = {Transactions of the American Mathematical Society},
    VOLUME = {368},
      YEAR = {2016},
    NUMBER = {9},
     PAGES = {6293--6324},
}

@article {Leechoriginal,
    AUTHOR = {Leech, R. B.},
     TITLE = {Factorization of analytic functions and operator inequalities},
   JOURNAL = {Integral Equations Operator Theory},
  FJOURNAL = {Integral Equations and Operator Theory},
    VOLUME = {78},
      YEAR = {2014},
    NUMBER = {1},
     PAGES = {71--73},
}

@book {RandR,
    AUTHOR = {Rosenblum, M. and Rovnyak, J.},
     TITLE = {Hardy classes and operator theory},
      NOTE = {Corrected reprint of the 1985 original},
 PUBLISHER = {Dover Publications, Inc., Mineola, NY},
      YEAR = {1997},
     PAGES = {xiv+161},
      ISBN = {0-486-69536-0},
}

@article {HartzIsomorphism,
    AUTHOR = {Hartz, M.},
     TITLE = {On the isomorphism problem for multiplier algebras of
              {N}evanlinna-{P}ick spaces},
   JOURNAL = {Canad. J. Math.},
  FJOURNAL = {Canadian Journal of Mathematics. Journal Canadien de
              Math\'ematiques},
    VOLUME = {69},
      YEAR = {2017},
    NUMBER = {1},
     PAGES = {54--106},
}

@article {Gregsimple,
    AUTHOR = {Knese, G.},
     TITLE = {A simple proof of necessity in the {M}c{C}ullough-{Q}uiggin
              theorem},
   JOURNAL = {Proc. Amer. Math. Soc.},
  FJOURNAL = {Proceedings of the American Mathematical Society},
    VOLUME = {148},
      YEAR = {2020},
    NUMBER = {8},
     PAGES = {3453--3456},
}

@book {StoutPolyCvx,
    AUTHOR = {Stout, E. L.},
     TITLE = {Polynomial convexity},
    SERIES = {Progress in Mathematics},
    VOLUME = {261},
 PUBLISHER = {Birkh\"auser Boston, Inc., Boston, MA},
      YEAR = {2007},
     PAGES = {xii+439},
      ISBN = {978-0-8176-4537-3},
   MRCLASS = {32E20},
  MRNUMBER = {2305474},
MRREVIEWER = {Marshall\ A.\ Whittlesey},
       DOI = {10.1007/978-0-8176-4538-0},
       URL = {https://doi.org/10.1007/978-0-8176-4538-0},
}

@article{AYcommutant,
    AUTHOR = {Agler, J. and Young, N. J.},
     TITLE = {A commutant lifting theorem for a domain in {$\mathbb{C}^2$}
              and spectral interpolation},
   JOURNAL = {J. Funct. Anal.},
    VOLUME = {161},
      YEAR = {1999},
     PAGES = {452--477},
       DOI = {10.1006/jfan.1998.3362},
}

@article{AYhyperbolic,
    AUTHOR = {Agler, J. and Young, N. J.},
     TITLE = {The hyperbolic geometry of the symmetrized bidisc},
   JOURNAL = {J. Geom. Anal.},
    VOLUME = {14},
      YEAR = {2004},
     PAGES = {375--403},
}

@article{AYrealization,
    AUTHOR = {Agler, J. and Young, N. J.},
     TITLE = {Realization of functions on the symmetrized bidisc},
   JOURNAL = {J. Math. Anal. Appl.},
    VOLUME = {453},
    NUMBER = {1},
      YEAR = {2017},
     PAGES = {227--240},
       DOI = {10.1016/j.jmaa.2017.04.003},
}

@article{BDStoeplitz,
    AUTHOR = {Bhattacharyya, T. and Das, B. K. and Sau, H.},
     TITLE = {Toeplitz operators on the symmetrized bidisc},
   JOURNAL = {Int. Math. Res. Not. IMRN},
      YEAR = {2021},
    NUMBER = {11},
     PAGES = {8763--8805},
}

@incollection{Holbrook,
    AUTHOR = {Holbrook, J. A.},
     TITLE = {Schur norms and the multivariate von {N}eumann inequality},
 BOOKTITLE = {Recent advances in operator theory and related topics
              ({S}zeged, 1999)},
    SERIES = {Oper. Theory Adv. Appl.},
    VOLUME = {127},
     PAGES = {375--386},
 PUBLISHER = {Birkh\"auser, Basel},
      YEAR = {2001},
}

@article{Varopoulos,
    AUTHOR = {Varopoulos, N. Th.},
     TITLE = {On an inequality of von {N}eumann and an application of the
              metric theory of tensor products to operators theory},
   JOURNAL = {J. Funct. Anal.},
    VOLUME = {16},
      YEAR = {1974},
     PAGES = {83--100},
}

@article{ArvesonExtensions,
    AUTHOR = {Arveson, W.},
     TITLE = {Notes on extensions of {$C^*$}-algebras},
   JOURNAL = {Duke Math. J.},
    VOLUME = {44},
    NUMBER = {2},
      YEAR = {1977},
     PAGES = {329--355},
}

@article{Voiculescu,
    AUTHOR = {Voiculescu, D.},
     TITLE = {A non-commutative {W}eyl--von {N}eumann theorem},
   JOURNAL = {Rev. Roumaine Math. Pures Appl.},
    VOLUME = {21},
    NUMBER = {1},
      YEAR = {1976},
     PAGES = {97--113},
}

@article{Ando,
    AUTHOR = {And\^o, T.},
     TITLE = {On a pair of commutative contractions},
   JOURNAL = {Acta Sci. Math. (Szeged)},
    VOLUME = {24},
      YEAR = {1963},
    NUMBER = {1-2},
     PAGES = {88--90},
       URL = {https://acta.bibl.u-szeged.hu/13987/},
}

@misc{MaZhang,
    AUTHOR = {Ma, H. and Zhang, W.},
     TITLE = {A three-point negative answer to {L}uo--{M}c{C}ullough--{T}sikalas's
              pseudo-{P}ick quotient question},
      NOTE = {Manuscript. Lean formalization: file \path{CompletePseudoPickQuotientRefutation.lean} in \url{https://github.com/the-omega-institute/trureturing}},
      YEAR = {2026},
}
\appendix 
\section{Proof of the WOT Montel lemma} \label{appendMontel}

\begin{proof}[Proof of Lemma \ref{noholmontel}]
The closed unit ball of
$\mathcal{B}(\mathcal{H}_k\otimes\mathcal{E},\mathcal{H}_k\otimes\mathcal{F})$
is WOT-compact. Thus $(M_{\Phi_\alpha})$ has a WOT-convergent subnet
$(M_{\Phi_\beta})$ with limit $T$, where $\|T\|\le1$.
WOT convergence is preserved under taking adjoints. Fix $x\in X$ with
$k(x,x)>0$ and $f\in\mathcal{F}$. Then
\[
  M_{\Phi_\beta}^*(k_x\otimes f)=k_x\otimes\Phi_\beta(x)^*f
\]
converges weakly to $T^*(k_x\otimes f)$. Since $k_x\otimes\mathcal{E}$ is a
weakly closed subspace and $k_x\ne0$, there is a uniquely determined vector
$\Phi(x)^*f\in\mathcal{E}$ such that
\begin{equation}\label{e:WOT-montel:1}
  T^*(k_x\otimes f)=k_x\otimes\Phi(x)^*f.
\end{equation}
Contractivity of $T^*$ implies that $\Phi(x)^*$ is a contractive linear operator
from $\mathcal{F}$ to $\mathcal{E}$. For $e\in\mathcal{E}$ and
$f\in\mathcal{F}$,
\[
  \langle\Phi_\beta(x)e,f\rangle
  =\frac{\langle k_x\otimes e,M_{\Phi_\beta}^*(k_x\otimes f)\rangle}
          {\|k_x\|^2}
  \longrightarrow\langle\Phi(x)e,f\rangle.
\]
Thus $\Phi_\beta(x)\to\Phi(x)$ WOT whenever $k(x,x)>0$.

Define $\Phi(x)=0$ when $k(x,x)=0$. Then \eqref{e:WOT-montel:1} holds for all
$x\in X$. For $g\in\mathcal{H}_k\otimes\mathcal{E}$ and $f\in\mathcal{F}$,
\begin{align*}
  \langle(Tg)(x),f\rangle
  &=\langle Tg,k_x\otimes f\rangle
   =\langle g,k_x\otimes\Phi(x)^*f\rangle\\
  &=\langle\Phi(x)g(x),f\rangle.
\end{align*}
Hence $T=M_\Phi$, so $\Phi$ is a contractive multiplier and the proof is
complete.
\end{proof}

\section{From Agler decompositions to realizations} \label{appendLurking}

\begin{proof}[Proof of Theorem~\ref{realization.}] 
The proof is a variation on the lurking isometry argument.
It is convenient to make a slight change in notation. Let $T(w)=Q(w)^*$ and 
view $T(w)$ as an element of $\mathcal{G}$. 
By assumption, there exists a Hilbert space $\mathcal{N}$ and a function 
$R:F\to\mathcal{B}(\mathcal{N}, \mathcal{J})$ such that, for $z, w\in F,$
\[
  V(z)V(w)^*-Y(z)Y(w)^*= \, T(z)^* \big[I-\Phi(z)\Phi(w)^*\big]T(w) \, R(z)R(w)^*.
\]
Rearranging yields, for $z,w\in F,$ 
\begin{equation*} 
\begin{split}  
  V(z)V(w)^*+ & R(z)R(w)^*T(z)^*\Phi(z)\Phi(w)^*T(w)
\\[3pt] & = Y(z)Y(w)^*+ T(z)^* T(w)\, R(z)R(w)^*
\end{split}
\end{equation*}
Define $E(w):\mathcal{N}\to \mathcal{G}\otimes \mathcal{N}$ via
\[
 E(w) v =  T(w)\otimes v
\]
 and note that 
\[
   E(w)R(w)^* \gamma = T(w)\otimes R(w)^* \gamma 
\]
for $\gamma\in \mathcal{J}.$ Hence, for $\gamma,\delta\in \mathcal{J},$
\[
  \langle  R(z) E(z)^* E(w) R(w)^*  \gamma,\delta \rangle
  = \langle  T(w),T(z) \rangle \, \langle R(w)^* \gamma, R(z)^* \delta\rangle.
\]
 Set 
\[
 \mathcal{Z}_1=\textup{span}\Biggl\{\begin{pmatrix}
V(w)^* \\[3pt]
\big(\Phi(w)^*\otimes I_{\mathcal{N}}\big)E(w)R(w)^*
\end{pmatrix}u\ : u\in\mathcal{J},\  w\in F  \Biggl\}\subset \LK \oplus (\mathcal{F}\otimes\mathcal{N}) 
\]
 and 
\[\mathcal{Z}_2=\textup{span}\Biggl\{\begin{pmatrix}
Y(w)^* \\[3pt]
E(w)R(w)^*
\end{pmatrix}u\ : u\in\mathcal{J},\  w\in F  \Biggl\}\subset \mathcal{L}\oplus (\mathcal{G}\otimes\mathcal{N}) .
\]
Since, 
\begin{align*}
 & \Bigg\langle  \begin{pmatrix}
V(w)^* \\[3pt]
\big(\Phi(w)^*\otimes I_{\mathcal{N}}\big)E(w)R(w)^*
\end{pmatrix} u, \, \begin{pmatrix}
V(z)^* \\[3pt]
\big(\Phi(z)^*\otimes I_{\mathcal{N}}\big)E(z)R(z)^*
\end{pmatrix} v \Bigg\rangle \\[3pt]
   = & \langle V(z) V(w)^* u,v\rangle
  + \langle  R(z) E(z)^* (\Phi(z)\Phi(w)^*\otimes I) E(w) R(w)^* u,v\rangle
 \\[3pt]  =  &\langle V(z) V(w)^* u,v\rangle
  + \langle  \Phi(z)\Phi(w)^* T(w) \otimes R(w)^* u,  T(z)\otimes R(z)^*v\rangle
  \\[3pt]  =  &\langle V(z) V(w)^* u,v\rangle +  \langle \Phi(w)^* T(w), \Phi(z)^* T(z)\rangle \,
   \langle R(w)^* u,R(z)^* v\rangle 
  \\[3pt]   = &\langle \left ( V(z)V(w)^* +  T(z)^* \Phi(z)\Phi(w)^* T(w) \, R(z)R(w)^* \right) u,v\rangle
  \\[3pt]   =  &\langle \left (  Y(z)Y(w)^*+ T(z)^* T(w)\, R(z)R(w)^* \right ) u,v\rangle
  \\[3pt]   = &\Bigg\langle \begin{pmatrix}
Y(w)^* \\[3pt]
E(w)R(w)^*
\end{pmatrix}u, \begin{pmatrix}
Y(z)^* \\[3pt]
E(z)R(z)^*
\end{pmatrix}v \Bigg\rangle,
\end{align*}
we obtain  a linear isometry  $U:\mathcal{Z}_1\to\mathcal{Z}_2$ determined by
\[ 
U:  \begin{pmatrix}
V(w)^* \\[3pt]
\big(\Phi(w)^*\otimes I_{\mathcal{N}}\big)E(w)R(w)^*
\end{pmatrix}u\mapsto \begin{pmatrix}
Y(w)^* \\[3pt]
E(w)R(w)^*
\end{pmatrix}u.
\] 
Extend $U$ continuously to the closure of $\mathcal{Z}_1$ and by zero on
its orthogonal complement in $\mathcal{K}\oplus(\mathcal{F}\otimes\mathcal{N})$.
The resulting operator, still denoted by $U$, is a contraction
\[
  U:\mathcal{K}\oplus(\mathcal{F}\otimes\mathcal{N})
    \longrightarrow\mathcal{L}\oplus(\mathcal{G}\otimes\mathcal{N}).
\]
Decompose $U$ as
 \[
\begin{array}{c c}
  & \begin{array}{cc}
    \mathcal{K} & \mathcal{F}\otimes\mathcal{N}
    \end{array} \\
  \begin{array}{c}
    \mathcal{L} \\
   \mathcal{G}\otimes\mathcal{N}
  \end{array}
  &
  \left(
  \begin{array}{cc}
    A \ \ \ \ & B \  \\
    C \ \ \ \  & D \ 
  \end{array}
  \right)
\end{array}
\]
and observe that
\begin{align}
AV(w)^*+B\big(\Phi(w)^*\otimes I_{\mathcal{N}}\big)E(w)R(w)^*&=Y(w)^* \label{eq1} \\ 
CV(w)^*+D\big(\Phi(w)^*\otimes I_{\mathcal{N}}\big)E(w)R(w)^*&=E(w)R(w)^* \label{eq2},
\end{align}
for all $w\in F.$
 Since $\|\Phi(w)\|<1$ for all $w,$ the operator $I_{\mathcal{G}\otimes\mathcal{N}}-D\big(\Phi(w)^*\otimes I_{\mathcal{N}}\big)$  is invertible. Solving  \eqref{eq2} for $E(w)R(w)^*$ obtains
\begin{equation}\label{eq3}
E(w)R(w)^*=\big(I_{\mathcal{G}\otimes\mathcal{N}}-D\big(\Phi(w)^*\otimes I_{\mathcal{N}}\big)\big)^{-1}CV(w)^*.
\end{equation}
Substituting equation~\eqref{eq3} into \eqref{eq1} yields 
\begin{equation}\label{eq4}
AV(w)^*+B\big(\Phi(w)^*\otimes I_{\mathcal{N}}\big)\big(I_{\mathcal{G}\otimes\mathcal{N}}-D\big(\Phi(w)^*\otimes I_{\mathcal{N}}\big)\big)^{-1}CV(w)^*=Y(w)^*,
\end{equation}
for all $w\in F.$ Now, define $\Psi:X\to\mathcal{B}(\mathcal{L},\mathcal{K})$ by 
\[\Psi(w)^*=A+B(\Phi(w)^*\otimes I_{\mathcal{N}})\big(I_{\mathcal{G}\otimes\mathcal{N}}-D\big(\Phi(w)^*\otimes I_{\mathcal{N}}\big)\big)^{-1}C .\]
The identity of equation~\eqref{eq4} immediately yields $V(z)\Psi(z)=Y(z)$, for all $z\in F.$ \par  It remains to show that $\Psi$ is, in fact, a contractive multiplier. Define $O:X\to\mathcal{B}(\mathcal{G}\otimes\mathcal{N}, \mathcal{K})$ by $O(w)^*=\big(I_{\mathcal{G}\otimes\mathcal{N}}-D\big(\Phi(w)^*\otimes I_{\mathcal{N}}\big)\big)^{-1}C$. Equivalently,
\begin{equation}
     \label{eq:O-variant}
     C+D\big(\Phi(w)^*\otimes I_{\mathcal{N}}\big)O(w)^*=O(w)^*,
\end{equation}
for all $w\in X.$   By definition of $\Psi,$ 
\begin{equation}
    \label{e:Psi}
A+B\big(\Phi(w)^*\otimes I_{\mathcal{N}}\big)O(w)^*=\Psi(w)^* 
\end{equation}
for all $w\in X.$ The equalities of equations~\eqref{eq:O-variant} and \eqref{e:Psi} in turn, imply that $U$ maps 
\[\begin{pmatrix}
I_{\mathcal{K}} \\
\big(\Phi(w)^*\otimes I_{\mathcal{N}}\big)O(w)^*
\end{pmatrix}v  \ \  \ \ \textup{ to }  \ \ \ \ 
\begin{pmatrix}
\Psi(w)^* \\
O(w)^*
\end{pmatrix}v,
\]
for all $w\in X$ and $v\in\mathcal{K}$.
Define
\[
  \Gamma(w)=\begin{pmatrix}
    I_{\mathcal{K}}\\
    (\Phi(w)^*\otimes I_{\mathcal{N}})O(w)^*
  \end{pmatrix},\qquad
  \Delta(z,w)=\Gamma(z)^*(I-U^*U)\Gamma(w).
\]
Since $U$ is contractive, $\Delta$ is a positive kernel. The identity
$U\Gamma(w)=\operatorname{col}(\Psi(w)^*,\allowbreak O(w)^*)$ gives
\[
  I_{\mathcal{K}}-\Psi(z)\Psi(w)^*
  =O(z)\big[I_{\mathcal{G}\otimes\mathcal{N}}
       -(\Phi(z)\otimes I_{\mathcal{N}})(\Phi(w)^*\otimes I_{\mathcal{N}})\big]O(w)^*
       +\Delta(z,w).
\]
Multiplying by $\ell(z,w)$ yields
\begin{align}
  &\big[I_{\mathcal{K}}-\Psi(z)\Psi(w)^*\big]\ell(z,w)\notag\\
  &\quad=O(z)\Big(\big[I_{\mathcal{G}\otimes\mathcal{N}}
       -(\Phi(z)\otimes I_{\mathcal{N}})(\Phi(w)^*\otimes I_{\mathcal{N}})\big]
       \ell(z,w)\Big)O(w)^*\notag\\
  &\qquad\quad+\Delta(z,w)\ell(z,w).\label{anothermult}
\end{align}
The first term on the right is a positive kernel because
$\Phi\in\Mult(\mathcal{H}_\ell\otimes\mathcal{F},
              \mathcal{H}_\ell\otimes\mathcal{G})$ is contractive, and the
second is positive by the Schur product theorem. Hence
\eqref{anothermult} is positive semidefinite, so
$\Psi\in\Mult(\mathcal{H}_\ell\otimes\mathcal{L},
              \mathcal{H}_\ell\otimes\mathcal{K})$
is a contractive multiplier, as desired.
\end{proof}

\section{Proof of \texorpdfstring{Theorem~\ref{t:radial-dilation}}{the radial dilation theorem}}
\label{append-radial}

Throughout this appendix, $\ell$ is a circularly symmetric kernel on a domain
$\Omega\subseteq\CC^{\lcal{d}}$ satisfying Assumption~\ref{assume}, $\cH=\cH_\ell$
and $S=M_z^\ell$. Since the polynomials are dense in $\cH$, the space $\cH$ is
separable.

The proof has three steps. Lemma~\ref{l:radial-approx} uses Voiculescu's theorem to
approximate a unital completely positive map $\varphi$ on $\mathcal{O}$, on
finite-dimensional subspaces, by maps of the form $X\mapsto V^*(X\otimes I_\HSF)V$ with $V$
an isometry. Lemma~\ref{l:radial-split} uses a positive lower bound for the Poisson kernel to
write $\varphi R_a$, which sends $p(S)$ to $\varphi(p(aS))$, exactly as the average of a map
of this form and a unital completely positive map. Iterating this splitting,
Proposition~\ref{p:radial-ext} writes $\varphi R_r$ as a countable convex combination of maps
of this form, which is again of this form. The exact dilation comes from the splitting; the
approximation in Lemma~\ref{l:radial-approx} alone would not give it.

\subsection*{The circle action}
Let $\theta\in\mathbb{R}$. By circular symmetry, $e^{i\theta}\Omega=\Omega$ and
$\langle\ell_{e^{i\theta}w},\ell_{e^{i\theta}v}\rangle=\ell(e^{i\theta}v,e^{i\theta}w)
=\ell(v,w)=\langle\ell_w,\ell_v\rangle$ for $v,w\in\Omega$. Hence
$\ell_w\mapsto\ell_{e^{i\theta}w}$ extends to a unitary operator on $\cH$, whose
adjoint $U_\theta$ satisfies
$(U_\theta f)(z)=\langle f,\ell_{e^{i\theta}z}\rangle=f(e^{i\theta}z)$. Consequently
$U_\theta U_\phi=U_{\theta+\phi}$ and $U_\theta S_jU_\theta^*=e^{i\theta}S_j$ for
$1\le j\le\lcal{d}$. For a polynomial $p$, the function $\theta\mapsto U_\theta p$
is continuous, since $p$ is a finite sum of monomials, each in $\cH$; as the
$U_\theta$ are unitary and the polynomials are dense, $\theta\mapsto U_\theta$ is
strongly continuous.

Let $\mathcal{A}=\mathcal{A}(\cH_\ell)$, let $C^*(\mathcal{A})$ be the
$C^*$-algebra it generates, and let $\mathcal{O}$ be the norm closure of
$\mathcal{A}+\mathcal{A}^*$. Set $\gamma_\theta(X)=U_\theta XU_\theta^*$ for
$X\in\cB(\cH)$. For $n\ge1$, let $\mathcal{O}_n$ be the finite-dimensional space
of operators $q(S)$ with $q$ a homogeneous polynomial of degree $n$, let
$\mathcal{O}_{-n}=\{X^*:X\in\mathcal{O}_n\}$, and let $\mathcal{O}_0=\CC I$. For
$X\in\mathcal{O}$, the function $\theta\mapsto\gamma_\theta(X)$ is norm continuous:
this is clear if $X=p(S)+q(S)^*$ for polynomials $p,q$, since
$\gamma_\theta(p(S))=p(e^{i\theta}S)$, and it passes to norm limits because each
$\gamma_\theta$ is isometric. Its Fourier coefficients
\[
  X_n=\int_0^{2\pi}e^{-in\theta}\gamma_\theta(X)\,\frac{d\theta}{2\pi},\qquad n\in\mathbb{Z},
\]
lie in $\mathcal{O}_n$: again this is clear if $X=p(S)+q(S)^*$, the map
$X\mapsto X_n$ is contractive, and $\mathcal{O}_n$ is closed.

For $0<a<1$, let $P_a(\theta)=(1-a^2)/|1-ae^{i\theta}|^2=\sum_{n\in\mathbb{Z}}a^{|n|}e^{in\theta}$
be the Poisson kernel, and define linear maps $E_0$ and $R_a$ on $\cB(\cH)$ by
\[
  \langle E_0(X)x,y\rangle=\int_0^{2\pi}\langle\gamma_\theta(X)x,y\rangle\,\frac{d\theta}{2\pi},
  \qquad
  \langle R_a(X)x,y\rangle=\int_0^{2\pi}P_a(\theta)\langle\gamma_\theta(X)x,y\rangle\,\frac{d\theta}{2\pi}
\]
for $x,y\in\cH$; the integrands are continuous because $\theta\mapsto U_\theta$
is strongly continuous. For $X\in\mathcal{O}$, these are norm-convergent
integrals of continuous $\cB(\cH)$-valued functions. The following properties
hold, where $c_a=(1-a)/(1+a)$.
\begin{enumerate}[(R1)]\itemsep=4pt
  \item Let $L^2=L^2([0,2\pi],d\theta/2\pi)$ and define $V_a:\cH\to\cH\otimes L^2$,
    identifying $\cH\otimes L^2$ with $L^2([0,2\pi],d\theta/2\pi;\cH)$, by
    $(V_ax)(\theta)=P_a(\theta)^{1/2}U_\theta^*x$. Then $V_a$ is an isometry,
    since $\int P_a\,d\theta/2\pi=1$, and $R_a(X)=V_a^*(X\otimes I_{L^2})V_a$.
  \item $R_aR_b=R_{ab}$ for $0<a,b<1$. Indeed, $\gamma_\theta\gamma_\phi=\gamma_{\theta+\phi}$,
    and the convolution of $P_a$ and $P_b$ is $P_{ab}$.
  \item For $X\in\mathcal{O}$, $R_a(X)=\sum_{n\in\mathbb{Z}}a^{|n|}X_n$ with norm
    convergence. In particular $R_a(\mathcal{O})\subseteq\mathcal{O}$ and
    $R_a(p(S))=p(aS)$ for every polynomial $p$.
  \item $E_0$ is completely positive, and $E_0(X)=X_0\in\CC I$ for
    $X\in\mathcal{O}$. Since $P_a\ge c_a$, the map $R_a-c_aE_0$ is completely
    positive.
\end{enumerate}
For a matrix $X\in M_m(\mathcal{O})$ we use the entrywise Fourier coefficients
$X_n\in M_m(\mathcal{O}_n)$, which are the Fourier coefficients of
$\theta\mapsto\gamma_\theta^{(m)}(X)$; thus $R_a^{(m)}(X)=\sum_na^{|n|}X_n$.

For a Hilbert space $H$, let $\mathfrak{D}(H)$ denote the set of maps
$\Theta:\cB(\cH)\to\cB(H)$ of the form $\Theta(X)=V^*(X\otimes I_\HSF)V$, where
$\HSF$ is a Hilbert space and $V:H\to\cH\otimes\HSF$ is an isometry. Such maps are
unital and completely positive. By (R1), $R_a\in\mathfrak{D}(\cH)$, and
$\Theta R_a\in\mathfrak{D}(H)$ whenever $\Theta\in\mathfrak{D}(H)$: if
$\Theta(X)=V^*(X\otimes I_\HSF)V$, then
$\Theta(R_a(X))=\big((V_a\otimes I_\HSF)V\big)^*(X\otimes I_{L^2\otimes\HSF})(V_a\otimes I_\HSF)V$.
Moreover, if $\Theta_j\in\mathfrak{D}(H)$, $\Theta_j(X)=V_j^*(X\otimes I_{\HSF_j})V_j$,
and $c_j\ge0$ with $\sum_jc_j=1$, then $\sum_jc_j\Theta_j\in\mathfrak{D}(H)$, as
one sees by taking $V=\big(\sqrt{c_j}\,V_j\big)_j:H\to\bigoplus_j(\cH\otimes\HSF_j)
=\cH\otimes\big(\bigoplus_j\HSF_j\big)$.

\subsection*{Approximation and splitting}
Below, ucp means unital and completely positive, and $\|\cdot\|_{cb}$ denotes the
completely bounded norm. For a ucp map $\varphi$ on $\mathcal{O}$, we write
$\varphi R_a$ for $\varphi\circ R_a|_{\mathcal{O}}$, which is defined by (R3).

\begin{lemma}\label{l:radial-approx}
  Let $H$ be a separable Hilbert space and $\varphi:\mathcal{O}\to\cB(H)$ a ucp
  map. For every finite-dimensional subspace $E\subseteq\mathcal{O}$ and every
  $\delta>0$, there is $\psi\in\mathfrak{D}(H)$ with
  $\|(\varphi-\psi)|_E\|_{cb}\le\delta$.
\end{lemma}

\begin{proof}
  By Arveson's extension theorem \cite[Chapter~7]{Paulsenbook}, $\varphi$ extends
  to a ucp map on $C^*(\mathcal{A})$, and by Stinespring's theorem
  \cite[Chapter~4]{Paulsenbook} the extension has the form $X\mapsto J^*\pi(X)J$
  for a unital $*$-representation $\pi$ of $C^*(\mathcal{A})$ on a Hilbert space
  $L$ and an isometry $J:H\to L$. Replacing $L$ by the closed span of
  $\pi(C^*(\mathcal{A}))JH$, we may assume that $L$ is separable, as
  $C^*(\mathcal{A})$ is generated by $S_1,\dots,S_{\lcal{d}}$ and $H$ is separable.
  The representation $\sigma(X)=X\otimes I_{\ell^2(\mathbb{N})}$ of $C^*(\mathcal{A})$ on
  $\cH\otimes\ell^2(\mathbb{N})$ is faithful and unital, and $\sigma(C^*(\mathcal{A}))$ contains
  no non-zero compact operator. Since $C^*(\mathcal{A})$, $\cH\otimes\ell^2(\mathbb{N})$ and $L$
  are separable, Voiculescu's theorem \cite{Voiculescu}, in the form
  \cite[Corollary~2, p.~339]{ArvesonExtensions} applied to the identity representation of
  $\sigma(C^*(\mathcal{A}))$ and to $\pi\circ\sigma^{-1}$, shows that
  $\sigma\oplus\pi$ is approximately unitarily equivalent to $\sigma$: there are
  unitary operators $U_j:\cH\otimes\ell^2(\mathbb{N})\to(\cH\otimes\ell^2(\mathbb{N}))\oplus L$ with
  $\|U_j\sigma(X)U_j^*-\sigma(X)\oplus\pi(X)\|\to0$ for every
  $X\in C^*(\mathcal{A})$. Let $\iota:L\to(\cH\otimes\ell^2(\mathbb{N}))\oplus L$ be the
  inclusion and $W_j=U_j^*\iota$, an isometry. Then
  $W_j^*\sigma(X)W_j=\iota^*U_j\sigma(X)U_j^*\iota\to\pi(X)$ in norm, so the maps
  $\psi_j(X)=J^*W_j^*(X\otimes I_{\ell^2(\mathbb{N})})W_jJ$, which belong to
  $\mathfrak{D}(H)$, satisfy $\|\psi_j(X)-\varphi(X)\|\to0$ for every
  $X\in\mathcal{O}$.

  Choose a basis $x_1,\dots,x_N$ of $E$ and linear functionals $f_1,\dots,f_N$ on
  $E$ with $X=\sum_if_i(X)x_i$. For a linear functional $f$ on $E$, a matrix
  $Z\in M_m(E)$ and $\alpha,\beta\in\CC^m$,
  $\langle f^{(m)}(Z)\alpha,\beta\rangle=f\big((\beta^*\otimes I_\cH)Z(\alpha\otimes I_\cH)\big)$,
  so $\|f\|_{cb}=\|f\|$. Hence
  $\|(\varphi-\psi_j)|_E\|_{cb}\le\sum_i\|f_i\|\,\|\varphi(x_i)-\psi_j(x_i)\|\to0$.
\end{proof}

\begin{lemma}\label{l:radial-split}
  Let $H$ be a separable Hilbert space, $\varphi:\mathcal{O}\to\cB(H)$ a ucp map
  and $0<a<1$. There are $\Theta\in\mathfrak{D}(H)$ and a ucp map
  $\varphi_1:\mathcal{O}\to\cB(H)$ such that
  $\varphi R_a=\frac12\Theta+\frac12\varphi_1$ on $\mathcal{O}$.
\end{lemma}

\begin{proof}
  Fix $0<\varepsilon<c_a$. Choose an integer $N$ with
  $4\sum_{n>N}a^n\le\varepsilon/2$, choose $\delta>0$ with
  $2\delta\sum_{n\ge1}a^n\le\varepsilon/2$, and, by
  Lemma~\ref{l:radial-approx}, choose $\psi\in\mathfrak{D}(H)$ with
  $\|(\varphi-\psi)|_{\mathcal{O}_{\le N}}\|_{cb}\le\delta$, where
  $\mathcal{O}_{\le N}=\mathcal{O}_{-N}+\cdots+\mathcal{O}_N$. Set $\Theta=\psi R_a$, which lies
  in $\mathfrak{D}(H)$, and $\Delta=\varphi R_a-\Theta=(\varphi-\psi)R_a$ on
  $\mathcal{O}$. Since $\varphi$, $\psi$ and $R_a$ are unital, $\Delta(I)=0$.

  Let $m\ge1$ and let $Y\in M_m(\mathcal{O})$ be positive with $E_0^{(m)}(Y)=I$.
  The function $\theta\mapsto\gamma_\theta^{(m)}(Y)$ is continuous, with
  positive values and integral $I$, so its Fourier coefficients satisfy
  $\|Y_n\|\le1$: for unit vectors $x,y$,
  \[
    |\langle Y_nx,y\rangle|
    \le\int_0^{2\pi}\langle\gamma_\theta^{(m)}(Y)x,x\rangle^{1/2}
    \langle\gamma_\theta^{(m)}(Y)y,y\rangle^{1/2}\frac{d\theta}{2\pi}\le1 .
  \]
  As $R_a^{(m)}(Y)=\sum_na^{|n|}Y_n$, $Y_0=I$ and $\|(\varphi-\psi)^{(m)}\|\le2$,
  \begin{align*}
    \|\Delta^{(m)}(Y)\|
    &\le\Big\|(\varphi-\psi)^{(m)}\Big(\sum_{0<|n|\le N}a^{|n|}Y_n\Big)\Big\|
     +2\Big\|\sum_{|n|>N}a^{|n|}Y_n\Big\|\\
    &\le2\delta\sum_{n=1}^Na^n+4\sum_{n>N}a^n\le\varepsilon .
  \end{align*}
  Since $\Delta^{(m)}(Y)$ is self-adjoint, $-\varepsilon I\preceq\Delta^{(m)}(Y)\preceq\varepsilon I$.

  Now let $X\in M_m(\mathcal{O})$ be positive and write $E_0^{(m)}(X)=A\otimes I_\cH$
  with $A\in M_m$ positive. For $\eta>0$ put $B=(A+\eta I)^{-1/2}$ and
  $Y=(B\otimes I_\cH)(X+\eta I)(B\otimes I_\cH)$. Each $\Lambda\in\{E_0,\Delta\}$ satisfies $\Lambda^{(m)}((C\otimes I)Z(D\otimes I))=(C\otimes I)\Lambda^{(m)}(Z)(D\otimes I)$
  for $C,D\in M_m$, so $E_0^{(m)}(Y)=I$. As $\Delta^{(m)}(\eta I)=0$, it follows
  that
  $\Delta^{(m)}(X)=\Delta^{(m)}(X+\eta I)=(B^{-1}\otimes I_H)\Delta^{(m)}(Y)(B^{-1}\otimes I_H)$,
  and hence
  \[
    -\varepsilon(A+\eta I)\otimes I_H\preceq\Delta^{(m)}(X)\preceq\varepsilon(A+\eta I)\otimes I_H .
  \]
  Letting $\eta\to0$ gives
  $-\varepsilon A\otimes I_H\preceq\Delta^{(m)}(X)\preceq\varepsilon A\otimes I_H$. On
  the other hand, by (R4), $R_a^{(m)}(X)-c_aA\otimes I_\cH$ is a positive element
  of $M_m(\mathcal{O})$, so $(\varphi R_a)^{(m)}(X)\succeq c_aA\otimes I_H$.
  Therefore $\varphi_1=2\varphi R_a-\Theta=\varphi R_a+\Delta$ satisfies
  $\varphi_1^{(m)}(X)\succeq(c_a-\varepsilon)A\otimes I_H\succeq0$. Thus $\varphi_1$
  is completely positive; it is unital, and
  $\varphi R_a=\frac12\Theta+\frac12\varphi_1$.
\end{proof}

\begin{proposition}\label{p:radial-ext}
  Let $H$ be a separable Hilbert space, $\varphi:\mathcal{O}\to\cB(H)$ a ucp map
  and $0<r<1$. There is $\Phi\in\mathfrak{D}(H)$ with $\Phi=\varphi R_r$ on
  $\mathcal{O}$.
\end{proposition}

\begin{proof}
  Choose $r=s_0<s_1<s_2<\cdots$ with $s_j\to1$, and put $a_j=s_{j-1}/s_j$. Starting
  from $\varphi_0=\varphi$, Lemma~\ref{l:radial-split} gives, for $j\ge1$,
  $\Theta_j\in\mathfrak{D}(H)$ and ucp maps $\varphi_j:\mathcal{O}\to\cB(H)$ with
  $\varphi_{j-1}R_{a_j}=\frac12\Theta_j+\frac12\varphi_j$ on $\mathcal{O}$. Since
  $R_{s_{j-1}}=R_{a_j}R_{s_j}$ by (R2) and $R_{s_j}(\mathcal{O})\subseteq\mathcal{O}$,
  induction gives
  \[
    \varphi R_r=\sum_{j=1}^n2^{-j}\Theta_jR_{s_j}+2^{-n}\varphi_nR_{s_n}
    \qquad\text{on }\mathcal{O}
  \]
  for every $n$. Each $\Theta_jR_{s_j}$ lies in $\mathfrak{D}(H)$, so
  $\Phi=\sum_{j\ge1}2^{-j}\Theta_jR_{s_j}$, a norm-convergent series of maps on
  $\cB(\cH)$, lies in $\mathfrak{D}(H)$. As $\|\varphi_nR_{s_n}\|\le1$, letting
  $n\to\infty$ gives $\Phi=\varphi R_r$ on $\mathcal{O}$.
\end{proof}

\begin{proof}[Proof of Theorem~\ref{t:radial-dilation}]
  First suppose that $H$ is separable. By Remark~\ref{SpInclusion} and
  Theorem~\ref{Arvesonsext}, there are a $*$-representation $\pi$ of
  $C^*(\mathcal{A})$ and an isometry $W$ with $p(T)=W^*\pi(p(S))W$ for all
  polynomials $p$. Thus $\varphi(X)=W^*\pi(X)W$ defines a ucp map
  $\varphi:\mathcal{O}\to\cB(H)$ with $\varphi(p(S))=p(T)$. By
  Proposition~\ref{p:radial-ext}, there is $\Phi\in\mathfrak{D}(H)$, say
  $\Phi(X)=V^*(X\otimes I_\HSF)V$, with $\Phi=\varphi R_t$ on $\mathcal{O}$. By
  (R3), for every polynomial $p$,
  \[
    p(tT)=\varphi(p(tS))=\varphi\big(R_t(p(S))\big)=\Phi(p(S))=V^*\big(p(S)\otimes I_\HSF\big)V .
  \]

  For general $H$, Zorn's lemma gives an orthogonal decomposition
  $H=\bigoplus_\alpha H_\alpha$ into closed subspaces of the form
  $H_\alpha=\ol{C^*(I,T)x_\alpha}$. Each $H_\alpha$ is separable, because
  $C^*(I,T)$ is generated by $T_1,\dots,T_{\lcal{d}}$, and reduces $T$, so
  $T|_{H_\alpha}$ is $\Mult(\cH_\ell)$-contractive. The separable case gives
  isometries $V_\alpha:H_\alpha\to\cH\otimes\HSF_\alpha$, and
  $V=\bigoplus_\alpha V_\alpha:H\to\cH\otimes\big(\bigoplus_\alpha\HSF_\alpha\big)$
  satisfies $p(tT)=V^*\big(p(S)\otimes I\big)V$ for every polynomial $p$.
\end{proof}

\subsection*{Weighted actions}
Suppose instead that $\Omega$ and $\ell$ are invariant under the action
$z\mapsto(\tau^{m_1}z_1,\dots,\tau^{m_{\lcal{d}}}z_{\lcal{d}})$, $\tau\in\T$, where
$m_1,\dots,m_{\lcal{d}}$ are positive integers. The construction at the start of the
subsection on the circle action gives unitary operators $U_\theta$ with
$U_\theta S_jU_\theta^*=e^{im_j\theta}S_j$. For
$n\ge1$, the space $\mathcal{O}_n$ is now spanned by the monomials $S^\alpha$ with
$\sum_jm_j\alpha_j=n$, which is finite-dimensional; $\mathcal{O}_0=\CC I$ because
every $m_j\ge1$; and
$R_a(p(S))=p(a^{m_1}S_1,\dots,a^{m_{\lcal{d}}}S_{\lcal{d}})$. With these changes,
Lemmas~\ref{l:radial-approx} and~\ref{l:radial-split},
Proposition~\ref{p:radial-ext} and the proof of Theorem~\ref{t:radial-dilation}
show that
$p(t^{m_1}T_1,\dots,t^{m_{\lcal{d}}}T_{\lcal{d}})=V^*\big(p(S)\otimes I_\HSF\big)V$
for every polynomial $p$, for some Hilbert space $\HSF$ and isometry $V$.

\newpage

 \printindex
 \end{document}